\documentclass[reqno,11pt]{amsart}
\usepackage{amsmath,amssymb,mathrsfs,amsthm,amsfonts}
\usepackage[normalem]{ulem}
\usepackage[inline]{enumitem}
\usepackage[usenames,dvipsnames]{xcolor}
\usepackage{hyperref}
\usepackage[percent]{overpic}
\usepackage{comment}
\usepackage{stmaryrd}
\usepackage{algorithm}
\usepackage{algorithmic}
\usepackage{pgfplots}
\usepackage{subcaption}
\usepackage{tikz}
\usetikzlibrary{calc}
\usetikzlibrary{arrows.meta,backgrounds}
\usepackage{multirow,array,longtable,booktabs}
\usetikzlibrary{arrows}
\usepackage{bm}
\usepackage{graphicx}

\hypersetup{
	colorlinks=true, linkcolor=blue,
	citecolor=ForestGreen
}
\usepackage[paper=letterpaper,margin=1in]{geometry}
\DeclareMathAlphabet{\mathpzc}{OT1}{pzc}{m}{it}

\newtheorem{theorem}{Theorem}[section]

\newtheorem{lemma}[theorem]{Lemma}
\newtheorem{proposition}[theorem]{Proposition}
\newtheorem{corollary}[theorem]{Corollary}
\newtheorem{definition}[theorem]{Definition}

\newtheorem{assumption}[theorem]{Assumption}
\newtheorem{notation}[theorem]{Notation}
\newtheorem{remark}[theorem]{Remark}
\theoremstyle{remark}
\newtheorem{example}[theorem]{Example}
\numberwithin{equation}{section}
\allowdisplaybreaks
\usepackage{acronym}

\acrodef{KPZ}{Kardar--Parisi--Zhang}
\acrodef{SHE}{Stochastic Heat Equation}
\acrodef{LDP}{Large Deviation Principle}

\newcommand{\be}{\begin{equation}}
\newcommand{\ee}{\end{equation}}

\newcommand{\E}{\mathbb{E}}	

\makeatletter
\newcommand\nd{\noindent}
\newcommand{\1}{{\bf 1}}
\newcommand{\cF}{\mathcal{F}}
\newcommand{\cA}{\mathcal{A}}

\newcommand{\cC}{\mathcal{C}}
\newcommand{\cO}{\mathcal{O}}

\newcommand{\cH}{\mathcal{H}}
\newcommand{\cG}{\mathcal{G}}

\newcommand{\cP}{\mathcal{P}}
\newcommand{\cS}{\mathcal{S}}
\newcommand{\cV}{\mathcal{V}}

\newcommand{\tr}{\text{tr}}

\newcommand{\la}{{\lambda}}
\newcommand{\La}{{\Lambda}}

\def\pa{\partial}
\newcommand{\Law}{\operatorname{Law}}

\newcommand{\R}{\mathbb{R}} 
\newcommand{\N}{\mathbb{N}}

\newcommand{\e}{\varepsilon}

\newcommand{\m}{\mathsf}

\renewcommand{\hat}{\widehat}

\renewcommand{\bar}{\overline}

\usepackage{comment}

\newcounter{hyp}
\title[Analytical Approach to Continuous-time Causal Optimal Transport]{Analytical Approach to Continuous-Time Causal Optimal Transport}

\author[J.\ Backhoff]{Julio Backhoff}
\address{J.\ Backhoff,
	Faculty of Mathematics, University of Vienna,
	\newline\hphantom{\quad \ \ J. Backhoff}
	Oskar-Morgenstern-Platz 1, Vienna, 1190, Austria}
\email{julio.backhoff@univie.ac.at}
\thanks{J.\ Backhoff is partially supported  by the Austrian Science Fund (FWF) DOI 10.55776/P36835.}

\author[E.\ Bayraktar]{Erhan Bayraktar}
\address{E.\ Bayraktar,
	Department of Mathematics, University of Michigan,
	\newline\hphantom{\quad \ \ E. Bayraktar}
	530 Church St, Ann Arbor, MI 48109, USA}
\email{erhan@umich.edu}
\thanks{E. Bayraktar is partially supported by the NSF grants DMS-2507940, and 2406232, and in part by the Susan M. Smith Professorship.}

\author[I.\ Ekren]{Ibrahim Ekren}
\address{I.\ Ekren,
	Department of Mathematics, University of Michigan,
	\newline\hphantom{\quad \ \ I. Ekren}
	530 Church St, Ann Arbor, MI 48109, USA}
\email{iekren@umich.edu}
\thanks{I.\ Ekren is partially supported by the NSF grant DMS-2406240.}

\author[A.\ Zitridis]{Antonios Zitridis}
\address{A.\ Zitridis,
	Department of Mathematics, University of Michigan,
	\newline\hphantom{\quad \ \ A. Zitridis}
	530 Church St, Ann Arbor, MI 48109, USA}
\email{zitridis@umich.edu}
\RequirePackage{xcolor} 
\providecommand{\DIFdeltex}[1]{} 
\providecommand{\DIFaddbegin}{\begingroup\color{blue}} 
\providecommand{\DIFaddend}{\endgroup} 
\providecommand{\DIFdelbegin}{} 
\providecommand{\DIFdelend}{} 
\providecommand{\DIFdelFL}[1]{} 
\providecommand{\DIFaddbeginFL}{\begingroup\color{blue}} 
\providecommand{\DIFaddendFL}{\endgroup} 
\providecommand{\DIFdelbeginFL}{} 
\providecommand{\DIFdelendFL}{} 
\newcommand{\DIFscaledelfig}{0.5}
\RequirePackage{settobox} 
\RequirePackage{letltxmacro} 
\newsavebox{\DIFdelgraphicsbox} 
\newlength{\DIFdelgraphicswidth} 
\newlength{\DIFdelgraphicsheight} 
\LetLtxMacro{\DIFOincludegraphics}{\includegraphics} 
\newcommand{\DIFaddincludegraphics}[2][]{{\color{blue}\fbox{\DIFOincludegraphics[#1]{#2}}}} 
\newcommand{\DIFdelincludegraphics}[2][]{
\sbox{\DIFdelgraphicsbox}{\DIFOincludegraphics[#1]{#2}}
\settoboxwidth{\DIFdelgraphicswidth}{\DIFdelgraphicsbox} 
\settoboxtotalheight{\DIFdelgraphicsheight}{\DIFdelgraphicsbox} 
\scalebox{\DIFscaledelfig}{
\parbox[b]{\DIFdelgraphicswidth}{\usebox{\DIFdelgraphicsbox}\\[-\baselineskip] \rule{\DIFdelgraphicswidth}{0em}}\llap{\resizebox{\DIFdelgraphicswidth}{\DIFdelgraphicsheight}{
\setlength{\unitlength}{\DIFdelgraphicswidth}
\begin{picture}(1,1)
\thicklines\linethickness{2pt} 
{\color[rgb]{1,0,0}\put(0,0){\framebox(1,1){}}}
{\color[rgb]{1,0,0}\put(0,0){\line( 1,1){1}}}
{\color[rgb]{1,0,0}\put(0,1){\line(1,-1){1}}}
\end{picture}
}\hspace*{3pt}}} 
} 
\LetLtxMacro{\DIFOaddbegin}{\DIFaddbegin} 
\LetLtxMacro{\DIFOaddend}{\DIFaddend} 
\LetLtxMacro{\DIFOdelbegin}{\DIFdelbegin} 
\LetLtxMacro{\DIFOdelend}{\DIFdelend} 
\DeclareRobustCommand{\DIFaddbegin}{\DIFOaddbegin \let\includegraphics\DIFaddincludegraphics} 
\DeclareRobustCommand{\DIFaddend}{\DIFOaddend \let\includegraphics\DIFOincludegraphics} 
\DeclareRobustCommand{\DIFdelbegin}{\DIFOdelbegin \let\includegraphics\DIFdelincludegraphics} 
\DeclareRobustCommand{\DIFdelend}{\DIFOaddend \let\includegraphics\DIFOincludegraphics} 
\LetLtxMacro{\DIFOaddbeginFL}{\DIFaddbeginFL} 
\LetLtxMacro{\DIFOaddendFL}{\DIFaddendFL} 
\LetLtxMacro{\DIFOdelbeginFL}{\DIFdelbeginFL} 
\LetLtxMacro{\DIFOdelendFL}{\DIFdelendFL} 
\DeclareRobustCommand{\DIFaddbeginFL}{\DIFOaddbeginFL \let\includegraphics\DIFaddincludegraphics} 
\DeclareRobustCommand{\DIFaddendFL}{\DIFOaddendFL \let\includegraphics\DIFOincludegraphics} 
\DeclareRobustCommand{\DIFdelbeginFL}{\DIFOdelbeginFL \let\includegraphics\DIFdelincludegraphics} 
\DeclareRobustCommand{\DIFdelendFL}{\DIFOaddendFL \let\includegraphics\DIFOincludegraphics} 
\RequirePackage{listings} 
\lstdefinelanguage{DIFcode}{ 
  moredelim=[il][\color{white}\tiny]{\%DIF\ <\ }, 
  moredelim=[il][\sffamily\bfseries]{\%DIF\ >\ } 
} 
\lstdefinestyle{DIFverbatimstyle}{ 
	language=DIFcode, 
	basicstyle=\ttfamily, 
	columns=fullflexible, 
	keepspaces=true 
} 
\lstnewenvironment{DIFverbatim}{\lstset{style=DIFverbatimstyle}}{} 
\lstnewenvironment{DIFverbatim*}{\lstset{style=DIFverbatimstyle,showspaces=true}}{} 

\begin{document}
\maketitle

\vspace{-10mm}
\begin{abstract}
We study causal optimal transport in continuous time, with Markovian
cost, between a finite-state Markov source and a diffusion target. By
replacing the source with its conditional law given the observation of the
target, we characterize the value of this transport problem through a
fully nonlinear parabolic master equation on 
an enlarged state space. 
We further show that this
value coincides with those of two equivalent stochastic control
problems on the simplex: a control of the Kushner--Stratonovich
filtering equation with a zero-mean condition, and a state-constrained stochastic
optimal control problem. Both formulations give rise to implementable
numerical schemes that approximate the value from above and below.
\end{abstract}

\tableofcontents

\section{Introduction}

Optimal transport has become a foundational framework in probability and analysis, with applications ranging from data science and image processing to economics, mathematical finance, and biology. Over the last decade and a half, however, it has become increasingly clear that classical optimal transport is not fully suited to studying stochastic processes. Its formulation does not take filtrations into account and therefore ignores the intrinsic direction of time. Consequently, two stochastic processes may be close in Wasserstein distance while having very different probabilistic structures: for instance, their Doob--Meyer decompositions may differ substantially, and the values of optimal stopping or control problems associated with them may be far apart.

This observation motivated the development of adapted topologies for stochastic processes, initiated in the works of Aldous, Hoover, Keisler, and others (see, e.g., \cite{aldous1981weak} and \cite{hoover1984adapted,Ho87}). These topologies capture the continuity and stability properties of genuinely dynamic objects, including Doob--Meyer decompositions and optimal stopping values. More recently, adapted Wasserstein distances, which arise as particular instances of causal optimal transport, have provided metric counterparts to these adapted topologies. This point of view has led to quantitative stability estimates, statistical consistency results, and a growing range of applications.

At the heart of causal optimal transport is the requirement that couplings respect the temporal information structure of the processes. Given a source process $X$ and a target process $Y$, one optimizes an expected cost over joint laws of $(X,Y)$ with prescribed marginals. In contrast to classical optimal transport, admissible couplings are required to be causal: the target process $Y$ must be generated from the source process $X$ in an adapted way, possibly using additional independent randomization. Equivalently, the conditional law of $(Y_s)_{s\le t}$ given the entire path of $X$ must depend only on the history $(X_s)_{s\le t}$, for every $t\in[0,T]$. The reader may recognize this condition as the H-hypothesis between the filtration of $X$ enlarged by the filtration of $Y$.

Despite the growing interest in causal optimal transport,
comparatively little is known about the characterization and
computation of the associated values and optimizers. Important
progress has been made in discrete time, both on dynamic programming
and duality \cite{veraguas2016causal} and on computation
\cite{EcPa24}. To the best of our knowledge, the present article is
the first to address the characterization and computation of causal
optimal transport in continuous time. Our approach is analytical. We
identify the correct state variables for the problem, show that they
differ from the original state variables $(X,Y)$, and characterize
the value function through a master equation on the resulting
enlarged state space. Because this master equation involves a
singular control problem, we introduce two relaxed formulations that
are more tractable and prove that all three problems have the same
value. This yields partial differential equations that characterize the causal transport value exactly, or through approximations from above and below. These equations provide a basis for numerical schemes, although the focus of the present paper is foundational: we establish a route to computing continuous-time causal transport values and isolate the main analytical difficulties.

We now describe our contribution more precisely. The source process $X$ is a continuous-time Markov chain on a $K$-element state space $\mathbb S$, so without loss of generality
\[
\mathbb{S}:=\{1,\ldots,K\}, \qquad K\in\mathbb N,
\]
with generator $\Lambda$. The target process $Y$ is a diffusion in $\mathbb R^d$ with characteristics $(b_\nu,\sigma_\nu^2)$. The laws of $X$ and $Y$ are fixed and denoted by $\mu$ and $\nu$, respectively. We study the causal optimal transport problem
\[
\sup_{\pi\in\Pi_c(\mu,\nu)}
\E^\pi\left[
\int_0^T f_0(X_s,Y_s)\,ds
+
g_0(X_T,Y_T)
\right],
\]
where the supremum is taken over causal couplings of $\mu$ and $\nu$. Equivalently, under an admissible coupling, the conditional law of $(Y_s)_{s\le t}$ given $X$ is $\sigma((X_s)_{s\le t})$-measurable for every $t\in[0,T]$.

A key asymmetry appears immediately. Under any causal coupling, the semimartingale characteristics of $X$ are preserved in the joint filtration of $(X,Y)$; see Remark~2.5. The same is not true for $Y$. This asymmetry is the main reason why the state space of the problem must be enlarged. The correct additional state variable, as we show, is the filter
\[
p_t^i
:=
\mathbb{P}^\pi\big(X_t=i\,\big|\,\mathcal F_t^Y\big),
\qquad i\in \mathbb{S},
\]
namely the conditional distribution of the source given the observation of the target. The filter takes values in the simplex $\mathcal P(\mathbb{S})$, and the objective can be rewritten as the linear functional
\[
\E^\pi\left[
\int_0^T \sum_{i=1}^K f_0(i,Y_s)p_s^i\,ds
+
\sum_{i=1}^K g_0(i,Y_T)p_T^i
\right].
\]

In one of our first results, we establish that causality implies that the filter satisfies dynamics of the form
\[
dp_t^i
=
(p_t\Lambda)_i\,dt
+
(Z_t^i)^\top \sigma_\nu(t,Y_t)^{-2}
\big(dY_t-b_\nu(t,Y_t)\,dt\big),
\qquad i\in \mathbb{S},
\]
where
\[
\sum_{i=1}^K Z_t^i=0,
\qquad
\mathbf 1_{\{p_t^i=0\}}Z_t^i=0,
\qquad i\in \mathbb{S}.
\]
Here both $p$ and $Z$ are adapted to the filtration of $Y$. For fixed $Z$, the equation for $p$ is linear, but the boundary condition
\[
\mathbf 1_{\{p_t^i=0\}}Z_t^i=0
\]
is nonlinear and is precisely what enforces the state constraint $p_t\in\mathcal P(\mathbb{S})$. If we optimize over all controlled processes $(p,Y,Z)$ satisfying these dynamics and the state constraint, we obtain a state-constrained stochastic control problem. Since this class of dynamics contains those generated by causal couplings, its value gives an upper bound for the original causal optimal transport problem.

Conversely, suppose we restrict to controls for which
\[
h_t^i:=\frac{Z_t^i}{p_t^i}
\quad\text{on }\{p_t^i>0\}
\]
is uniformly bounded. This is stronger than the boundary condition above. In this case the filter dynamics become the Kushner--Stratonovich equation
\[
dp_s^i
=
(p_s\Lambda)_i\,ds
+
p_s^i h(s,Y_\cdot,i)^\top \sigma_\nu(s,Y_s)^{-2}
\big(dY_s-b_\nu(s,Y_s)\,ds\big),
\qquad i\in \mathbb{S},
\]
together with the zero-conditional-mean constraint
\[
\sum_{i=1}^K h(t,Y_\cdot,i)p_t^i
=
0.
\]
In other words, $p_t=(p_t^i)_{i\in\mathbb S}$ is the conditional law of $X_t$ given the observation of $Y$ in the filtering system
\[
dY_t
=
\big(h(t,Y_\cdot,X_t)+b_\nu(t,Y_t)\big)\,dt
+
\sigma_\nu(t,Y_t)\,dW_t,
\]
where $X$ has generator $\Lambda$. Maximizing the same cost over couplings generated by such filtering systems gives a lower bound for the causal optimal transport problem.

Our main theorem shows that these two bounds are sharp:

\begin{theorem} \label{thm:main_intro}
The causal optimal transport problem coincides with both the filtering lower bound and the state-constrained upper bound described above. Moreover, if $Y_0$ is deterministic, then the value function of the problem, written in the enlarged state space $(t,y,p)$, is the unique viscosity solution of the variational inequality
\[
\begin{cases}
H_{\mathrm{vi}}\big(t,y,p,\partial_t V,D_{(p,y)}V,D^2_{(p,y)}V\big)\le 0,
& \text{on } [0,T)\times\mathbb R^d\times\mathcal P(\mathbb S),\\[3pt]
H_{\mathrm{vi}}\big(t,y,p,\partial_t V,D_{(p,y)}V,D^2_{(p,y)}V\big)\ge 0,
& \text{on } [0,T)\times\mathbb R^d\times\mathcal P(\mathbb S)^\circ,\\
\mathcal{B}_0\big(t,y,p,\partial_t V,D_{(p,y)}V,D^2_{(p,y)}V\big)\,\,\ge 0,&\text{on } [0,T)\times\mathbb R^d\times\pa\mathcal{P}(\mathbb S),
\end{cases}
\]
with the terminal condition $V(T,y,p)=\sum_{i=1}^Kg_0(i,y)p^i$ and where $\mathcal{P}(\mathbb{S})$ is understood as a subset of $\R^K$; see \eqref{Hsc}, \eqref{Hvi}, \eqref{correctPDE}, \eqref{boundaryop}. If $Y_0$ is not deterministic, then the value of the causal optimal transport problem is obtained by solving a weak optimal transport problem between the laws of $X_0$ and $Y_0$, with cost given by $V(0,\cdot,\cdot)$; see \eqref{eq:initialdistthm}.
\end{theorem}

This result gives, to the best of our knowledge, the first PDE characterization of causal optimal transport in continuous time. Furthermore, the Hamiltonian $H_{\mathrm{vi}}$ is explicitly given; for ease of presentation, we refer the reader to \eqref{Hsc}, \eqref{Hvi} for $H_{\text{vi}}$, and to \eqref{boundaryop} for the operator $\mathcal{B}_0$. The comparison with bicausal optimal transport is instructive. In the bicausal problem, admissible couplings are required to be causal in both directions, from $X$ to $Y$ and from $Y$ to $X$. This additional symmetry simplifies the analysis: the state space need not be enlarged at all. When both processes are diffusions, a PDE characterization of the bicausal value was obtained by Bion--Nadal and Talay \cite{bion-nadal2019wasserstein}, and the recent work of Cont and Lim \cite{cont-lim2024causal} further clarifies the associated probabilistic structure.

The causal problem is substantially more delicate. The asymmetry between source and target destroys the finite-dimensional state-space description available in the bicausal setting. The remedy is to adjoin the filter $p$, the conditional law of the source given the target observation, as an additional state variable. The resulting value function satisfies a fully nonlinear parabolic PDE on
\[
[0,T]\times\mathbb R^d\times\mathcal P(\mathbb{S}).
\]
Structurally, this equation is related to the master and Bellman equations that arise in the control of partially observed diffusions and Wonham-filter-type problems \cite{gozzi2000hamilton,bandini-cosso-fuhrman-pham2019,bensoussan1992partial,bayraktar2025comparison,bayraktar2026comparison}. The new feature is the pointwise constraint on the controls $
\sum_{i=1}^K z^i=0$.
This constraint is also the zero-conditional-mean condition
$
\sum_{i=1}^K h^i p^i=0,
$
and reflects the fact that the law of $Y$ is prescribed rather than chosen by the controller. In the terminology of \cite{cont-lim2024causal}, the Brownian motion driving $Y$ may acquire a drift under the enlarged filtration generated by $(X,Y)$, but this drift must disappear after conditioning on the filtration of $Y$ alone. The zero-conditional-mean condition enforces precisely this cancellation. This mechanism is also reminiscent of the inconspicuous trading condition in \cite{kyle,cho,back2020optimal,bose2020kyle,bose2021multidimensional}.

This perspective also clarifies a structural difficulty of causal optimal transport. The set of admissible causal couplings is large and difficult to manipulate, because causality is not a pointwise constraint on the joint law. Unlike in mean-field games or control, the full laws of the marginal processes are fixed, not just their marginals at given times. A natural strategy is therefore to identify a tractable subclass of couplings that is still rich enough to preserve the value of the original problem. In the Markovian setting considered here, the relevant subclass consists of the couplings generated by the controlled Kushner--Stratonovich equation, or equivalently by the $\mathbb R^d$-valued control $h$ subject to the zero-conditional-mean constraint. Our main result shows that this subclass already attains the full causal optimal transport value. Thus the causal transport problem can be reduced to a finite-dimensional stochastic control problem on the simplex. We also emphasize that our main result, Theorem \ref{thm:main_intro}, is an initial step toward computational methods for causal optimal transport in continuous time. For instance, the state-constrained upper bound, together with the idea of bounding the controls by $N$ and sending $N\to\infty$, yields a monotone approximation from below in terms of more tractable PDEs; this is the content of Theorem \ref{Nmainpdes}. On the other hand, the state-constrained upper bound together with convex duality can be used to produce a monotone sequence of approximations from above, as in Theorem \ref{thm:dual}.

We expect the same ideas to extend beyond the Markovian setting. For non-Markovian sources, targets, or costs, the simplex-valued filter $p$ should be replaced by its path-space analogue: the conditional law of the source path given the observation of the target path. The corresponding master equation would then live on a space of probability measures over paths, and both the formulation and the analysis would become substantially more involved. For this reason, we develop the Markovian case first, where the geometry of the simplex keeps the main ideas transparent, and leave the path-dependent extension for future work. Similarly, we assume that $X$ is a finite-state continuous-time Markov chain, rather than a more general Markov process. In the general case, the difficulty lies in dealing with a state-constrained PDE in infinite dimensions, for which new techniques, such as an appropriate comparison principle, would first have to be developed.

\subsection{Literature review.}

Classical optimal transport has a vast literature, beginning with the Monge--Kantorovich formulation and its modern analytic development; see, among many others, \cite{b,m,gangbo1996geometry,villani,santambrogio2015optimal}. When the objects to be transported are stochastic processes, however, the classical Wasserstein topology does not encode the temporal information structure of the processes. This limitation motivated several adapted notions of weak convergence and transport. Early topological approaches go back to Aldous' extended weak topology, the adapted distribution framework of Hoover and Keisler, and Hellwig's information topology \cite{aldous1981weak,hoover1984adapted,Ho87,He96}. In stochastic programming and statistics, the same issue appears through the nested distance of Pflug and Pichler \cite{pflug-pichler2012} or the Markov constructions of R{\"u}schendorf \cite{Ru85}. The connection to optimal transport is captured by causal transport plans/couplings, and the associated Monge--Kantorovich framework was introduced by Lassalle \cite{lassalle2018causal}, coining the term causal optimal transport. However, causal transport plans had been considered earlier, for instance implicitly in the proof of the celebrated Yamada--Watanabe criterion, or in the theory of enlargement of filtrations, where they are related to the H-hypothesis (see, e.g., \cite{AkJe17} and references therein). Subsequent work clarified the relation between these approaches and the adapted Wasserstein distance; see, for instance, \cite{backhoff2020fundamental,backhoff2020all,backhoff2020adapted,BaBePa21,bartl2025wasserstein,Pa24}.

Causal optimal transport in discrete time was developed in
\cite{veraguas2016causal}, where dynamic programming and duality results were
obtained and applications to stochastic optimization were studied. On the
computational side, Eckstein and Pammer \cite{EcPa24} developed
linear-programming and entropy-based methods for causal optimal transport as well as its more tractable relative, bicausal transport, in which the causality constraint is symmetrized by demanding causality from source to target and vice versa. For bicausal transport in discrete time,
Bayraktar and Han \cite{BayraktarHanFVI2025} proposed fitted value iteration
methods and Pflug and Pichler \cite{pflug-pichler2012} developed a discretization approach.  By contrast, the present paper derives a
continuous-time PDE characterization for a causal transport problem.
In continuous time, \cite{acciaio2020causal} related causal transport to
enlargement of filtrations and stochastic optimization, while
\cite{bion-nadal2019wasserstein} obtained a PDE characterization for a
Wasserstein-type bicausal distance between diffusion laws. See
\cite{BaKaRo25,HiRo24} for probabilistic counterparts to
\cite{bion-nadal2019wasserstein}. More recently, \cite{cont-lim2024causal}
studied the structure of causal and bicausal couplings on path space, in
particular for laws of stochastic differential equations. The present paper
differs from these works in that it treats a one-sided (i.e., non-symmetrized) causal transport problem
in continuous time and identifies a finite-dimensional Markovian state
description through the conditional law of the source given the target
observation. We refer to
\cite{backhoff2022estimating,LaPaWi25,BlWiZhZh26,BlLaPaWi24,ChEc25,ChVi25} for applications
of causal and bicausal transport in statistics, to
\cite{backhoff2020adapted,BaWi22,JiOb24,Ji24} for applications in mathematical
finance, and to \cite{AcHoPa25,AcKrPa14,AcBaGrHoPa26,xu2020cot,BoLiOb23,JiLi25,GuWo25,LiPa25,SaLeLiHoDaLy21} for applications in machine learning and related subjects. 

The filtering reformulation used here is connected to the classical theory of nonlinear filtering. The Kushner--Stratonovich and Wonham filtering equations originate in \cite{Kushner1964,wonham1965applications}; see also \cite{pardoux1982equations,liptser1977statistics,bensoussan1992partial}. Control problems under partial observation and their separated formulations lead naturally to equations on spaces of probability measures; related infinite-dimensional viscosity and Bellman equations appear in \cite{lions1989viscosity2,lions1989viscosity3,gozzi2000hamilton,bandini-cosso-fuhrman-pham2019,martini2023kolmogorov}. Our formulation has a different origin: the filter is not introduced to estimate an unobserved signal for a controller, but rather to encode the information constraint imposed by causality while preserving the prescribed law of the target process.

The PDE and control aspects of the paper are also related to viscosity solutions for constrained stochastic control. We use viscosity techniques in the spirit of the general theory \cite{crandall1992user,barles1994solutions}, together with ideas from state-constrained control such as \cite{soner1986optimal,AlvarezLasryLions1997,katsoulakis1994viscosity,ishii2002class}. 
The viscosity characterization is also related to stochastic Perron's method, which provides a route to identifying value functions with viscosity solutions without first proving a dynamic programming principle. For stochastic control and Hamilton--Jacobi--Bellman (HJB) equations, this approach was developed in \cite{bayraktar2013stochastic}; see also \cite{bayraktar-sirbu2012linear,bayraktar-sirbu2014obstacle} for the linear and obstacle/Dynkin-game cases. Extensions to stochastic target problems and target games appear in \cite{bayraktar-li2016target,bayraktar-li2016games}, while the state-constrained version of \cite{rokhlin2014stochastic} is particularly close in spirit to the simplex constraint appearing in our state-constrained formulation. A related concave-envelope phenomenon appears in martingale optimal
transport with stopping: Bayraktar, Cox, and Stoev
\cite{BayraktarCoxStoev2018} use stochastic Perron's method to obtain a
finite-dimensional viscosity characterization and identify the solution
as the concave envelope of the payoff with respect to the atoms of the
terminal law. Related stochastic Perron arguments for constrained or singular financial-control
problems include lifetime ruin with transaction costs
\cite{BayraktarZhang2015ruin} and goal-based portfolio selection with mental
accounting \cite{BayraktarHanGoal2025}. Here the constraint is different: it is
the simplex constraint on the filter induced by causality.
The approximation results are in line with the broader literature on numerical schemes for stochastic control and fully nonlinear parabolic PDEs, including \cite{kushner-dupuis2001,barles-souganidis1991,gobet-lemor-warin2005,han-jentzen-e2018,hure-pham-warin2020,sirignano-spiliopoulos2018}. Finally, the zero-conditional-mean constraint in the filtering formulation is reminiscent of the inconspicuous trading condition in Kyle-type models of asymmetric information; see \cite{kyle,cho,back2020optimal,bose2020kyle,bose2021multidimensional}. In particular, \cite{first_paper} already pointed out the causality of the coupling as a main limitation in constructing an equilibrium in Kyle's model. In the present work this condition arises for a different reason: it enforces the fixed marginal law of the target process after conditioning on the target filtration.

The rest of the paper is organized as follows.
Section~\ref{sec:preliminaries} introduces the notation, standing
assumptions, and the notions of causal and filtering couplings used
throughout the paper. Section~\ref{sec:dynamic-formulations} derives
the dynamic formulations of the problem, including the filtering
formulation, the state-constrained control problem, and the associated
HJB equations. Section~\ref{sec:main-results}
states the main results, including the equivalence between the causal
transport value, the filtering formulation, and the state-constrained
formulation, as well as the approximation and duality results.
Section~\ref{sec:examples} presents examples and numerical
illustrations. Section~\ref{sec:main-proofs} contains the viscosity
characterization, comparison arguments, and the proofs of the main
theorems. Appendix~\ref{sec:technical-proofs} collects technical
proofs that are omitted in the main text.

\section{Preliminaries}\label{sec:preliminaries}

\subsection{Notation} 

\nd
Let $T>0$, $K\in\N$ with $K\ge 2$, and $d \in \mathbb{N}$. Denote by
$\cS_d$ the set of symmetric matrices of dimension $d$ and by
$\cS^{>0}_d$ its subset of positive definite elements. We consider
$\mathbb{S}=\{1,\ldots, K\}$ and $\Omega^X:=D([0,T];\mathbb{S})$, the
Skorokhod space of c\`adl\`ag paths from $[0,T]$ to $\mathbb{S}$. We also set $\Sigma_K:=\left\{(p_1,\ldots,p_K)\in \R^K: p_i\in [0,1],\;\sum_ip_i=1\right\}$.\\
We denote by $X$ the canonical process on $\Omega^X$, i.e.\
$X_t(\omega)=\omega(t)$ for $\omega\in \Omega^X$, and by
$\cF^X=(\cF^X_t)_{t\in[0,T]}$ its natural filtration. The space of all
probability measures on $\Omega^X$ is denoted by $\cP(\Omega^X)$.\\
Let $\Omega^Y:=C([0,T];\R^d)$ be the space of continuous paths on
$[0,T]$ taking values in $\R^d$. We denote by $Y$ the canonical
process on $\Omega^Y$, i.e.\ $Y_t(\omega)=\omega(t)$ for
$\omega\in\Omega^Y$, and by $\cF^Y=(\cF^Y_t)_{t\in[0,T]}$ its natural
filtration. We write $\mathbb{W}$ for the $d$-dimensional Wiener
measure, i.e., the law of standard $d$-dimensional Brownian motion.
\vspace{1mm}

\nd
We define the product space $\Omega:=\Omega^X\times \Omega^Y$, equipped with the product $\sigma$-algebra $\cF^X_T\otimes \cF^Y_T$ and the product filtration $(\cF^X_t\otimes \cF^Y_t)_{t\in[0,T]}$. This is the space on which all couplings of the laws of $X$ and $Y$ are defined.
\vspace{1mm}

\nd
\textit{Completed filtrations.} Given a probability measure $\mathbb{Q}$ on any of the above spaces, we write ${}^{\mathbb{Q}}\cG_t$ for the $\mathbb{Q}$-completion of a filtration $(\cG_t)_{t\in[0,T]}$, i.e.\ the smallest right-continuous filtration containing $(\cG_t)$ and all $\mathbb{Q}$-null sets. In particular, for $\mu\in\cP(\Omega^X)$, $\nu\in\cP(\Omega^Y)$ and $\pi\in\cP(\Omega)$, we use ${}^\mu\cF^X_t$, ${}^\nu\cF^Y_t$, and ${}^\pi(\cF^X_t\otimes\cF^Y_t)$ for the respective completed filtrations.
\vspace{1mm}

\nd
\textit{Marginals.} For $\mu\in\cP(\Omega^X)$, $\nu\in\cP(\Omega^Y)$
and $t\in[0,T]$, we write $\mu_t:=\mu\circ X_t^{-1}\in\cP(\mathbb{S})$
and $\nu_t:=\nu\circ Y_t^{-1}\in\cP(\R^d)$ for the laws of $X_t$ and
$Y_t$. We identify $\mu_t$ with the vector
$(\mu_t(\{i\}))_{i\in\mathbb{S}}\in\Sigma_K$.
\vspace{1mm}

\nd
\textit{Row vector convention for the simplex.} Throughout the paper, elements $p\in \Sigma_K$ are treated as \emph{row vectors} in $\R^K$. For a matrix $\Lambda\in \R^{K\times K}$, the product $p\Lambda\in \R^K$ is the standard right-multiplication, with components
\[
(p\Lambda)^i:=\sum_{j=1}^K p^j\,\Lambda_{ji},\qquad i=1,\ldots,K.
\]
Similarly, $p\,e^{t\Lambda}\in \R^K$ denotes the row vector obtained by right-multiplication.
\vspace{2mm}

\nd
\textit{General notation.} We use the notation $Z_\cdot$ for the path of a stochastic process $Z$. For any metric space $(\mathcal{X},d)$ we denote by $B_r(x)=\{ z\in\mathcal{X}:\;d(x,z)<r\}$ the open ball centered at $x\in \mathcal{X}$ with radius $r>0$, and by $\overline{B}_r(x)$ its closure.\\
For $A\subset \R^n$, for some $n\in\mathbb{N}$, and a function $u(t,y,p)=u\colon [0,T]\times \R^d\times \Sigma_K\to \R$, we denote by $\partial_tu,\; D_yu,\; D_pu$ its its first order derivatives. We write $u\in C^{1,2,2}([0,T]\times \R^d\times A)$ or simply $u\in C^{1,2,2}$, if its first order derivatives, as well as  the second order derivatives $D^2_{yy}u,\; D^2_{pp}u$ and $D^2_{yp}u=D_pD_yu$ exist and are continuous.

\subsection{Setup} We fix $\nu\in \cP(\Omega^Y)$ and $\mu\in \cP(\Omega^X)$ such that the following assumption holds.
\begin{assumption} \label{assumption1}
Under $\mu$, $X$ is an irreducible Markov chain with generator $\La$ and transition matrix $P(t)=e^{t\La}$. Under $\nu$, $Y$ satisfies
    \be \label{Ydynamics}
    dY_t=b_{\nu}(t,Y_t)dt+\sigma_{\nu}(t,Y_t)dW_t,
    \ee
    where $W$ is a $(\nu,\cF^Y)$ standard $d$-dimensional Brownian motion and $b_{\nu}\colon [0,T]\times \R^d\to \R^d$ and $\sigma_{\nu}\colon [0,T]\times \R^d\to \cS_d$ satisfy
    \begin{itemize}
        \item $\sigma_\nu,b_\nu$ are bounded, jointly continuous, and Lipschitz in the $y$ variable;
        \item $\sigma_\nu(t,y)\ge \kappa I_d$, for some $\kappa>0$.
    \end{itemize}
\end{assumption}

\nd
Under these assumptions, by \cite[Chapter 3]{sv} there exists a kernel
$K_{\nu}$ such that
$$\frac{d\nu_t}{dx}(x)=\int K_{\nu}(s,t,y,x)\;d\nu_s(y),$$
for all $0\le s\le t\le T$.
\vspace{1mm}

\nd
Denote by $\Pi(\mu,\nu)\subset\cP(\Omega)$ the set of all couplings
$\pi$ of $\mu$ and $\nu$ such that, under $\pi$, the law of $X$ is
$\mu$ and the law of $Y$ is $\nu$. We will define certain subclasses
of $\Pi(\mu,\nu)$ and, for each one, an optimal transport problem
based on the cost function $\cC\colon\Omega^X\times\Omega^Y\to\R$. We
assume that $\cC$ satisfies the following.
\begin{assumption}\label{costassumption}
    $\cC\colon \Omega\to \R$ has the form
    $$\cC(X,Y)=\int_0^Tf_0(X_s,Y_s)ds+g_0(X_T,Y_T),$$
    where $f_0,g_0\colon \mathbb{S}\times \R^d\to \R $ are Lipschitz in $y$.
\end{assumption}
\vspace{3mm}

\nd
\textit{Causal Couplings.} The first class of couplings we consider
is the class of causal couplings. For $\pi\in\Pi(\mu,\nu)$ we denote
by $\pi_x$ the disintegration of $\pi$ along $X$, which exists and is
$\mu$-a.e.\ unique since $\Omega^X$ and $\Omega^Y$ are Polish spaces
and hence every Borel probability measure on
$\Omega=\Omega^X\times\Omega^Y$ is Radon (see e.g.\
\cite[Theorem~5.3.1]{ambrosio2005gradient}).
\begin{definition}\label{def:causal}
A coupling $\pi \in \Pi(\mu,\nu)$ is causal from $X$ to $Y$ if for
all $t\in [0,T]$ and $A\in \cF^Y_t$, the map
$\Omega^X\ni x\mapsto \pi_x(A)$ is ${}^\mu\cF^X_{t}$-measurable. We
denote by $\Pi_{c}(\mu,\nu)$ the set of such couplings.
\end{definition}
\nd
The following proposition is established in \cite{acciaio2020causal, bartl2025wasserstein}.
\begin{proposition}\label{causalprop}
    Let $\pi \in \Pi(\mu,\nu)$. The following statements are equivalent.
    \begin{itemize}
        \item[(i)] $\pi\in \Pi_{c}(\mu,\nu).$
        \item[(ii)]  For all $t\in [0,T]$, $\{\emptyset,\Omega^X\}\otimes \cF_t^Y$ and $\cF_T^X\otimes \{\emptyset,\Omega^Y\}$ are $\pi$-conditionally independent given $\cF^X_{t}\otimes \{\emptyset,\Omega^Y\}$.
        \item[(iii)] $\pi(\Omega^X\times A_t | \cF^X_{t}\otimes \{\emptyset,\Omega^Y\})=\pi(\Omega^X\times A_t|\cF^X_T\otimes \{\emptyset,\Omega^Y\})$ for all $A_t\in \cF^Y_t.$
        \item[(iv)] ($H$-hypothesis) Any $\left({}^\mu \cF_{t}^X\otimes \{\emptyset,\Omega^Y\}\right)_{t\in[0,T]}$-martingale is also an $\left({}^\mu \cF_{t}^X\otimes \cF_t^Y\right)_{t\in[0,T]}$-martingale.
    \end{itemize}
\end{proposition}
\begin{remark}\label{Hhyprem}
    The $H$-hypothesis means that under $\pi$, passing from the
    semimartingale decomposition of $X$ in
    $\left({}^\mu \cF_{t}^X\otimes \{\emptyset,\Omega^Y\}\right)_{t\in[0,T]}$
    to its decomposition in
    $\left({}^\mu \cF_{t}^X\otimes \cF_t^Y\right)_{t\in[0,T]}$ does
    not introduce an additional drift. That is, additional knowledge
    of $(Y_s)_{s\in[0,t]}$ does not change the dynamics of $X$ if
    $(X_s)_{s\in[0,t]}$ is already known.
\end{remark}
\nd
We define the causal optimal transport problem (COT) as
\be\label{causalot}\tag{COT}
V_c(\mu,\nu):=\sup_{\pi\in \Pi_c(\mu,\nu)}\E^\pi[\cC(X,Y)].
\ee

\nd
Our goal is to compare this with the value of a filtering-based
problem, defined on a second subclass of $\Pi(\mu,\nu)$: the class
of filtering couplings.
\vspace{2mm}

\nd
\textit{Filtering Couplings.} The second class, denoted
$\Pi_f(\mu,\nu)$ (the subscript $f$ standing for \emph{filtering}),
consists of all probability measures $\pi$ on $\Omega$ under which
\be\label{filteringcond}
\begin{cases}
    X\text{ has law }\mu,\\[2pt]
    dY_t=\bigl(h(t,Y_\cdot,X_t)+b_{\nu}(t,Y_t)\bigr)\,dt
      +\sigma_\nu(t,Y_t)\,dW_t^\pi,
\end{cases}
\ee
where $W^\pi$ is a $\pi$-Brownian motion with respect to the
($\pi$-completed) filtration ${}^\pi(\cF^X_t\otimes\cF^Y_t)$, and
$h\colon[0,T]\times\Omega^Y\times\mathbb{S}\to\R^d$ is a bounded
measurable functional such that $h(t,\cdot,i)$ is
$\cF_t^Y$-measurable for every $(t,i)\in[0,T]\times\mathbb{S}$, and
satisfies the zero-conditional-mean condition
\be\label{hcond}
\E^\pi\!\left[h(t,Y_\cdot,X_t)\,\middle|\,\cF_t^Y\right]=0.
\ee
Note that this latter condition is nonlinear in $h$ since $\pi$ depends on $h$. 
The associated \emph{filtering causal optimal transport} problem is
\be\label{FCOT}\tag{FCOT}
V_f(\mu,\nu):=\sup_{\pi\in \Pi_f(\mu,\nu)}\E^\pi[\cC(X,Y)].
\ee

\nd
The terminology is justified by Proposition~\ref{inclusions} below,
where we show that \eqref{hcond} forces the $\pi$-marginal of $Y$
to be $\nu$.

\begin{proposition}\label{inclusions} Under Assumption~\ref{assumption1},
    $\Pi_f(\mu,\nu)\subseteq \Pi_c(\mu,\nu)$ and
    $V_f(\mu,\nu)\le V_c(\mu,\nu)$.
\end{proposition}
\begin{proof}
    See Appendix~\ref{technical}.
\end{proof}


\begin{remark}
The filtering formulation \eqref{filteringcond} is consistent with
the $H$-hypothesis. Under any $\pi\in\Pi_f(\mu,\nu)$, the
$H$-hypothesis holds between
$\left({}^\mu \cF_t^X\otimes \{\emptyset,\Omega^Y\}\right)_{t\in[0,T]}$
and
$\left({}^\mu \cF_t^X\otimes \cF_t^Y\right)_{t\in[0,T]}$, because
enlarging the $X$-filtration by $\cF^Y$ does not alter the
semimartingale decomposition of $X$: in \eqref{filteringcond}, $X$ is
the underlying process whose dynamics are fixed by $\mu$, while the
dynamics of $Y$ are controlled by $h$. Indeed, the $H$-hypothesis
need not hold with the roles of $X$ and $Y$ exchanged: for an
appropriate choice of $h$ (or any choice of $h$ as above), it fails between
$\left(\{\emptyset,\Omega^X\}\times{}^\nu \cF_t^Y\right)_{t\in[0,T]}$
and
$\left({}^\pi(\cF_t^X\otimes \cF_t^Y)\right)_{t\in[0,T]}$, since
enlarging the $Y$-filtration by $\cF^X$ can alter the drift of $Y$.

For a concrete example, take $d=1$, $\mathbb{S}=\{-1,1\}$,
$\Lambda=0$ (so $X_t\equiv X_0$ with $X_0$ uniform on $\{-1,1\}$),
$b_\nu\equiv 0$, $\sigma_\nu\equiv 1$, and $Y_0=0$, so that
$\nu=\mathbb{W}$. Define $\pi$ by conditioning $Y$ on
$\{\operatorname{sign}(Y_T)=X_0\}$. Since
$\mathbb{W}\bigl(\operatorname{sign}(Y_T)=i\bigr)=\tfrac12$ for
$i\in\{-1,1\}$, the $Y$-marginal of $\pi$ is again $\mathbb{W}$,
so $\pi\in\Pi(\mu,\nu)$. This conditioning is a Doob transform with
kernel
\[
\mathbb{W}\bigl(\operatorname{sign}(Y_T)=i\,\big|\,Y_t=y\bigr)
   =\Phi\!\left(iy/\sqrt{T-t}\right),\qquad i\in\{-1,1\},
\]
where $\Phi$ and $\phi$ denote the standard normal cdf and density.
The resulting dynamics under $\pi$ are
\[
dY_t=\frac{X_0\,\phi\!\left(Y_t/\sqrt{T-t}\right)}{\sqrt{T-t}\,\Phi\!\left(X_0 Y_t/\sqrt{T-t}\right)}\,dt
     +dW_t^\pi,
\]
where $W^\pi$ is a Brownian motion in
${}^\pi(\cF_t^X\otimes\cF_t^Y)$. This fits
\eqref{filteringcond} with
\[
h(t,y,i):=\frac{i\,\phi\!\left(y/\sqrt{T-t}\right)}{\sqrt{T-t}\,\Phi\!\left(iy/\sqrt{T-t}\right)}:
\]
Bayes' rule yields $\pi(X_0=i\mid\cF_t^Y)=\Phi(iY_t/\sqrt{T-t})$, and
a direct computation gives $\E^\pi[h(t,Y_t,X_0)\mid\cF_t^Y]=0$, so
\eqref{hcond} holds. Consequently $Y$ is a martingale in
${}^\nu\cF_t^Y$ but acquires a nonzero drift in
${}^\pi(\cF_t^X\otimes\cF_t^Y)$, illustrating the announced
asymmetry of the $H$-hypothesis.

Strictly speaking, this example sits just outside the standing framework:
$\Lambda=0$ violates the irreducibility requirement in
Assumption~\ref{assumption1}, and $h$ is unbounded near $t=T$,
violating the boundedness condition in the definition of
$\Pi_f(\mu,\nu)$. A routine truncation --- replacing $T$ by
$T-\varepsilon$ in $h$ and perturbing $\Lambda$ to a small
irreducible generator --- produces an element of $\Pi_f(\mu,\nu)$
exhibiting the same asymmetry, with the displayed example recovered
in the limit $\varepsilon\downarrow 0$.
\end{remark}

\nd
We prove $V_f=V_c$ by exhibiting both as the value function of a
continuous-time stochastic optimal control problem governed by a
master equation (Theorem~\ref{mainpdes}, Corollary~\ref{dropassinitial}).
This master equation approach yields a new numerical scheme for the
causal optimal transport value $V_c$ (Theorem~\ref{Nmainpdes}). We
also present a dual representation for the state-constrained problem,
approximation results from above and below for the value function,
and a partial convergence result for finite state approximations.
Our main results are stated in Section~\ref{mainresults}.

\section{Dynamic formulations and associated PDEs}\label{sec:dynamic-formulations}

\nd
As mentioned in the introduction, we construct dynamic formulations
for \eqref{FCOT} and the associated PDEs. We first assume that the
initial distribution of $Y$ is a Dirac mass; this assumption will be
removed in Corollary~\ref{dropassinitial}.

\begin{assumption}\label{initialmeasures}
    The initial marginals are deterministic: $\nu_0=\delta_{y_0}$ for some $y_0\in \R^d$, and $\mu_0=p_0$ for some $p_0\in \Sigma_K$.
\end{assumption}

\nd
For any $\pi\in\Pi(\mu,\nu)$ and $t\in[0,T]$, define the
$\Sigma_K$-valued $\cF^Y$-adapted process $p_t=p_t^\pi$ by
$$p^i_t:=\pi(\{X_t=i\}\times\Omega\mid\cF_t^Y),\qquad i=1,\ldots,K.$$

\subsection{Filtering problem}
Suppose $\pi\in\Pi_f(\mu,\nu)$, so that the dynamics of $(X,Y)$ are
given by \eqref{filteringcond} for some bounded
$h\colon[0,T]\times\Omega^Y\times\mathbb{S}\to\R^d$ satisfying
\eqref{hcond}. The process $p^i$ satisfies the Kushner equation
(see \cite{Kushner1964})
\be\label{Kushner}
dp^i_s
  = (p_s\La)^i\,ds
  + p^i_s\,h(s,Y_\cdot,i)^\top
    \sigma_\nu(s,Y_s)^{-2}
    \bigl(dY_s - b_\nu(s,Y_s)\,ds\bigr),
\ee
where $h(s,Y_\cdot,i)\in\R^d$ is regarded as a column vector,
$\sigma_\nu(s,Y_s)^{-2}:=\bigl(\sigma_\nu(s,Y_s)^2\bigr)^{-1}\in\cS_d$
is the inverse of the diffusion covariance matrix --- well-defined
since $\sigma_\nu\ge\kappa I_d$ --- and the product
$h(s,Y_\cdot,i)^\top\sigma_\nu(s,Y_s)^{-2}
 \bigl(dY_s-b_\nu(s,Y_s)\,ds\bigr)\in\R$
is to be read as the canonical inner product of two $\R^d$-valued
quantities. Equivalently, with
$d\tilde W_s:=\sigma_\nu(s,Y_s)^{-1}(dY_s-b_\nu(s,Y_s)\,ds)$ the
$(\pi,\cF^Y)$-Brownian motion given by Proposition~\ref{inclusions},
\eqref{Kushner} reads
\begin{align}
    \label{eq:kushnerinnov}
    dp^i_s = (p_s\La)^i\,ds
        + p^i_s\,h(s,Y_\cdot,i)^\top\sigma_\nu(s,Y_s)^{-1}\,d\tilde W_s.
\end{align}

Note that the zero-mean condition \eqref{hcond} --- equivalently,
$\hat h(s,Y_\cdot,p_s):=\sum_{i=1}^K h(s,Y_\cdot,i)\,p^i_s
 =0\in\R^d$ --- has been used to simplify the innovation term.

\nd
Define $f,g\colon\Sigma_K\times\R^d\to\R$ by
\be\label{costs}
f(p,y):=\sum_{i=1}^K f_0(i,y)\,p^i,
\qquad
g(p,y):=\sum_{i=1}^K g_0(i,y)\,p^i.
\ee
Using \eqref{costs} and Assumption~\ref{costassumption},
\begin{align}
    \E^\pi\left[ \cC(X,Y)\right]
    &=\int_0^T\E^\pi\!\left[\E^\pi\!\left[f_0(X_s,Y_s)
      \mid \cF^Y_s \right]\right]ds
     +\E^\pi\!\left[\E^\pi[g_0(X_T,Y_T)\mid\cF^Y_T]\right]
     \nonumber\\
    &= \E^\nu\!\left[\int_0^T f(p_s,Y_s)\,ds+g(p_T,Y_T)\right].
    \label{calc}
\end{align}
We therefore consider the optimal control problem
\be\label{dynamicfcot}\tag{FCOT-df}
U_f(t,y,p):=\sup_{h,\;(p_s)_{s\in [t,T]}}
\E^\nu\!\left[\int_t^T f(p_s,Y_s)\,ds+g(p_T,Y_T)\right],
\ee
where the supremum is taken over all bounded processes
$h\colon[t,T]\times\Omega^Y\times\mathbb{S}\to\R^d$ and processes
$(p_s)_{s\in[t,T]}$ taking values in $\Sigma_K$ such that $p$
satisfies \eqref{Kushner} on $[t,T]$ with $p_t=p$, and $h$ is bounded and satisfies
the compatibility condition
\be\label{hcondp}
\hat h(s,Y_\cdot,p_s)
  :=\sum_{i=1}^K h(s,Y_\cdot,i)\,p_s^i = 0\in\R^d,
\ee
which is the expression of \eqref{hcond}. Note that in these
expressions $p$ depends on $h$, and the condition
$\hat h(s,Y_\cdot,p_s)=0$ is not a linear condition in $h$ due to the
dependence of $p$ on $h$. The above calculation shows that, under
Assumption~\ref{initialmeasures},
\be\label{filteringeq}
V_f(\mu,\nu)=U_f(0,y_0,p_0).
\ee


\subsection{State-constrained control problem}
The set $\Pi_f(\mu,\nu)$ provides a subset of couplings, so
\eqref{FCOT} provides a lower bound for \eqref{causalot}. Our goal
now is to obtain an upper bound. For this purpose, we identify a
property satisfied by every $\pi\in\Pi_c(\mu,\nu)$ that leads to a
dynamic value.

\begin{proposition}[Filter dynamics for COT]\label{lem:forward}
Let $\pi\in\Pi_c(\mu,\nu)$. Then there exists an $\cF^Y$-adapted,
square-integrable process $Z=(Z_t)_{t\in[0,T]}$ with $Z_t^i\in\R^d$
for each $i\in\mathbb{S}$, such that
\begin{align}\label{Zcond}
\sum_{i=1}^K Z_t^i = 0\in\R^d,
\qquad\text{and}\qquad
\mathbf{1}_{p_t^i=0}\,Z_t^i = 0\in\R^d,
\end{align}
and the conditional law
$p_t=\bigl(\pi(X_t=i\mid\cF_t^Y)\bigr)_{i=1,\ldots,K}$
satisfies, for $i=1,\ldots,K$,
\begin{align}\label{eq:dynfiltercot}
    dp^i_t = (p_t\La)^i\,dt
     + (Z_t^i)^\top \sigma_\nu^{-2}(t,Y_t)\bigl(dY_t-b_\nu(t,Y_t)\,dt\bigr),
\end{align}
with initial condition
\begin{align}\label{eq:initialcot}
    p^i_0 = \mu(X_0=i).
\end{align}
Here, $(Z_t^i)^\top\sigma_\nu^{-2}(t,Y_t)(dY_t-b_\nu(t,Y_t)\,dt)\in\R$
is to be read as the canonical inner product of two $\R^d$-valued
quantities. Equivalently, with
$d\tilde W_t:=\sigma_\nu(t,Y_t)^{-1}(dY_t-b_\nu(t,Y_t)\,dt)$ the
$d$-dimensional $(\pi,\cF^Y)$-Brownian motion,
$$
dp^i_t = (p_t\La)^i\,dt
  +\bigl(Z_t^i\bigr)^\top\sigma_\nu^{-1}(t,Y_t)\,d\tilde W_t,
$$
where we used the symmetry $\sigma_\nu^\top=\sigma_\nu$.
\end{proposition}

\begin{proof}
    See Appendix~\ref{technical}.\end{proof}

\nd
Proposition~\ref{lem:forward} allows us to pinpoint the upper value
as follows. By \eqref{calc}, for any $\pi\in\Pi_c(\mu,\nu)$ the
expected cost $\E^\pi[\cC(X,Y)]$ depends on $\pi$ only through the
joint dynamics of $(p_t,Y_t)_{t\in[0,T]}$, where
$p_t=(p_t^i)_{i\in\mathbb{S}}$ is the $\pi$-conditional law of $X_t$
given $\cF_t^Y$. By Proposition~\ref{lem:forward}, the process $p$
satisfies \eqref{eq:dynfiltercot}--\eqref{eq:initialcot} for some
$\cF^Y$-progressively measurable process $Z=(Z_t)_{t\in[0,T]}$ with
$Z_t^i\in\R^d$ for each $i\in\mathbb{S}$, subject to \eqref{Zcond}.

The following lemma shows that these conditions in fact enforce a state constraint.
\begin{lemma}\label{invariance}
    Suppose Assumption~\ref{assumption1} holds, and let
    $(p_t)_{t\in[0,T]}$ satisfy \eqref{Zcond},
    \eqref{eq:dynfiltercot} and \eqref{eq:initialcot}. Then
    $p_t\in\Sigma_K$ for $dt\otimes\nu$-a.e.\ $(t,Y_\cdot)$.
\end{lemma}
\begin{proof}
    See Appendix~\ref{technical}.
\end{proof}

\nd
This motivates the following upper bound for \eqref{causalot},
which is a \emph{state-constrained optimal control}
problem~(SC):
\be\label{dynamiccot}\tag{SC}
U_{sc}(t,y,p):=\sup_{Z\in\cA(t,y,p)}
\E^\nu\!\left[\int_t^T f(p_s,Y_s)\,ds+g(p_T,Y_T)\right],
\ee
where $\cA(t,y,p)$ denotes the set of $\cF^Y$-progressively
measurable processes $Z=(Z_s)_{s\in[t,T]}$ with $Z_s^i\in\R^d$ for
each $i\in\mathbb{S}$ such that \eqref{Zcond} holds, and the pair
$(p_s,Y_s)_{s\in[t,T]}$ solves \eqref{eq:dynfiltercot} and
\eqref{Ydynamics} on $[t,T]$ with initial condition
$(p_t,Y_t)=(p,y)$.
\vspace{2mm}

\nd
By construction,
\be\label{scineq}
U_f(0,y_0,p_0)\le V_c(\mu,\nu)\le U_{sc}(0,y_0,p_0).
\ee

\nd

\begin{remark}
A combination of Proposition~\ref{lem:forward} and Lemma~\ref{invariance}
suggests that $\pi\in\Pi_c(\mu,\nu)$ if and only if the conditional
distribution
\[
p^i_t:=\pi\bigl(\{X_t=i\}\times\Omega\,\big|\,\cF_t^Y\bigr),
\qquad i=1,\ldots,K,
\]
is the filter associated with some $Z$ satisfying \eqref{Zcond}. If
this equivalence could be established directly, the identity
$U_{sc}(0,y_0,p_0)=V_c(\mu,\nu)$ would follow at once under
Assumption~\ref{initialmeasures}. However, the natural construction
of $\pi$ from $Z$ proceeds via a Girsanov change of measure driven by
the process $Z^i/p^i$, which need not be bounded — nor even
square-integrable — when $p^i$ approaches the boundary of $\Sigma_K$,
and Novikov's condition cannot be verified in general. We therefore
take a different route: we establish only the inequality
$V_c(\mu,\nu)\le U_{sc}(0,y_0,p_0)$ at this stage, and recover the
matching lower bound via a PDE argument later, bypassing the
a priori identification of $V_c(\mu,\nu)$ with $U_{sc}(0,y_0,p_0)$.
\end{remark}
\subsection{Associated PDEs} Ignoring the unboundedness of the controls, standard optimal control
considerations applied to \eqref{dynamicfcot} and \eqref{dynamiccot}
lead us to expect that $U_f$ and $U_{sc}$ satisfy
Hamilton--Jacobi--Bellman equations on the state space
$[0,T)\times\R^d\times\Sigma_K$. Since the belief process $p_t$ is
constrained to remain in the simplex $\Sigma_K$, these equations take
the form of \emph{state-constrained} PDEs in the sense of
\cite{soner1986optimal,ishii2002class,katsoulakis1994viscosity,rokhlin2014stochastic}:
the supersolution property is required only in the interior
$\Sigma_K^\circ$, while the subsolution property must hold on all of
$\Sigma_K$, including the boundary. We make this precise below.

\vspace{2mm}
\nd
Throughout this section we work with the joint gradient and Hessian of
a smooth function $V$ in the variable $(p,y)\in\Sigma_K\times\R^d$:
\be\label{jointGH}
D_{(p,y)}V
  := \begin{pmatrix}D_pV\\D_yV\end{pmatrix}\in\R^{K+d},
\qquad
D^2_{(p,y)}V
  := \begin{pmatrix}D^2_{pp}V & D_pD_yV\\(D_pD_yV)^\top & D^2_{yy}V\end{pmatrix}
    \in\cS_{K+d},
\ee
where $D_pV\in\R^K$, $D_yV\in\R^d$, $D^2_{pp}V\in\cS_K$,
$D_pD_yV\in\R^{K\times d}$, and $D^2_{yy}V\in\cS_d$.
We introduce the joint drift
\be\label{jointdrift}
\beta(t,y,p)
  := \bigl(p\Lambda,\;b_\nu(t,y)\bigr)\in\R^{K+d},
\ee
collecting the drift of the filter $p$ under $\Lambda$ and the drift
of $Y$. For a control matrix $z\in\R^{K\times d}$ with rows
$(z^i)_{i=1}^K\subset\R^d$, we define the joint quadratic-variation
matrix
\be\label{Qtilde}
\widetilde{Q}(t,y,z)
  := \begin{pmatrix}
       z\,\sigma_\nu(t,y)^{-2}\,z^\top & z \\
       z^\top & \sigma_\nu^2(t,y)
     \end{pmatrix}\in\cS_{K+d}.
\ee
This is the instantaneous covariance of the joint state $(p_t,Y_t)$
under the controlled dynamics: the top-left block
$z\sigma_\nu^{-2}z^\top$ is the quadratic variation of $p$, the
$(yy)$-block $\sigma_\nu^2$ is that of $Y$, and the off-diagonal
block $z$ is the cross-variation $d[p,Y]_t=z\,dt$ derived from
Proposition~\ref{lem:forward}.

\vspace{2mm}
\nd
The HJB operator $H_{sc}\colon[0,T)\times\R^d\times\Sigma_K
\times\R\times\R^{K+d}\times\cS_{K+d}\to\R$ is defined by
\begin{align}\label{Hsc}
&H_{sc}\!\left(t,y,p,\,\partial_tV,\,D_{(p,y)}V,\,
  D^2_{(p,y)}V\right)\notag\\
  &\qquad:= -\partial_tV
     - \beta(t,y,p)\cdot D_{(p,y)}V
     - \sup_{\substack{z\,\in\,\R^{K\times d} \\
       \sum_i z^i\,=\,0}}
       \tfrac{1}{2}\tr\!\bigl(\widetilde{Q}(t,y,z)\;D^2_{(p,y)}V\bigr)
     - f(p,y).
\end{align}
Expanding $\widetilde{Q}$ via \eqref{Qtilde} and writing
$D^2_{(p,y)}V=\bigl(\begin{smallmatrix}X&m\\m^\top&A\end{smallmatrix}\bigr)$
with $X:=D^2_{pp}V\in\cS_K$, $m:=D_pD_yV\in\R^{K\times d}$, and
$A:=D^2_{yy}V\in\cS_d$, we get
\begin{align}
\tfrac{1}{2}\tr\!\bigl(\widetilde{Q}(t,y,z)\;D^2_{(p,y)}V\bigr)
  &= \tfrac{1}{2}\tr\!\bigl(z\,\sigma_\nu(t,y)^{-2}\,z^\top X\bigr)
    + \tr\!\bigl(z\,m^\top\bigr)
    + \tfrac{1}{2}\tr\!\bigl(\sigma_\nu\sigma_\nu^\top(t,y)\,A\bigr),
    \label{Hscexpanded}
\end{align}
so \eqref{Hsc} is equivalently written as
\be\label{Hscalt}
H_{sc}
  = -\partial_tV
    - \cA_\nu V
    - D_pV\cdot p\Lambda
    - f(p,y)
    - \sup_{\substack{z\,\in\,\R^{K\times d}\\[1pt]\sum_i z^i\,=\,0}}
      \Bigl\{
        \tr\!\bigl(z\,m^\top\bigr)
        + \tfrac{1}{2}\tr\!\bigl(z\,\sigma_\nu(t,y)^{-2}\,z^\top X\bigr)
      \Bigr\},
\ee
where $\cA_\nu$ denotes the generator of $Y$ from \eqref{Ydynamics},
given by
\be\label{Ygenerator}
\cA_\nu
= \tfrac{1}{2}\tr\!\bigl(\sigma_\nu\sigma_\nu^\top(t,y)\,D^2_{yy}\bigr)
  + b_\nu(t,y)\cdot D_y.
\ee
\vspace{1mm}

\nd
We note that $H_{sc}$ also arises as the Hamiltonian of the filtering
problem \eqref{dynamicfcot}. Indeed, writing the filtering control as
$z^i=p^i h^i$ with $h^i\in\R^d$, the constraint
$\sum_i z^i=\sum_i p^i h^i=0$ in $\R^d$ matches \eqref{hcondp}, and
the supremum over $h$ is equivalent to the supremum over $z$ in
\eqref{Hsc}. On the boundary of $\Sigma_K$, the filtering formulation
automatically enforces $z^i=0$ whenever $p^i=0$, but this is already
captured by the state-constrained viscosity framework (subsolution
tested only in $\Sigma_K^\circ$). Therefore, both $U_f$ and $U_{sc}$
are expected to satisfy the same state-constrained PDE with
Hamiltonian $H_{sc}$.

\begin{definition}[State-constrained viscosity solution]
\label{def:stateconstrvisc}
Set $\cO_T:=[0,T)\times\R^d\times\Sigma_K$ and suppose that $H,\;\mathcal{B}_0\colon[0,T)\times\R^d\times\Sigma_K\times\R\times\R^{K+d}\times
\cS_{K+d}\to\R$ are given functions.

\vspace{1mm}
\nd (i) An upper semicontinuous function $V\colon\cO_T\to\R$ is a
\emph{viscosity subsolution} of the state-constrained problem if
for every $\varphi\in C^{1,2,2}(\cO_T)$ and every point
$(t_0,y_0,p_0)\in\cO_T$ at which $V-\varphi$ attains a local maximum,
\be\label{eq:subsol}
H\!\left(t_0,y_0,p_0,\,\partial_t\varphi,\,
  D_{(p,y)}\varphi,\,D^2_{(p,y)}\varphi\right)\le 0.
\ee
That is, the subsolution inequality is tested at all points
of $\Sigma_K$, including its boundary.

\vspace{1mm}
\nd (ii) A lower semicontinuous function $V\colon\cO_T\to\R$ is a
\emph{viscosity supersolution} of the state-constrained problem if
for every $\varphi\in C^{1,2,2}(\cO_T)$ and every point
$(t_0,y_0,p_0)\in[0,T)\times\R^d\times\Sigma_K^\circ$ at which
$V-\varphi$ attains a local minimum,
\be\label{eq:supersol}
H\!\left(t_0,y_0,p_0,\,\partial_t\varphi,\,
  D_{(p,y)}\varphi,\,D^2_{(p,y)}\varphi\right)\ge 0.
\ee
That is, the supersolution inequality is tested only in the interior
of $\Sigma_K$. In addition, if $(t_0,y_0,p_0)\in [0,T)\times \R^d\times \partial\Sigma_K$, then
\be\label{eq:bdsupersol}
\mathcal{B}_0\left(t_0,y_0,p_0,\partial\varphi, D_{(p,y)}\varphi,D^2_{p,y}\varphi\right)\ge 0.
\ee

\vspace{1mm}
\nd (iii) A continuous function $V$ is a \emph{viscosity solution} of
the state-constrained problem if it is both a viscosity subsolution
and a viscosity supersolution.
\end{definition}

\begin{remark}
The asymmetry between interior and boundary in
Definition~\ref{def:stateconstrvisc} reflects the geometry of the
state constraint $p_t\in\Sigma_K$. The supersolution property on
$\Sigma_K^\circ$ encodes optimality: the value function cannot be
improved by any admissible perturbation in the interior. The
subsolution property on all of $\Sigma_K$ encodes the constraint: at
a boundary point $p_0\in\pa\Sigma_K$, the dynamics must push $p_t$
back into the simplex (cf.\ Lemma~\ref{invariance}), and the value
function must be consistent with this inward drift. In the
terminology of
\cite{soner1986optimal,katsoulakis1994viscosity}, the subsolution
condition on $\pa\Sigma_K$ plays the role of a
\emph{state-constrained boundary condition}, replacing classical
boundary conditions such as Dirichlet or Neumann.

The generator $\mathcal{B}_0$ in \eqref{eq:bdsupersol} plays the role of $H_{in}$ in \cite{ishii2002class}, which is the generator corresponding to the inward-pointing control $z=0$.
\end{remark}

\nd
Since the controls $z^i\in\R^d$ in our problem are unbounded, the
supremum in \eqref{Hsc} is finite only when
$\tr(z\,\sigma_\nu^{-2}\,z^\top X)\le 0$ for all admissible $z$,
i.e.\ when $X=D^2_{pp}V$ is negative semidefinite on the subspace
$\bigl\{z\in\R^{K\times d}:\sum_{k=1}^K z^{k,i}=0
  \text{ for all }i=1,\ldots,d\bigr\}$.
Following \cite{bayraktar2013stochastic}, we replace $H_{sc}$ by the
modified operator
\be\label{Hvi}
H_{vi}\!\left(t,y,p,\,\partial_tV,\,D_{(p,y)}V,\,D^2_{(p,y)}V\right)
  :=\min\Bigl\{
      H_{sc}\!\left(t,y,p,\,\partial_tV,\,D_{(p,y)}V,\,D^2_{(p,y)}V\right),\;
      -\la_{\mathrm{max}}\!\left(D^2_{pp}V\right)
    \Bigr\},
\ee
where $\la_{\mathrm{max}}(D^2_{pp}V)$ denotes the largest eigenvalue
of the $pp$-block of $D^2_{(p,y)}V$. The PDE that $U_f$ and $U_{sc}$
are expected to satisfy then takes the form of a
\emph{variational inequality}:
\be\label{correctPDE}
\begin{cases}
H_{vi}\!\left(t,y,p,\,\partial_tV,\,D_{(p,y)}V,\,D^2_{(p,y)}V\right)
  \le 0
  & \text{on }[0,T)\times\R^d\times\Sigma_K,\\[4pt]
H_{vi}\!\left(t,y,p,\,\partial_tV,\,D_{(p,y)}V,\,D^2_{(p,y)}V\right)
  \ge 0
  & \text{on }[0,T)\times\R^d\times\Sigma_K^\circ,\\[6pt]
\mathcal{B}_0\!\left(t,y,p,\,\partial_tV,\,D_{(p,y)}V,\,D^2_{(p,y)}V\right)
  \ge 0
  & \text{on }[0,T)\times\R^d\times\pa\Sigma_K,
\end{cases}
\ee
in the viscosity sense (Definition~\ref{def:stateconstrvisc}) with
$H=H_{vi}$ and
\be\label{boundaryop}
\mathcal{B}_0\!\left(t,y,p,\,\partial_tV,\,D_{(p,y)}V,\,D^2_{(p,y)}V\right):=-\partial_tV-\mathcal{A}_\nu V-p\Lambda\cdot D_pV -f(p,y).
\ee

\section{Main results}\label{sec:main-results}\label{mainresults}
\nd

We now present the main results of the paper, covering Theorem \ref{thm:main_intro} from the introduction. We then explain how to construct monotone approximation sequences that are more tractable than the original problem and its upper and lower bounds.

\begin{theorem}\label{mainpdes}
Under Assumptions~\ref{assumption1} and \ref{costassumption}, the
variational inequality \eqref{correctPDE} admits a comparison
principle in the class of viscosity subsolutions $V_1\colon [0,T]\times \R^d\times \Sigma_K\to \R$ and supersolutions $V_2\colon [0,T]\times\R^d\times \Sigma_K\to \R$
of linear growth in $y$ (in the sense of
Definition~\ref{def:stateconstrvisc}) such that
\be\label{terminalassumption}
\limsup_{(t,y',p')\to (T,y,p)}V_1(t,y',p')-\liminf_{(t,y',p')\to (T,y,p)}V_2(t,y',p')\le 0.
\ee
In particular, there exists a unique continuous function $V\colon[0,T]\times\R^d\times\Sigma_K\to\R$ with linear growth in $y$ and $V(T,y,p)=g(p,y)$ that is a viscosity solution of \eqref{correctPDE}, in the sense of Definition \ref{def:stateconstrvisc}, and for all $(t,y,p)\in[0,T]\times\R^d\times\Sigma_K$, satisfies
$$V(t,y,p)=U_f(t,y,p)=U_{sc}(t,y,p)
 \mbox{ and } V(T,y,p)=U_f(T,y,p)=U_{sc}(T,y,p)=g(p,y).$$
If, in addition, Assumption~\ref{initialmeasures} holds, then
$V_c(\mu,\nu)=V(0,y_0,p_0)$.
\end{theorem}
\nd
\nd
Before stating our proof strategy, we contrast it with the classical
approach. The standard route to characterizing the value of a
stochastic control problem via a PDE proceeds in two stages: one
first establishes a \emph{dynamic programming principle} (DPP) for
the value function, and then derives the associated
Hamilton--Jacobi--Bellman equation (in the viscosity sense) from
this DPP.

This route is not directly available in our setting. The controls
$z\in\R^{K\times d}$ in \eqref{dynamiccot} are unbounded, and the
controls $h$ in \eqref{dynamicfcot} admit no a priori uniform
bound. As a consequence, it is not even clear whether $U_f$ and
$U_{sc}$ are measurable, let alone whether they satisfy any usable
form of the DPP.

The remedy, due to \cite{bayraktar2013stochastic,
rokhlin2014stochastic}, is the \emph{stochastic Perron method}. The
idea is to bypass the DPP for $U_f$ and $U_{sc}$ themselves and
instead bracket each of them between auxiliary functions for which
a one-sided version of the DPP can be checked along smooth test
functions. From these one-sided DPPs one reads off viscosity sub-
and super-solution properties of the appropriate semicontinuous
envelopes of $U_f$ and $U_{sc}$, even without first establishing
that $U_f$ and $U_{sc}$ are themselves measurable. A comparison
principle for the PDE then forces these envelopes to coincide,
which in turn forces $U_f=U_{sc}$ and identifies their common
value as the unique viscosity solution.

Concretely, the proof proceeds in three steps.

\vspace{2mm}
\nd
\textit{Step 1: Bracketing $U_f$ and $U_{sc}$ by stochastic Perron
envelopes.} We introduce the notions of \emph{stochastic sub-} and
\emph{supersolutions} of \eqref{correctPDE} (the precise
definitions are given in Section~\ref{mainproofs}) and set
\[
v^+ := \sup\bigl\{\,\text{stochastic subsolutions}\,\bigr\},
\qquad
v^- := \inf\bigl\{\,\text{stochastic supersolutions}\,\bigr\}.
\]
By construction,
\[
v^+\,\le\, U_f\,\le\, U_{sc}\,\le\, v^-,
\]
and each of $v^+$ and $v^-$ satisfies a one-sided sub-/super-DPP
that, unlike the DPP for $U_f$ and $U_{sc}$, can be tested against
smooth functions.

\vspace{1mm}
\nd
\textit{Step 2: Viscosity properties of the envelopes of $U_f$ and
$U_{sc}$.} Denote by $(U_{sc})^*$ the upper semicontinuous envelope
of $U_{sc}$ and by $(U_f)_*$ the lower semicontinuous envelope of
$U_f$. The one-sided DPPs of Step~1, together with the bracketing
$v^+\le U_f\le U_{sc}\le v^-$, imply, in the sense of
Definition~\ref{def:stateconstrvisc}, that
\begin{itemize}
    \item $(U_{sc})^*$ is a viscosity subsolution of
      \eqref{correctPDE};
    \item $(U_f)_*$ is a viscosity supersolution of
      \eqref{correctPDE}.
\end{itemize}
By construction, $(U_f)_*\le (U_{sc})^*$.

\vspace{1mm}
\nd
\textit{Step 3: Comparison and identification.} We establish a
comparison principle for \eqref{correctPDE} within the class of
semicontinuous sub- and supersolutions of linear growth in $y$.
Applied to $(U_{sc})^*$ and $(U_f)_*$, comparison yields the
reverse inequality $(U_{sc})^*\le (U_f)_*$, so the two envelopes
coincide with a single continuous function $V$. Combining with the
chain $(U_f)_*\le U_f\le U_{sc}\le (U_{sc})^*$ forces
$U_f = U_{sc} = V$, the unique viscosity solution of
\eqref{correctPDE}.

\vspace{2mm}
\nd
Dropping Assumption~\ref{initialmeasures}, we obtain a
characterization of $V_c(\mu,\nu)$ for general $\mu$ and $\nu$.
\begin{corollary}\label{dropassinitial}
Suppose Assumptions~\ref{assumption1} and \ref{costassumption} hold,
and let $V\colon[0,T]\times\R^d\times\Sigma_K\to\R$ be the viscosity
solution of \eqref{correctPDE} provided by
Theorem~\ref{mainpdes}. Then
\begin{align}\label{eq:initialdistthm}
    V_c(\mu,\nu)
  =\sup_{R\in \Pi(\mu_0,\nu_0)}
    \int_{\R^d} V(0,y,R_y)\,\nu_0(dy),
\end{align}
where $R_y\in\Sigma_K$ denotes the disintegration of $R$ with respect to
$\nu_0$ at the point $y\in\R^d$.
\end{corollary}
\begin{remark} Problem
\eqref{eq:initialdistthm} is a weak optimal transport problem, in the sense of Gozlan, Roberto, Samson and Tetali \cite{GoRoSaTe17}, between the time-zero marginals of the laws of the processes.
\end{remark}

\subsection{Approximation and Duality}
\nd
We now show that the value function $V$ characterized in
Theorem~\ref{mainpdes} admits a dual representation in terms of a
more tractable auxiliary singular stochastic control problem,
yielding an approximation from above. We then present a
complementary approximation from below, obtained by truncating the
controls in the state-constrained formulation \eqref{dynamiccot}.
\vspace{1mm}

\textit{Dual representation.} We derive a dual representation for $V$
by reformulating \eqref{dynamiccot} as an optimization over the
terminal random variable $p_T$. By Theorem~\ref{mainpdes},
\be\label{eq:primal_state}
V(t,y,p)
= \sup_{p_T} \E^\nu\!\left[\sum_{i=1}^K g_0(i,Y_T)\,p^i_T
+ \int_t^T\!\!\sum_{i,j=1}^K p^i_T\,P_{i,j}(s-T)\,f_0(j,Y_s)\,ds\right],
\ee
where the supremum is taken over $\cF^Y_T$-measurable
$\Sigma_K$-valued random variables $p_T$ subject to
\begin{align}
&(\textup{C1})\quad
\sum_{i=1}^K \E^\nu\!\left[p^i_T\,\big|\,\cF^Y_s\right]
  P_{i,j}(s-T)\ge 0
\qquad\text{for all }j\in\mathbb{S}\text{ and }s\in[t,T],
\label{eq:Cone}\\
&(\textup{C2})\quad
\sum_{i=1}^K \E^\nu\!\left[p^i_T\right]P_{i,j}(t-T)= p^j
\qquad\text{for all }j\in\mathbb{S}.\label{eq:Ctwo}
\end{align}

\nd
The constraints \eqref{eq:Cone}--\eqref{eq:Ctwo} arise from solving
the Kushner equation \eqref{eq:kushnerinnov} backward in time, as
follows. Recalling that $P(t)=e^{t\La}$, set $\tilde p_s:=p_s\,P(T-s)$.
Since $P(T-s)$ is deterministic with
$\tfrac{d}{ds}P(T-s)=-\La\,P(T-s)$, the product rule and
\eqref{eq:kushnerinnov} give
\[
d\tilde p^i_s
  = (dp_s\,P(T-s))^i - (p_s\La\,P(T-s))^i\,ds
  = \sum_{j=1}^K p^j_s\,h(s,Y_\cdot,j)^\top
    \sigma_\nu(s,Y_s)^{-1}P_{j,i}(T-s)\,d\tilde W_s,
\]
the deterministic drifts cancelling. Hence $\tilde p$ is an
$\cF^Y$-martingale, so
\begin{align}
\E^\nu\!\bigl[p_T\,\big|\,\cF^Y_s\bigr]=\tilde p_s=p_s\,P(T-s),
\qquad\text{equivalently}\qquad
p_s = \E^\nu\!\bigl[p_T\,\big|\,\cF^Y_s\bigr]\,P(s-T). \label{eq:tilde_p}
\end{align}
The requirement $p_s\in\Sigma_K$ for all $s\in[t,T]$ therefore reduces
to componentwise non-negativity --- since $P(s-T)$ has row sums equal
to one, the unit-sum is automatic --- which is exactly \eqref{eq:Cone}.
The initial condition $p_t=p$ becomes \eqref{eq:Ctwo}.
\vspace{1mm}

\nd
We dualize \eqref{eq:Cone}--\eqref{eq:Ctwo} by introducing Lagrange
multipliers: a static $\lambda\in\R^K$ for the equality
\eqref{eq:Ctwo}, and a family of non-negative
$\cF^Y$-progressively measurable processes
$(\lambda^j_s)_{j\in\mathbb{S},\,s\in[t,T]}$ for the inequality
\eqref{eq:Cone}. The set of $\cF^Y_T$-measurable $\Sigma_K$-valued
$p_T$ is convex and weak-$\ast$ compact in $L^\infty$, and the
Lagrangian is linear in both $p_T$ and the multipliers, so Sion's
minimax theorem permits the interchange of $\sup$ and $\inf$. Since
the supremum of a linear functional over $\Sigma_K$ is attained at an
extreme point, we obtain
\be\label{eq:saddle}
\begin{split}
V(t,y,p) &= \inf_{\lambda,\,(\lambda^j_s)}
  \E^\nu\!\Bigg[\,\max_{i\in\mathbb{S}}\!\bigg(
     g_0(i,Y_T)
   - \sum_{j=1}^K P_{i,j}(t-T)\,\lambda^j\\
&\hspace{3cm}
   + \int_t^T \sum_{j=1}^K P_{i,j}(s-T)\,
     \bigl(f_0(j,Y_s)+\lambda^j_s\bigr)\,ds\bigg)\Bigg]
+ \sum_{j=1}^K p^j\lambda^j,
\end{split}
\ee
where the infimum is taken over $\lambda\in\R^K$ and over
$\cF^Y$-progressively measurable $\R_+^K$-valued processes
$(\lambda^j_s)_{j\in\mathbb{S},\,s\in[t,T]}$ satisfying
$\sum_{j=1}^K\int_t^T\E^\nu[\lambda^j_s]\,ds<\infty$. Any admissible
choice of multipliers gives an upper bound on $V(t,y,p)$, yielding
the announced approximation from above.

\vspace{2mm}
\nd
\textit{Approximation from above.}
For each $i\in\mathbb{S}$, let
$G_0(i,\cdot,\cdot)\colon[0,T]\times\R^d\to\R$ denote the unique
bounded classical solution of the Kolmogorov backward equation
\be\label{eq:bkwd}
\partial_t G_0(i,t,y) + \cA_\nu G_0(i,t,y) = 0
  \quad\text{on }[0,T)\times\R^d,
\qquad G_0(i,T,y)=g_0(i,y),
\ee
where $\cA_\nu$ is given by \eqref{Ygenerator}. Existence is ensured
by Assumption~\ref{assumption1} and the boundedness of $g_0$ from
Assumption~\ref{costassumption}. We write
$G_0(t,y):=(G_0(i,t,y))_{i\in\mathbb{S}}\in\R^K$ for the resulting
vector. We further define the auxiliary value function
$u\colon[0,T]\times\R^d\times\R^K\to\R$ by
\be\label{eq:aux_u}
u(t,y,z)
:= \inf_{(\lambda^j_s)} \E^\nu\!\bigg[\max_{i\in\mathbb{S}}\bigg(z^i
+ \int_t^T \partial_y G_0(i,s,Y_s)\,\sigma_\nu(s,Y_s)\,dW_s
+ \int_t^T \sum_{j=1}^K P_{i,j}(s-T)\bigl(f_0(j,Y_s)+\lambda^j_s\bigr)\,ds\bigg)\bigg],
\ee
where $(Y_s)_{s\in[t,T]}$ solves \eqref{Ydynamics} with $Y_t=y$ and
the infimum is over $\cF^Y$-progressively measurable
$\R_+^K$-valued processes
$(\lambda^j_s)_{j\in\mathbb{S},\,s\in[t,T]}$ with
$\sum_{j=1}^K\int_t^T\E^\nu[\lambda^j_s]\,ds<\infty$.


\nd
For $N\in\N$, define
\be\label{eq:uN}
u^N(t,y,z)
:= \inf_{(\lambda^j_s)} \E^\nu\!\bigg[\max_{i\in\mathbb{S}}\bigg(z^i
+ \int_t^T \partial_y G_0(i,s,Y_s)\,\sigma_\nu(s,Y_s)\,dW_s
+ \int_t^T \sum_{j=1}^K P_{i,j}(s-T)\bigl(f_0(j,Y_s)+\lambda^j_s\bigr)\,ds\bigg)\bigg],
\ee
where the infimum is now over $\cF^Y$-progressively measurable
processes $(\lambda^j_s)_{j\in\mathbb{S},\,s\in[t,T]}$ satisfying
$0\le\lambda^j_s\le N$. Set
\be\label{eq:vN}
v^N(t,y,p)
:= \inf_{\lambda\in\R^K} \Bigl\{
   u^N\!\bigl(t,y,\,G_0(t,y)-P(t-T)\lambda\bigr)
   + p\cdot\lambda\Bigr\},
\ee
where $p\cdot\lambda:=\sum_{j=1}^K p^j\lambda^j$. The bound $\lambda^j_s\le N$ in \eqref{eq:uN} places $u^N$ within the scope of
standard numerical methods for stochastic control with bounded
controls (cf.\ \cite{kushner-dupuis2001,barles-souganidis1991,
fahim-touzi-warin2011,gobet-lemor-warin2005,han-jentzen-e2018,
hure-pham-warin2020,sirignano-spiliopoulos2018}).

\begin{theorem}\label{thm:dual}
Under Assumptions~\ref{assumption1} and \ref{costassumption}, we have for every $(t,y,p)\in[0,T]\times\R^d\times\Sigma_K$:
\begin{enumerate}
\item[\textup{(i)}] \emph{(Dual representation)} The value function
$V$ from Theorem~\ref{mainpdes} satisfies
\be\label{eq:final_representation}
V(t,y,p) = \inf_{\lambda\in\R^K}
  \Bigl\{u\bigl(t,y,\,G_0(t,y) - P(t-T)\lambda\bigr)
  + p\cdot\lambda\Bigr\}.
\ee
\item[\textup{(ii)}] \emph{(Approximation from above)} The sequence
$(v^N(t,y,p))_{N\in\N}$ is non-increasing, $v^N(t,y,p)\ge V(t,y,p)$
for every $N$, and
$$v^N(t,y,p)\;\downarrow\;V(t,y,p)\quad\text{as }N\to\infty.$$
\end{enumerate}
\end{theorem}

\vspace{2mm}
\nd
\textit{Approximation from below.} We now construct an approximation
$V^N\uparrow V$ by truncating the controls $Z$ in the
state-constrained formulation \eqref{dynamiccot}. Throughout,
$|z^i|$ and $|Z^i_s|$ denote the Euclidean norm on $\R^d$.

\nd
For $N\in\N$, define the truncated Hamiltonian $H_{sc}^N\colon[0,T)\times\R^d\times\Sigma_K\times\R\times\R^{K+d}\times\cS_{K+d}\to\R$ by
\begin{align}
&H_{sc}^N\!\left(t,y,p,\,\partial_tV,\,D_{(p,y)}V,\,
  D^2_{(p,y)}V\right)\notag\\
  &\qquad:= -\partial_tV
     - \beta(t,y,p)\cdot D_{(p,y)}V
     - \sup_{\substack{z\in\R^{K\times d}\\[1pt]\sum_{i=1}^K z^i=0\\[1pt]
                       |z^i|\le N\ \forall i\in\mathbb{S}}}
       \tfrac{1}{2}\tr\!\bigl(\widetilde{Q}(t,y,z)\;D^2_{(p,y)}V\bigr)
     - f(p,y),\label{HscN}
\end{align}
with $\beta$ and $\widetilde{Q}$ as in \eqref{jointdrift} and
\eqref{Qtilde}. This is the HJB operator associated with
\eqref{dynamiccot} under the additional admissibility constraint
$|Z^i_s|\le N$ for all $i\in\mathbb{S}$ and $s\in[t,T]$.

\begin{theorem}\label{Nmainpdes}
Suppose Assumptions~\ref{assumption1} and \ref{costassumption}
hold. Then the state-constrained problem
\be\label{NPDE}
\begin{cases}
    H_{sc}^N\!\left(t,y,p,\,\partial_tV^N,\,D_{(p,y)}V^N,\,
      D^2_{(p,y)}V^N\right)\leq0,
    & \text{in }[0,T)\times\R^d\times\Sigma_K,\\[4pt]
    H_{sc}^N\!\left(t,y,p,\,\partial_tV^N,\,D_{(p,y)}V^N,\,
      D^2_{(p,y)}V^N\right)\geq 0,
    & \text{in }[0,T)\times\R^d\times\Sigma_K^\circ,\\[4pt]
    \mathcal{B}_0\!\left(t,y,p,\,\partial_tV^N,\,D_{(p,y)}V^N,\,
      D^2_{(p,y)}V^N\right)\ge 0,& \text{in }[0,T)\times\R^d\times\pa\Sigma_K,\\
    V^N(T,y,p)=g(p,y),
    & \text{in }\R^d\times\Sigma_K,
\end{cases}
\ee
admits a unique viscosity solution
$V^N\colon[0,T]\times\R^d\times\Sigma_K\to\R$ (in the sense of
Definition~\ref{def:stateconstrvisc}) of linear growth in $y$.
Moreover, the sequence $(V^N)_{N\in\N}$ is non-decreasing,
$V^N(t,y,p)\le V(t,y,p)$ for every $N$, and
$$V^N(t,y,p)\;\uparrow\;V(t,y,p)\quad\text{as }N\to\infty,$$
for every $(t,y,p)\in[0,T]\times\R^d\times\Sigma_K$, where $V$ is
as in Theorem~\ref{mainpdes}.
\end{theorem}

\nd
We now provide a verification theorem characterizing $V_c$ in the
case where \eqref{correctPDE} admits a smooth solution. The proof
follows standard arguments and is omitted.

\begin{theorem}\label{verification}
Suppose that Assumption~\ref{assumption1} holds, except for the irreducibility
of $X$, that Assumption~\ref{costassumption}
holds, and that the variational inequality \eqref{correctPDE} admits a
solution
\[
V\in C^{1,2,2}\!\bigl([0,T)\times\R^d\times\Sigma_K\bigr)
\,\cap\, C^{0}\!\bigl([0,T]\times\R^d\times\Sigma_K\bigr).
\]
Suppose further that the maximizer $Z^*(t,y,p)$ in \eqref{Hsc} is
such that \eqref{eq:dynfiltercot} driven by $Z^*$ admits a
$\Sigma_K$-valued solution. Then the conclusions of
Theorem~\ref{mainpdes} hold.
\end{theorem}

\subsection{Extensions and Future Perspectives}
A question for future investigation is whether $X$ can be allowed to have
a continuous state space. In that case, although the stochastic optimal
control problems \eqref{dynamicfcot} and \eqref{dynamiccot} can still be
formulated, the PDE expected to be satisfied by the value functions is
defined on an infinite-dimensional space and, once again, has a state
constraint. To our knowledge, the methods used here to obtain a
comparison principle for \eqref{correctPDE} do not extend to infinite
dimensions, and new techniques are required to address this problem.
\vspace{2mm}

\nd
However, we believe that the causal optimal transport value for a
finite-state process $X$ may be useful for approximating the causal
optimal transport value for a continuous-state process $X$. Consider the
stochastic differential equation (together with an initial condition)
    \be \label{Xdynamics}
    dX_t=b_{\mu}(t,X_t)dt+\sigma_{\mu}(t,X_t)dW_t,
    \ee
    and assume that this equation satisfies pathwise uniqueness. Let
    $\mu$ be the law of its solution $X$. As explained in
    \cite[Theorem 8.1]{KuDi78}, if the coefficients of this equation
    are bounded, one can construct a sequence $\mu_n$ converging weakly
    to $\mu$, where $\mu_n$ is the law of a continuous-time Markov
    process $X^n$ with a countable state space. By a further
    approximation argument, we may assume that $X^n$ evolves in a
    finite state space. We show how the causal transport problem from
    $\mu$ to $\nu$ relates to the causal transport problem from $\mu_n$
    to $\nu$, at least in one direction:

\begin{proposition}\label{thm.cv}Suppose Assumption~\ref{costassumption} holds, and suppose that
$\mu$ and each $\mu_n$ are Feller, with $\mu_n\to\mu$ weakly. Then,
    $$\limsup_{n\to \infty}V_c(\mu_n,\nu)\le V_c(\mu,\nu).$$
\end{proposition}

\nd
We first need some preparation:

\begin{lemma}
    Let $\{\mu_n\}_n $ be a sequence of Feller processes converging weakly to a Feller process $\mu_\infty$. Suppose $\pi_n \in\Pi_{c}(\mu_n,\nu)$ are such that $\pi_n\to\pi_\infty$ weakly. Then $\pi_\infty \in\Pi_{c}(\mu_\infty,\nu)$.
\end{lemma}
\begin{proof}
It is known that $\pi \in\Pi_{c}(\mu,\nu)$ is equivalent to the statement
$$\mathbb E_\pi\left [ f_t(Y)\left\{  g(X) -\mathbb E_\mu[g(X)|\mathcal F^X_t] \right\} \right ]=0,$$
for all continuous bounded functions $f_t,g\in C_b(\Omega)$, with
$f_t$ being $\mathcal F^Y_t$-measurable, and every $t$. By standard
approximation arguments, when we specialize to the case where $\mu$ is
Markov, this is equivalent to
$$\mathbb E_\pi\left [ f_t(Y)\left\{  g(X_{t_1},\dots,X_{t_k}) -\mathbb E_\mu[g(X_{t_1},\dots,X_{t_k})|X_t] \right\} \right ]=0,$$
for all continuous bounded functions $f_t\in C_b(\Omega)$, with $f_t$
being $\mathcal F^Y_t$-measurable, every $g\in C_b(\mathbb R^k)$,
every $k\in\mathbb N$, and every $t\leq t_1\leq\dots\leq t_k$. An
application of the functional monotone class theorem further shows that
we may take $g$ of the form
$g(x_1,\dots,x_k)=g_1(x_1)\times\dots\times g_k(x_k)$. We now fix all
these quantities and functions. To obtain
$\pi_\infty \in\Pi_{c}(\mu_\infty,\nu)$, it suffices to show the uniform
convergence of
$\mathbb E_{\mu_n}[g_1(X_{t_1})\times\dots\times g_k(X_{t_k})|X_t=\cdot]$
to
$\mathbb E_{\mu_\infty}[g_1(X_{t_1})\times\dots\times g_k(X_{t_k})|X_t=\cdot]$.

\nd
For $n\in\mathbb N\cup\{\infty\}$, define the functions $G_{n,i}$
recursively as follows: set $G_{n,k}(x_k):= g_k(x_k)$, and, assuming
that $G_{n,i+1}(x_{i+1})$ has been defined for $1<i+1\leq k$, set
$G_{n,i}(x_i):=g_i(x_i)\times\mathbb E_{\mu_n}[G_{n,i+1}(X_{t_{i+1}})|X_{t_i}=x_i]$.
Thus, by the tower and Markov properties,
$\mathbb E_{\mu_n}[g_1(X_{t_1})\times\dots\times g_k(X_{t_k})|X_t=\cdot]=G_{n,1}(\cdot)$.
Clearly, $G_{n,k}\to G_{\infty,k}$ uniformly as $n\to\infty$. Now
assume that $G_{n,i+1}\to G_{\infty,i+1}$ uniformly as $n\to\infty$,
and note that
$$G_{n,i}(x_i):=g_i(x_i)\times P^n_{t_{i+1}-t_i}G_{n,i+1}(x_i),$$
where $(P^n_s)_s$ is the semigroup associated to $\mu_n$. Next we write 
$$P^\infty_{t_{i+1}-t_i}G_{\infty,i+1}-P^n_{t_{i+1}-t_i}G_{n,i+1}= P^n_{t_{i+1}-t_i}\{G_{\infty,i+1}-G_{n,i+1}\}-\{P^n_{t_{i+1}-t_i}-P^\infty_{t_{i+1}-t_i}\}G_{\infty,i+1},$$
so the first term on the right-hand side goes to zero uniformly by the
induction hypothesis, while the second goes to zero uniformly by the
assumption on weak convergence and \cite[Theorem 19.25]{Ka01}. We
conclude that $G_{n,1}\to G_{\infty,1}$ uniformly, as desired.
\end{proof}

\begin{proof}[Proof of Proposition \ref{thm.cv}]
If $\pi_{n}$ is (almost) optimal for $V_c(\mu_{n},\nu)$, then the previous lemma shows that any accumulation point is feasible for $V_c(\mu,\nu)$. By tightness, we conclude that 
$$\limsup_{n\to \infty}V_c(\mu_n,\nu)\leq V_c(\mu,\nu).$$
\end{proof}

\nd
We leave it as an interesting and challenging open question whether
Proposition \ref{thm.cv} holds as an equality. In fact, it may even
happen that the limit of $V_c(\mu_n,\nu)$ is not quite $V_c(\mu,\nu)$
but a ``relaxed'' version of it as in \cite{bartl2025wasserstein}.

Another possible generalization of our work is to drop the Markovian
structure of the cost function and/or the marginal processes. In such
cases one would need to do filtering in path space, and both the
analysis, which becomes infinite-dimensional, and the notation would
become significantly more complicated.

\section{Examples}\label{sec:examples}
\subsection{Examples beyond the irreducible case}\label{examples-no-irreducibility}

We next give examples in which the finite-state process $X$ is not
irreducible.  These examples are therefore not directly covered by
Theorem~\ref{mainpdes}.  Instead, we exhibit explicit candidates and
verify the variational inequality \eqref{correctPDE}; the identification
with the causal transport value then follows from the smooth verification
result, Theorem~\ref{verification}.  In the absorbing two-state example
below, we also record a dual/free-boundary representation.

Throughout this subsection, $Y$ is a one-dimensional Brownian motion
under $\nu=\mathbb W$, so $b_\nu=0$ and $\sigma_\nu=1$.  We write
$\Lambda$ for the generator of $X$ and $P(t):=e^{t\Lambda}$.  Since the
filter process $(p_t)_{t\in[0,T]}$ is required to satisfy
\eqref{eq:dynfiltercot}, the process $p_t e^{(T-t)\Lambda}$ is an
$(\cF^Y_t)$-martingale.  Equivalently,
\be\label{eq:filter-id}
p_t=\E\!\left[p_T\,\middle|\,\cF^Y_t\right]P(t-T),
\qquad 0\le t\le T,
\ee
where $P(t-T)=e^{(t-T)\Lambda}$ denotes the inverse semigroup.

\begin{example}[$\nu=\mathbb W$, $f\equiv0$, $g(x,y)=xy$, $X_t\equiv X_0$]\label{ex:constant}
Let $\mu$ be the law of a constant process $X_t\equiv X_0$, where
\[
\mu_0=\sum_{i=1}^K p_i\delta_{x_i},
\qquad x_1<\cdots<x_K,
\]
and take $f_0\equiv0$ and $g_0(i,y)=x_i y$.  For $p\in\Sigma_K$ set
\[
\rho_p:=\sum_{i=1}^K p_i\delta_{x_i},
\qquad
m(p):=\sum_{i=1}^K x_i p_i,
\]
and define, for probability measures $\eta,\rho$ on $\R$,
\[
\mathcal M(\eta,\rho)
:=
\sup_{\gamma\in\Pi(\eta,\rho)}
\int_{\R^2}uv\,\gamma(du,dv).
\]
For $\tau:=T-t$, the value of the causal problem is
\begin{equation}
V^{\rm term}(t,y,p)
=
\mathcal M\bigl(\rho_p,N(y,\tau)\bigr)
=
y\,m(p)+\sqrt{\tau}\,L(p),
\qquad
L(p):=\mathcal M\bigl(\rho_p,N(0,1)\bigr).
\end{equation}

Indeed, the upper bound follows by forgetting causality and optimizing
only over the terminal pair $(X_T,Y_T)$.  Conversely, let $\gamma^*$ be
the comonotone coupling between $\rho_p$ and $N(y,\tau)$, known to be optimal for $\mathcal M\bigl(\rho_p,N(y,\tau)\bigr)$, and disintegrate
it as
\[
\gamma^*(dx,dw)=\gamma^*_w(dx)N(y,\tau)(dw).
\]
Run a Brownian path $Y$ from $Y_t=y$, sample $X_0$ conditionally on
$Y_T=w$ according to $\gamma^*_w$, and set $X_s\equiv X_0$ for
$s\in[t,T]$.  This realizes the optimal terminal joint law.  It is also
causal: since $X$ is constant, the full $X$-path is determined by $X_0$,
and $\cF^X_s=\sigma(X_0)$ for every $s$.  Therefore the conditional law
of $Y_{[t,s]}$ given the full $X$-path is a function of $X_0$, hence is
$\cF^X_s$-measurable.

We now verify \eqref{correctPDE}.  Let $\Phi$ and $\varphi$ denote the
standard normal cdf and density, and put
\[
h(u):=\varphi(\Phi^{-1}(u)),
\qquad h(0)=h(1)=0.
\]
For $s_i:=p_1+\cdots+p_i$, $s_0:=0$, the monotone coupling gives
\[
L(p)=\sum_{i=1}^K x_i\{h(s_{i-1})-h(s_i)\}.
\]
Using reduced coordinates $(p_1,\ldots,p_{K-1})$ on the simplex, direct
differentiation yields
\[
\partial^2_{p_i p_j}L(p)
=
\sum_{\ell=i\vee j}^{K-1}
\frac{x_\ell-x_{\ell+1}}{h(s_\ell)}.
\]
Consequently, for every $\alpha,\beta>0$,
\begin{equation}
\sup_{\sum_i z_i=0}
\left\{
\alpha\sum_{i=1}^K x_i z_i
+
\frac{\beta}{2}z^\top D^2_{pp}L(p)z
\right\}
=
\frac{\alpha^2}{2\beta}L(p),
\end{equation}
where the quadratic form is read in the reduced coordinates.  The
maximizer is
\[
z_i^*
=
\frac{\alpha}{\beta}\{h(s_{i-1})-h(s_i)\},
\qquad i=1,\ldots,K,
\]
which satisfies $\sum_i z_i^*=0$.

For $V^{\rm term}$, the Hamiltonian term in \eqref{correctPDE} is the
preceding identity with $\alpha=1$ and $\beta=\sqrt{\tau}$, hence equals
$L(p)/(2\sqrt{\tau})$.  On the other hand,
\[
-\partial_tV^{\rm term}(t,y,p)
=
\frac{L(p)}{2\sqrt{\tau}},
\qquad
A_\nu V^{\rm term}=0,
\qquad
f\equiv0.
\]
Thus $V^{\rm term}$ solves \eqref{correctPDE} in the interior of each
face of the simplex; boundary faces are obtained by deleting zero-mass
states.  Since $V^{\rm term}(T,y,p)=y\,m(p)=g(p,y)$, the terminal
condition is satisfied.  Theorem~\ref{verification} therefore identifies
the displayed formula with the causal transport value.
\end{example}

\begin{remark}[Bicausality]
The preceding construction uses causality from $X$ to $Y$.  If the
opposite causality is also imposed and $Y_0$ is deterministic, then the
only admissible coupling is the product coupling.  Indeed, bicausality
requires $\operatorname{Law}(X_0\mid Y)$ to be $\cF^Y_0$-measurable;
since $\cF^Y_0$ is trivial and $X$ is constant, $X$ must be independent
of $Y$.  Hence, unless $\mu_0$ is degenerate, the causal and bicausal
values differ.
\end{remark}

\begin{example}[Initial optimization for constant $X$]\label{ex:constant2}
We next isolate the initial optimization that underlies both the terminal ($f\equiv 0$)
and running ($g\equiv 0$) versions  of the constant-state example ($X_0=X_t,\forall t$).  Assume
\[
\nu_0=\mu_0=\sum_{j=1}^K q_j\delta_{x_j},
\qquad x_1<\cdots<x_K,
\]
and let $Y_0\sim\nu_0$.  For $a>0$ and $s>0$, let
$Z\sim N(0,1)$ be independent of $Y_0$ and define
\[
S_{a,s}:=aY_0+sZ.
\]
Whenever the deterministic-start value has the form
\[
V^{a,s}(0,y,r)=\mathcal M\bigl(\rho_r,N(ay,s^2)\bigr),
\qquad r\in\Sigma_K,
\]
the conditioning argument used in the proof of
Corollary~\ref{dropassinitial}, together with the verification from
Theorem~\ref{verification}, reduces the initial problem to the static
maximal-covariance problem \eqref{eq:initialdistthm}, i.e.,
\begin{equation}
V_c^{a,s}(\mu,\nu)
=
\mathcal M\bigl(\mu_0,\operatorname{Law}(S_{a,s})\bigr).
\end{equation}

Let $Q_k:=q_1+\cdots+q_k$, $Q_0:=0$, and define
\[
F_{a,s}(r)
:=
\mathbb{P}(S_{a,s}\le r)
=
\sum_{j=1}^K q_j
\Phi\left(\frac{r-a x_j}{s}\right).
\]
Set
\[
w_0:=-\infty,
\qquad
w_K:=+\infty,
\qquad
w_k:=F_{a,s}^{-1}(Q_k),
\quad k=1,\ldots,K-1.
\]
The comonotone optimizer assigns the statistic $S_{a,s}$ to the atom
$x_k$ precisely on the quantile bin $(w_{k-1},w_k]$.  Hence the optimal
initial kernel is
\begin{equation}
R_y^{*,a,s}(\{x_k\})
=
\Phi\left(\frac{w_k-a y}{s}\right)
-
\Phi\left(\frac{w_{k-1}-a y}{s}\right),
\qquad k=1,\ldots,K.
\end{equation}
The marginal constraint follows immediately:
\[
\sum_{j=1}^K q_j R_{x_j}^{*,a,s}(\{x_k\})
=
\mathbb{P}\bigl(S_{a,s}\in(w_{k-1},w_k]\bigr)
=
q_k.
\]

For the terminal cost of Example~\ref{ex:constant}, the relevant
statistic is $Y_T$, so $a=1$ and $s=\sqrt T$.  Thus
\[
V_c^{\rm term}(\mu,\nu)
=
\mathcal M\bigl(\mu_0,\mu_0*N(0,T)\bigr),
\]
and the displayed formula for $R_y^{*,a,s}$ gives the corresponding
optimal initial kernel of $X_0$ given $Y_0=y$.
\end{example}

\begin{example}[$\nu=\mathbb W$, $f(x,y)=xy$, $g(x,y)=0$, $X_t=X_0$ constant]\label{ex:constant3}
We now replace the terminal cost by a running cost.  Keep
$X_t\equiv X_0$, but take
\[
f_0(i,y)=x_i y,
\qquad
g_0\equiv0.
\]
For $\tau:=T-t$,
\[
\int_t^T Y_s\,ds
=
\tau y+\int_t^T (T-s)\,dY_s
\sim
N\left(\tau y,\frac{\tau^3}{3}\right)
\quad\text{conditionally on }Y_t=y.
\]
Since $X$ is constant,
\[
\int_t^T X_sY_s\,ds
=
X_0\int_t^T Y_s\,ds.
\]
The same considerations as in Example~\ref{ex:constant} give the optimal causal value
\begin{equation}
V^{\rm run}(t,y,p)
=
\mathcal M\left(\rho_p,N\left(\tau y,\frac{\tau^3}{3}\right)\right)
=
\tau y\,m(p)+\frac{\tau^{3/2}}{\sqrt3}L(p),
\end{equation}
where $L(p)=\mathcal M(\rho_p,N(0,1))$.

The PDE verification is a direct reuse of the static Hamiltonian identity
from Example~\ref{ex:constant}.  Here
\[
D_yD_pV^{\rm run}\cdot z
=
\tau\sum_{i=1}^K x_i z_i,
\qquad
D^2_{pp}V^{\rm run}
=
\frac{\tau^{3/2}}{\sqrt3}D^2_{pp}L.
\]
Thus the Hamiltonian term equals the identity of Example~\ref{ex:constant}
with $\alpha=\tau$ and $\beta=\tau^{3/2}/\sqrt3$, namely
\[
\frac{\alpha^2}{2\beta}L(p)
=
\frac{\sqrt3}{2}\sqrt{\tau}\,L(p).
\]
On the other hand,
\[
-\partial_tV^{\rm run}(t,y,p)-f(p,y)
=
\frac{\sqrt3}{2}\sqrt{\tau}\,L(p),
\qquad
A_\nu V^{\rm run}=0.
\]
Hence \eqref{correctPDE} holds, with terminal condition
$V^{\rm run}(T,y,p)=0$.  By Theorem~\ref{verification}, the displayed
formula is the causal transport value.

For a non-degenerate initial law, the relevant scalar statistic is not
$Y_T$ but the time integral of $Y$.  If $Y_0\sim\nu_0$ and
$Z\sim N(0,1)$ is independent of $Y_0$, then
\[
\int_0^T Y_s\,ds
\quad\text{has the same conditional law as}\quad
TY_0+\frac{T^{3/2}}{\sqrt3}Z.
\]
Therefore
\begin{equation}
V_c^{\rm run}(\mu,\nu)
=
\mathcal M\left(
\mu_0,
\operatorname{Law}\left(TY_0+\frac{T^{3/2}}{\sqrt3}Z\right)
\right).
\end{equation}
In the special case
$\nu_0=\mu_0=\sum_j q_j\delta_{x_j}$, the optimal initial kernel is the
one in Example~\ref{ex:constant2} with
\[
a=T,
\qquad
s=\frac{T^{3/2}}{\sqrt3}.
\]
Thus the terminal and running constant-state examples differ only in the
scalar statistic used for the final static rearrangement.
\end{example}

We now leave the constant-state case.  The next example is still outside
the irreducible setting, since one of the two states is absorbing.  We
derive its dual/free-boundary representation directly from the
terminal-filter constraints \eqref{eq:Cone}--\eqref{eq:Ctwo}.  The
resulting candidate is then justified by Theorem~\ref{verification}.

\begin{example}[$\nu=\mathbb W$, $f\equiv0$, $g(x,y)=xy$, $X_t$ two-state with absorption]\label{ex:absorption}
Let $X$ take values in $S=\{1,2\}$, where state $1$ is non-absorbing and
state $2$ is absorbing, with generator and transition matrix
\[
\Lambda=
\begin{pmatrix}
-a & a\\
0 & 0
\end{pmatrix},
\qquad
a>0,
\qquad
P(t)=e^{t\Lambda}
=
\begin{pmatrix}
e^{-at} & 1-e^{-at}\\
0 & 1
\end{pmatrix},
\qquad t\in\R.
\]
State 1 is identified with $+1$ and state 2 with $0$.
We take $f_0\equiv0$, $g_0(1,y)=y$, and $g_0(2,y)=0$.  Fix
$(t,y,p)\in[0,T)\times\R\times\Sigma_2$ and write $p=p^1\in[0,1]$ for
the first coordinate.

\smallskip
\nd
\emph{Kolmogorov backward equations and auxiliary value.}
Since $A_\nu=\frac{1}{2}\partial^2_{yy}$ and the terminal data are harmonic (see \eqref{eq:bkwd}),
\[
G_0(1,t,y)=y,
\qquad
G_0(2,t,y)=0.
\]
Thus $G_0(t,y)=(y,0)$, $\partial_yG_0(1,\cdot,\cdot)=1$, and
$\partial_yG_0(2,\cdot,\cdot)=0$.

\smallskip
\nd
\emph{Direct dual reduction.}
Although Theorem~4.4 is stated under irreducibility, the Lagrange-duality
calculation leading to the dual representation can be repeated directly
here.  The constraints \eqref{eq:Cone}--\eqref{eq:Ctwo} simplify
considerably.  The $j=1$ inequality in \eqref{eq:Cone} reads
\[
\E^\nu[p_T^1\mid\cF^Y_s]e^{a(T-s)}\ge0,
\]
which is automatic.  Hence the corresponding process multiplier may be
set to zero.  The $j=2$ equality in \eqref{eq:Ctwo} is redundant, because
it follows from the $j=1$ equality and $p_T^1+p_T^2=1$.

Keeping only the remaining static multiplier $\lambda:=\lambda^1\in\R$
and the remaining non-negative process multiplier
$\lambda_\cdot:=\lambda^2_\cdot\ge0$, and using
\[
P_{1,2}(s-T)=1-e^{a(T-s)},
\qquad
P_{2,2}(s-T)=1,
\]
we obtain
\be\label{eq:v-inf}
V(t,y,p)
=
\inf_{\lambda\in\R,\,\lambda_\cdot\ge0}
\E^\nu\left[
\left(
Y_T-e^{a(T-t)}\lambda
-\int_t^T e^{a(T-s)}\lambda_s\,ds
\right)_+
+p\lambda+\int_t^T\lambda_s\,ds
\right].
\ee
The positive part arises from the inner maximum over $i\in\{1,2\}$:
after factoring out the common term $\int_t^T\lambda_s\,ds$, the two
alternatives are the displayed expression and zero.  The corresponding
optimal terminal filter is
\[
p_T^*
=
\mathbf 1_{\left\{
Y_T\ge e^{a(T-t)}\lambda
+\int_t^T e^{a(T-s)}\lambda_s\,ds
\right\}}.
\]

\smallskip
\nd
\emph{PDE reformulation.}
Introduce the auxiliary state
\[
\xi_s
:=
x+\int_t^s e^{a(T-r)}\lambda_r\,dr,
\qquad
\xi_t=x,
\]
and define
\[
u(t,x,y)
:=
\inf_{\lambda_\cdot\ge0}
\E^\nu\left[
(Y_T-\xi_T)_+
+\int_t^T\lambda_s\,ds
\right].
\]
Standard Bellman analysis gives
\be\label{PDEex2}
\partial_tu+\frac{1}{2}\partial^2_{yy}u
+
\inf_{\lambda\ge0}
\left\{
\lambda\left(1+e^{a(T-t)}\partial_xu\right)
\right\}
=0,
\qquad
u(T,x,y)=(y-x)_+.
\ee
By construction, \eqref{eq:v-inf} becomes
\be\label{eq:v-conj}
V(t,y,p)
=
\inf_{\lambda\in\R}
\left\{
u\bigl(t,e^{a(T-t)}\lambda,y\bigr)
+p\lambda
\right\}.
\ee

\smallskip
\nd
\emph{Reduction to one space variable and free boundary.}
Since the terminal condition depends only on $z:=y-x$, so does $u$.
Writing $u(t,x,y)=w(t,y-x)$, we obtain
\be\label{eq:wPDE}
\partial_tw+\frac{1}{2}\partial^2_{zz}w
+
\inf_{\lambda\ge0}
\left\{
\lambda\left(1-e^{a(T-t)}\partial_zw\right)
\right\}
=0,
\qquad
w(T,z)=z_+.
\ee
The Hamiltonian is finite exactly on the gradient-constrained region
\[
\partial_zw(t,z)\le e^{-a(T-t)}.
\]
This gives a free boundary $b:[0,T]\to\R$ determined by the smooth-pasting
condition
\[
\partial_zw(t,b(t))=e^{-a(T-t)},
\qquad
b(T)=0.
\]
For $z<b(t)$, $w$ solves the heat equation
\[
\partial_tw+\frac{1}{2}\partial^2_{zz}w=0,
\]
whereas for $z\ge b(t)$, instantaneous intervention gives the linear
extension
\[
w(t,z)=w(t,b(t))+(z-b(t))e^{-a(T-t)}.
\]
This is a one-sided monotone follower problem; see
Karatzas~\cite{Kar81} and Karatzas--Shreve~\cite{KS84} for the classical
monotone-follower formulation, and
Bene\v{s}, Shepp, and Witsenhausen~\cite{BSW80} and
Karatzas--Shreve~\cite{KS85} for related singular-control and
bounded-variation-follower problems.  By the classical regularity theory
for parabolic free-boundary problems with monotone oblique-derivative
condition \cite{kinderlehrer1978smoothness}, $b\in C^\infty(0,T)$ and
$w(t,\cdot)\in C^\infty(\R)$ for every $t<T$.  Consequently
\eqref{eq:v-conj} defines a smooth candidate on
$[0,T)\times\R\times\Sigma_2$.

Finally, substituting this candidate into \eqref{correctPDE} gives the
same HJB/free-boundary system as \eqref{PDEex2}--\eqref{eq:wPDE}, and
the free-boundary constraint is precisely the condition ensuring that the
maximizing filter remains in $\Sigma_2$.  
Therefore the hypotheses of
Theorem~\ref{verification} are satisfied, and the candidate
\eqref{eq:v-conj} is the causal transport value in the absorbing
two-state model.
\end{example}

\subsection{Example in the irreducible case}
\begin{example}[$\nu=\mathbb{W}$, $f\equiv 0$, $g(x,y)=xy$, $X$ a two-state irreducible chain]\label{ex:twostates}
Let $X$ take values in $\mathbb{S}=\{1,2\}$, where state~$1$ is
identified with $+1$ and state~$2$ with $-1$, with irreducible
generator
\[
\Lambda=\begin{pmatrix}-a & a\\b & -b\end{pmatrix},\qquad
c:=a+b,\qquad a,b>0.
\]
The transition matrix at time $s\in\R$ is
\[
P(s)=e^{s\Lambda}=\frac{1}{c}\begin{pmatrix}
b+ae^{-cs} & a(1-e^{-cs})\\
b(1-e^{-cs}) & a+be^{-cs}
\end{pmatrix}.
\]
We take $\nu=\mathbb{W}$ ($b_\nu=0$, $\sigma_\nu=1$), costs
$f_0\equiv 0$, $g_0(1,y)=y$, $g_0(2,y)=-y$, so $g(p,y)=y(2p^1-1)$.
Fix $(t,y,p)\in[0,T]\times\R\times\Sigma_2$; abusing notation, we
write $p=p^1\in[0,1]$ for the first coordinate.

\smallskip
\nd
\emph{Kolmogorov backward equations and auxiliary value.}\
With $\cA_\nu=\tfrac12\partial^2_{yy}$ and terminal data linear in
$y$, the solution of \eqref{eq:bkwd} is
\[
G_0(1,t,y)=y,\qquad G_0(2,t,y)=-y,
\]
so $G_0(t,y)=(y,-y)$ and $\partial_y G_0(1,\cdot,\cdot)=1$,
$\partial_y G_0(2,\cdot,\cdot)=-1$.

\smallskip
\nd
\emph{Specializing the constraints \eqref{eq:Cone}--\eqref{eq:Ctwo}.}\
Let 
\[
M_s:=\E^\nu\!\bigl[p^1_T\,\big|\,\cF^Y_s\bigr],\qquad s\in[t,T].
\]
By the martingale identity
$\E^\nu[ p_T\mid\cF^Y_s]= p_s\,P(T-s)$ from \eqref{eq:tilde_p}, we have
\[
M_s=\ell(s)+e^{-c(T-s)}\,p_s,\qquad
\ell(s):=\frac{b\bigl(1-e^{-c(T-s)}\bigr)}{c},\quad
u(s):=\frac{b+a\,e^{-c(T-s)}}{c},
\]
with $u(s)-\ell(s)=e^{-c(T-s)}$. A direct computation shows that the
inequalities \eqref{eq:Cone} for $j=1,2$ are equivalent to
\[
\ell(s)\le M_s\le u(s),
\]
respectively, and that \eqref{eq:Ctwo} for $j=1$ fixes the initial
value
\[
M_t=\frac{b}{c}+e^{-c(T-t)}\!\left(p-\frac{b}{c}\right).
\]
The equality \eqref{eq:Ctwo} for $j=2$ follows from the $j=1$
equality together with $p^1_T+p^2_T=1$ and is therefore redundant;
as in Example~\ref{ex:absorption}, the corresponding static
multiplier $\lambda^2$ may be set to $0$. In contrast with the
absorbing case, however, \emph{both} process multipliers
$\lambda^1_\cdot,\lambda^2_\cdot$ are active here, since the filter
is constrained from both above and below.

\smallskip
\nd
\emph{Primal form.}\
Since $X_T\in\{+1,-1\}$ and $\E^\pi[X_T\mid\cF^Y_T]=2p^1_T-1$,
\be\label{eq:primal-2state}
V(t,y,p)=2\sup_{M}\E^\nu\!\bigl[Y_T M_T\bigr]-y,
\ee
where the supremum is over $\cF^Y$-martingales $(M_s)_{s\in[t,T]}$
with prescribed initial value $M_t$ as above and satisfying
$\ell(s)\le M_s\le u(s)$ a.s.\ for all $s\in[t,T]$.

\smallskip
\nd
\emph{Dualization.}\
Following the strategy of Example~\ref{ex:absorption}, we
specialize Theorem~\ref{thm:dual}\,(i). With $\lambda^2=0$,
$\lambda:=\lambda^1\in\R$, and using that $M$ is a martingale (so
$\E^\nu[M_T]=M_t$), the dual representation
\eqref{eq:final_representation} takes the form
\be\label{eq:v-dual-2state}
V(t,y,p)=\inf_{\lambda\in\R}\Bigl\{w(t,\lambda)+2M_t(y-\lambda)\Bigr\}-y,
\ee
where the dual value is the bounded-variation-follower value
\[
w(t,z):=2\inf_{\xi^+,\xi^-\uparrow}
\E^\nu\!\biggl[(Z_T)_+
+\int_t^T e^{-c(T-s)}\bigl(d\xi^+_s+d\xi^-_s\bigr)\biggr],
\qquad Z_t=z,
\]
the infimum being over non-decreasing $\cF^Y$-adapted processes
$\xi^+,\xi^-$ with $\xi^\pm_t=0$, and
\[
Z_s=Y_s-\xi^+_s\,e^{a(T-s)}+\xi^-_s\,e^{b(T-s)}.
\]
Here $d\xi^+_s,d\xi^-_s$ play the role of the process multipliers
$\lambda^1_s\,ds,\lambda^2_s\,ds$ in \eqref{eq:saddle}, enforcing the
lower and upper bounds $M_s\ge\ell(s)$ and $M_s\le u(s)$
respectively.

\smallskip
\nd
\emph{Consistency check at $t=T$.}\
Since $\xi^\pm_T=0$, $M_T=p$, and the running integral vanishes, we
have $w(T,z)=2(z)_+$, hence
\[
V(T,y,p)
=\inf_{\lambda\in\R}\bigl\{2(\lambda)_++2p(y-\lambda)\bigr\}-y
=2py-y=y(2p-1)=g(p,y),
\]
the minimiser being $\lambda=0$.

\smallskip
\nd
\emph{Reduced PDE.}\
The function $\tilde w(t,z):=w(t,z)/2$ satisfies $\tilde w(T,z)=(z)_+$
and, formally, the one-dimensional parabolic variational inequality
with two oblique singular controls
\begin{align}
\partial_t\tilde w+\tfrac12\partial^2_{zz}\tilde w
&+\inf_{\lambda^+\ge 0}\lambda^+\!\Bigl(1-e^{2a(T-t)}\partial_z\tilde w-e^{-c(T-t)}\Bigr)\notag\\
&+\inf_{\lambda^-\ge 0}\lambda^-\!\Bigl(1+e^{2b(T-t)}\partial_z\tilde w-e^{-c(T-t)}\Bigr)
=0,
\label{eq:1D-dual-asymmetric}
\end{align}
or, equivalently, in complementarity form
\begin{align*}
\min\!\Bigl\{
\partial_t\tilde w+\tfrac12\partial^2_{zz}\tilde w,\;\;
1-e^{2a(T-t)}\partial_z\tilde w-e^{-c(T-t)},\;\;
1+e^{2b(T-t)}\partial_z\tilde w-e^{-c(T-t)}
\Bigr\}=0.
\end{align*}


This is a parabolic bounded-variation-follower problem; see, for instance,
Karatzas and Shreve~\cite{KS84,KS85} and the survey of Bene\v{s}, Shepp, and
Witsenhausen~\cite{BSW80} for classical singular stochastic control and
bounded-variation follower problems. In the absorbing case of Example~\ref{ex:absorption},
one of the two controls is inactive and the problem reduces to a one-sided
monotone follower problem in the sense of
Karatzas~\cite{Kar81} and Karatzas--Shreve~\cite{KS84}.

\end{example}
\subsection{Numerical Examples}
\begin{example}[A numerical illustration of Theorems~\ref{Nmainpdes} and~\ref{thm:dual}]\label{ex:numerics}
We consider the simplest non-trivial two-state case. Take $K=2$, $d=1$,
$T=1$, and
\[
    dY_t=dW_t,\qquad Y_0=0,
\]
so that $b_\nu\equiv 0$ and $\sigma_\nu\equiv 1$. Let $X$ have generator
\[
    \Lambda=
    \begin{pmatrix}
        -1/2 & 1/2\\[2pt]
        1/2 & -1/2
    \end{pmatrix},
\qquad
P(s)=e^{s\Lambda}=
\begin{pmatrix}
\tfrac12(1+e^{-s}) & \tfrac12(1-e^{-s})\\[2pt]
\tfrac12(1-e^{-s}) & \tfrac12(1+e^{-s})
\end{pmatrix},
\]
so $c:=a+b=1$, and choose the initial filter
\[
    p_0=(1/2,1/2),
\]
which is the (unique) stationary distribution of $X$. We set $f_0\equiv 0$
and choose the bounded terminal payoff
\[
    g_0(1,y)=\rho(y):=\frac{1}{1+e^{-8y}},\qquad g_0(2,y)=0,
\]
so $\rho$ is a smoothed indicator of $\{y>0\}$ with transition zone
$|y|\lesssim 0.2$. The objective is then to correlate the terminal filter
mass on state $1$ with the value of $\rho(Y_T)$.

\medskip
\noindent\textit{Discretisation.} We approximate $Y$ by a
binomial tree with $M=12$ steps,
\[
    \Delta t=\tfrac{1}{12},\qquad
    \Delta Y=\pm\sqrt{\Delta t}\ \text{each with $\nu$-probability $1/2$.}
\]
We write $\cF^Y_n:=\sigma(Y_0,\ldots,Y_n)$ for the tree filtration and
let $q_n:=p^1_n$ denote the first component of the (discrete-time) filter,
so $\tilde p_n=(q_n,1-q_n)$.

The tree analogue of the Wonham conditional-expectation identity
$\E^\nu[\tilde p_{n+1}\mid\cF^Y_n]=\tilde p_n\,P(\Delta t)$ reads
\be\label{eq:tree_drift}
    \tfrac12\bigl(q_{n+1}^{+}+q_{n+1}^{-}\bigr)
    =\tfrac12+\bigl(q_n-\tfrac12\bigr)e^{-\Delta t},
\ee
where $q_{n+1}^\pm$ are the two values of $q_{n+1}$ in the $\pm$ branches
out of node $n$. The state constraint $\tilde p\in\Sigma_2$ becomes
\be\label{eq:tree_state}
0\le q_{n+1}^\pm\le 1,
\ee
and the truncation $|Z^1|\le N$ from Theorem~\ref{Nmainpdes} (which,
since $\sum_i Z^i=0$ forces $Z^2=-Z^1$ when $K=2$, automatically
implements $|Z^2|\le N$ as well) becomes
\be\label{eq:tree_trunc}
    |q_{n+1}^{+}-q_{n+1}^{-}|
    \le 2N\sqrt{\Delta t}.
\ee

\medskip
\noindent\textit{Lower approximation.} The truncated state-constrained
problem of Theorem~\ref{Nmainpdes} discretises to the finite-dimensional
linear program
\[
    V^{N,M}
    :=
    \sup_{q}\ \E\bigl[q_M\,\rho(Y_M)\bigr],
\]
where the supremum is taken over $\cF^Y$-adapted processes
$(q_n)_{n=0}^M$ on the tree satisfying $q_0=1/2$ together with
\eqref{eq:tree_drift}--\eqref{eq:tree_trunc} at every node.

\medskip
\noindent\textit{Upper approximation.} For the dual approximation of
Theorem~\ref{thm:dual}, we use the saddle form \eqref{eq:saddle}.
Its tree counterpart is
\[
\begin{aligned}
    v^{N,M}
    :=
    \inf_{\lambda,(\lambda^j_n)}\,
    \E\Bigg[
        \max_{i\in\{1,2\}}
        \bigg\{
            g_0(i,Y_M)
            &-\sum_{j=1}^2 P_{i,j}(-1)\,\lambda^j\\
            &+\Delta t\sum_{n=0}^{M-1}\sum_{j=1}^2
                P_{i,j}(t_n-1)\,\lambda^j_n
        \bigg\}
    \Bigg]
    +p_0\cdot\lambda,
\end{aligned}
\]
where $\lambda\in\R^2$, the processes $(\lambda^j_n)$ are $\cF^Y$-adapted
on the tree with $0\le\lambda^j_n\le N$, and
$p_0\cdot\lambda:=\sum_j p_0^j\lambda^j$. After introducing one auxiliary
variable per leaf to linearise the inner $\max_i$, this is again a
finite-dimensional linear program.

\medskip
\noindent\textit{Results.} Let $V^M$ denote the value of the
fully discretised tree problem, i.e.\ the common limit
$\lim_{N\to\infty}V^{N,M}=\lim_{N\to\infty}v^{N,M}$. Solving the two LPs
gives:
\begin{table}[h]
\centering
\begin{tabular}{c|cc}
\toprule
$N$ & $V^{N,M}$ & $v^{N,M}$ \\
\midrule
$0$      & $0.250000$ & $0.441703$\\
$0.25$   & $0.313991$ & $0.436859$\\
$0.5$    & $0.375733$ & $0.432210$\\
$1$      & $0.416128$ & $0.426318$\\
$2$      & $0.417777$ & $0.422032$\\
$5$      & $0.417777$ & $0.418043$\\
$10$     & $0.417777$ & $0.417777$\\
\bottomrule
\end{tabular}
\caption{Tree values $V^{N,M}$ and $v^{N,M}$ as functions of the truncation
level $N$, with $M=12$ time steps. The two sequences sandwich the limiting
tree value $V^M\approx 0.417777$.}
\label{tab:numerics}
\end{table}

\noindent
The table illustrates the
two monotone approximations
\[
    V^{N,M}\uparrow V^M,
    \qquad
    v^{N,M}\downarrow V^M,
    \qquad
    V^M\approx 0.417777,
\]
so the bounded-control state-constrained problems approximate the value
from below, as in Theorem~\ref{Nmainpdes}, while the truncated dual
problems approximate the same value from above, as in
Theorem~\ref{thm:dual}. The asymmetric speeds are informative: the
primal $V^{N,M}$ stabilises to six digits already at $N=2$, suggesting
that the optimiser $Z^*$ is bounded by approximately $2$ and the
truncation \eqref{eq:tree_trunc} becomes inactive; the dual $v^{N,M}$
requires $N\gtrsim 10$, reflecting the fact that the singular dual
multipliers $\lambda^j_n$ can be much larger than the primal control.
\end{example}
\section{Viscosity properties of the value functions and proofs of main results}\label{sec:main-proofs}\label{mainproofs}

\subsection{Dimension reduction and notation}
Since $\Sigma_K\subset\R^K$ has empty interior in $\R^K$, we
transfer the problem to the standard $(K-1)$-simplex
$$
\Sigma:=\Bigl\{(p_1,\ldots,p_{K-1})\in[0,1]^{K-1}:
  \sum_{i=1}^{K-1}p_i\le 1\Bigr\},
$$
which has non-empty interior in $\R^{K-1}$. We adopt the following
notation throughout this section.

\begin{notation}\label{restrictions}
For $p=(p_1,\ldots,p_K)\in\R^K$, we write
$\tilde p:=(p_1,\ldots,p_{K-1})\in\R^{K-1}$. Conversely, for
$p=(p_1,\ldots,p_{K-1})\in\Sigma$, we write
$p':=\bigl(p_1,\ldots,p_{K-1},\,1-\sum_{i=1}^{K-1}p_i\bigr)\in\Sigma_K$.
For a function $u\colon\Sigma_K\to\R$, its restriction
$\tilde u\colon\Sigma\to\R$ is defined by
$$
\tilde u(p_1,\ldots,p_{K-1})
  :=u\bigl(p_1,\ldots,p_{K-1},\,1-\sum_{i=1}^{K-1}p_i\bigr).
$$
\end{notation}
\vspace{2mm}

\nd
Accordingly, for any
$V\colon[0,T]\times\R^d\times\Sigma_K\to\R$ we define
$\tilde V\colon[0,T]\times\R^d\times\Sigma\to\R$ by
\be\label{res}
\tilde V(t,y,p_1,\ldots,p_{K-1})
  := V\!\Bigl(t,y,\,p_1,\ldots,p_{K-1},\,1-\sum_{i=1}^{K-1}p_i\Bigr).
\ee
The first derivatives satisfy
$$
\partial_{p_i}\tilde V
  =\partial_{p_i}V-\partial_{p_K}V,
\qquad i=1,\ldots,K-1,
$$
where $p_K=1-\sum_{i=1}^{K-1}p_i$. The second derivatives satisfy
$$
\partial^2_{p_ip_j}\tilde V
  =\partial^2_{p_ip_j}V-\partial^2_{p_ip_K}V-\partial^2_{p_jp_K}V
  +\partial^2_{p_Kp_K}V,
\qquad i,j=1,\ldots,K-1.
$$

\nd
Any control $z\in\R^{K\times d}$ subject to $\sum_{i=1}^K z^i=0$ in
$\R^d$ can be written as $z=J\tilde z$, where
$\tilde z\in\R^{(K-1)\times d}$ collects the first $K-1$ rows of $z$
and
$$
J = \begin{pmatrix}
 1 & 0 & 0 & \cdots & 0 \\
 0 & 1 & 0 & \cdots & 0 \\
 0 & 0 & 1 & \cdots & 0 \\
 \vdots & & & \ddots & \vdots \\
 0 & 0 & 0 & \cdots & 1 \\
 -1 & -1 & -1 & \cdots & -1
\end{pmatrix}
\in\R^{K\times(K-1)}.
$$
The map $\tilde z\mapsto J\tilde z$ is a linear bijection from
$\R^{(K-1)}$ onto $\bigl\{z\in\R^{K}:\sum_{i=1}^K
z_i=0\bigr\}$, with inverse $z\mapsto(z^1,\ldots,z^{K-1})$.
In matrix form, the identities above read
$$
D_p\tilde V= J^\top D_pV,
\qquad
D_pD_y\tilde V= J^\top D_pD_yV,
\qquad
D^2_{pp}\tilde V = J^\top D^2_{pp}V\,J,
$$
and consequently, for any $z=J\tilde z$,
$$
z^\top D^2_{pp}V\,z = \tilde z^\top D^2_{pp}\tilde V\,\tilde z,
\qquad
\tr\!\bigl(z\,m^\top\bigr) = \tr\!\bigl(\tilde z\,\tilde m^\top\bigr)
  \text{ with }\tilde m:=J^\top m.
$$
The map $\tilde z\mapsto J\tilde z$ is a linear bijection from
$\R^{K-1}$ onto the subspace $\{z\in\R^K:\sum_i z_i=0\}$, but it is
not an isometry. Thus the numerical constrained maximum eigenvalue of
$D^2_{pp}V$ on this subspace need not equal
 $\lambda_{\max}(D^2_{pp}\tilde V)$ computed with the Euclidean norm in
$\R^{K-1}$. What is invariant, and what is needed for the obstacle
term, is the sign condition:
\[
z^\top D^2_{pp}V\,z\le 0\quad\text{for all }\sum_i z_i=0
\quad\Longleftrightarrow\quad
\tilde z^\top D^2_{pp}\tilde V\,\tilde z\le 0
\quad\text{for all }\tilde z\in\R^{K-1}.
\]
Consequently the reduced obstacle may be written using $\lambda_{\max}(D^2_{pp}\tilde V)$, since only the sign of this
quantity enters \eqref{Hvi}.

\vspace{2mm}
\nd
We now define the joint gradient and Hessian of $\tilde V$ in the
reduced variable $(p,y)\in\Sigma\times\R^d$, in parallel with
\eqref{jointGH}:
\be\label{jointGHtilde}
D_{(p,y)}\tilde V
  := \begin{pmatrix}D_p\tilde V\\D_y\tilde V\end{pmatrix}
     \in\R^{K-1+d},
\qquad
D^2_{(p,y)}\tilde V
  := \begin{pmatrix}
       D^2_{pp}\tilde V & D_pD_y\tilde V \\
       (D_pD_y\tilde V)^\top & D^2_{yy}\tilde V
     \end{pmatrix}
     \in\cS_{K-1+d},
\ee
where $D_p\tilde V\in\R^{K-1}$, $D_y\tilde V\in\R^d$,
$D^2_{pp}\tilde V\in\cS_{K-1}$, $D_pD_y\tilde V\in\R^{(K-1)\times d}$,
and $D^2_{yy}\tilde V\in\cS_d$.

\nd
The joint drift in reduced coordinates is
\be\label{jointdrifttilde}
\beta(t,y,p)
  := \bigl(\widetilde{p'\Lambda},\,b_\nu(t,y)\bigr)\in\R^{K-1+d},
\ee
where $\widetilde{p'\Lambda}\in\R^{K-1}$ is the unique vector
satisfying $J\,\widetilde{p'\Lambda}=p'\Lambda$ (which is well-defined
because $\sum_i(p'\Lambda)_i=0$, and $J$ is bijective onto this set). For
$\tilde z\in\R^{(K-1)\times d}$, the joint quadratic-variation matrix
in reduced coordinates is
\be\label{Qtildered}
\widetilde Q(t,y,\tilde z)
  := \begin{pmatrix}
       \tilde z\,\sigma_\nu(t,y)^{-2}\,\tilde z^\top & \tilde z \\
       \tilde z^\top & \sigma_\nu^2(t,y)
     \end{pmatrix}
     \in\cS_{K-1+d},
\ee
the $(K-1+d)\times(K-1+d)$ analogue of \eqref{Qtilde} after the
simplex reduction.

\nd
Define $\tilde f,\tilde g\colon\Sigma\times\R^d\to\R$ by
\be\label{tildefg}
\tilde f(p,y)
  := f_0(K,y) + \sum_{i=1}^{K-1}\bigl(f_0(i,y)-f_0(K,y)\bigr)p_i,
\qquad
\tilde g(p,y)
  := g_0(K,y) + \sum_{i=1}^{K-1}\bigl(g_0(i,y)-g_0(K,y)\bigr)p_i,
\ee
i.e., $\tilde f(p,y)=f(p',y)$ and $\tilde g(p,y)=g(p',y)$ in the
notation of \eqref{costs}. Whenever a function is initially defined on
$[0,T)$ only, the notation $V(T^-,\cdot,\cdot)$ below denotes a
specified terminal trace, when such a trace exists; otherwise the
terminal viscosity inequalities are to be understood through the
corresponding upper and lower half-relaxed terminal traces. The
identities collected above allow us to rewrite \eqref{correctPDE} as
the reduced problem
\be\label{correctPDE2}
\begin{cases}
    \min\!\Bigl\{-\partial_t\tilde V
    -\beta(t,y,p)\cdot D_{(p,y)}\tilde V
    -\tilde f(p,y)\\[2pt]
    \hspace{1.5cm}-\!\displaystyle\sup_{\tilde z\in\R^{(K-1)\times d}}\!
       \tfrac{1}{2}\tr\!\bigl(\widetilde Q(t,y,\tilde z)\,
         D^2_{(p,y)}\tilde V\bigr),\;
       -\lambda_{\max}\!\bigl(D^2_{pp}\tilde V\bigr)\Bigr\}=0,&\text{in }[0,T)\times\R^d\times\Sigma,\\[8pt]
-\partial_t \tilde{V}-\mathcal{A}_\nu\tilde{V}-\widetilde{p'\Lambda}\cdot D_p\tilde{V}-\tilde{f}(p,y)\ge0,&\text{in }[0,T)\times \R^d\times \pa\Sigma.

\end{cases}
\ee
As in \eqref{correctPDE}, the equality on the first line is to be understood in the
state-constrained sense: $\ge 0$ on the corresponding interior
$[0,T)\times\R^d\times\Sigma^\circ$ and $\le 0$ on the full set
$[0,T)\times\R^d\times\Sigma$.

Define the operators
$F,\; \widetilde{\mathcal{B}_0}\colon[0,T]\times\R^d\times\Sigma\times\R\times\R^{K-1+d}\times
\cS_{K-1+d}\to\R$, $G\colon\Sigma\times\R^d\times\R\to\R$ by
\begin{align}
F(t,y,p,r,\cG,\cH)
  &:= -r - \beta(t,y,p)\cdot\cG
     - \sup_{\tilde z\in\R^{(K-1)\times d}}
       \tfrac{1}{2}\tr\!\bigl(\widetilde Q(t,y,\tilde z)\,\cH\bigr)
     - \tilde f(p,y),
     \label{F}\\
\widetilde{\mathcal{B}_0}(t,y,p,r,\cG,\cH) &:= -r-\beta(t,y,p)\cdot \mathcal{G}-\frac{1}{2}\tr\left(\widetilde{Q}(t,y,0)\cH \right)-\tilde{f}(p,y).\label{B}
\end{align}
Viscosity solutions of \eqref{correctPDE2} are defined in direct
analogy with Definition~\ref{def:stateconstrvisc}, with $\Sigma$
replacing $\Sigma_K$ and the operators $F,\widetilde{\mathcal{B}_0}$ replacing $H_{sc},H_{vi}$ and $\mathcal{B}_0$
as appropriate. We have the following lemma.

\begin{lemma}\label{viscositysol}
Let $V\colon[0,T)\times\R^d\times\Sigma_K\to\R$ be a
(semi-)continuous function and
$\tilde V\colon[0,T)\times\R^d\times\Sigma\to\R$ its restriction
defined in \eqref{res}. Then:
\begin{enumerate}
    \item[(i)] $V$ is a viscosity sub-solution of \eqref{correctPDE}
    on $[0,T)\times\R^d\times\Sigma_K$ if and only if $\tilde V$
    is a viscosity sub-solution of \eqref{correctPDE2} on
    $[0,T)\times\R^d\times\Sigma$; that is,
    \begin{align*}
        \min\bigl\{F\!\left(t_0,y_0,p_0,\,\partial_t\varphi,\,
          D_{(p,y)}\varphi,\,D^2_{(p,y)}\varphi\right),\,
          -\lambda_{\max}(D^2_{pp}\varphi)\bigr\}\le 0
    \end{align*}
    at every point $(t_0,y_0,p_0)\in[0,T)\times\R^d\times\Sigma$
    where $\tilde V-\varphi$ attains a local maximum, for any test
    function $\varphi\in C^{1,2,2}$.

    \item[(ii)] $V$ is a viscosity super-solution of
    \eqref{correctPDE} on $[0,T)\times\R^d\times\Sigma_K^\circ$ if
    and only if $\tilde V$ is a viscosity super-solution of
    \eqref{correctPDE2} on $[0,T)\times\R^d\times\Sigma^\circ$;
    that is, at every point
    $(t_0,y_0,p_0)\in[0,T)\times\R^d\times\Sigma$ where
    $\tilde V-\varphi$ attains a local minimum, for any test
    function $\varphi\in C^{1,2,2}$ we have
    \begin{align*}
        \min\bigl\{F\!\left(t_0,y_0,p_0,\,\partial_t\varphi,\,
          D_{(p,y)}\varphi,\,D^2_{(p,y)}\varphi\right),\,
          -\lambda_{\max}(D^2_{pp}\varphi)\bigr\}&\ge 0,\;\;\text{if }p_0\in \Sigma^\circ,\\
           \widetilde{\mathcal{B}_0}\!\left(t_0,y_0,p_0,\,\partial_t\varphi,\,
          D_{(p,y)}\varphi,\,D^2_{(p,y)}\varphi\right)&\ge 0,\;\;\text{if }p_0\in \partial\Sigma.
    \end{align*}
\end{enumerate}
In particular, $V$ is a viscosity solution of \eqref{correctPDE} in
the sense of Definition~\ref{def:stateconstrvisc} if and only if
$\tilde V$ is a viscosity solution of \eqref{correctPDE2} in the
sense described above.
\end{lemma}

\begin{proof}
All three equivalences follow from the chain-rule identities
established above. Indeed, the relations $D_p\tilde V=J^\top D_pV$,
$D_pD_yV\cdot z = D_pD_y\tilde V\cdot\tilde z$ (for $z=J\tilde z$),
and $z^\top D^2_{pp}V\,z = \tilde z^\top D^2_{pp}\tilde V\,\tilde z$
show that, at any test-function touch point, the arguments of $H_{vi}$
in \eqref{correctPDE} and the arguments of $F$ in \eqref{correctPDE2}
are related by the same identities. The domain correspondence
$\Sigma_K^\circ\leftrightarrow\Sigma^\circ$ (resp.\
$\Sigma_K\leftrightarrow\Sigma$) is immediate from the definition of
the restriction map \eqref{res}. Finally, the sign of the constrained
quadratic form of $D^2_{pp}V$ on $\{\sum_i z_i=0\}$ is the sign of the
quadratic form associated with $D^2_{pp}\tilde V$, by the bijection
$z=J\tilde z$. This transfers the sign condition on the obstacle term between the two formulations and completes the argument.
\end{proof}

\subsection{Proof of the comparison principle for the main PDE (\ref{correctPDE})}\label{comp}
Our proof for the comparison principle utilizes ideas from
\cite{ishii2002class} in order to handle the state constraint. The
results there do not apply directly to our setting because of the
unboundedness of the controls, so we adapt the arguments to our
context. 

\begin{lemma}[{\cite[Lemma 3.3]{ishii2002class}}]\label{1st}
    Consider $\Lambda$ from Assumption~\ref{assumption1} and recall
    Notation~\ref{restrictions}. There exists a function
    $\psi\in C^\infty(\Sigma)$ such that
    \be\label{ineq1}
    \widetilde{p'\Lambda}\cdot D\psi(p)\ge 1
    \qquad\text{for every }p\in\partial\Sigma.
    \ee
\end{lemma}
\begin{proof}
Since $X$ is irreducible, $\Lambda$ admits a unique left null vector
$\pi\in\Sigma_K$ (the stationary distribution), satisfying
$\pi\Lambda=0$ and $\pi_j>0$ for all $j\in\mathbb{S}$.

\vspace{1mm}

\nd
We claim that $\widetilde{p'\Lambda}\neq 0$ for every
$p\in\partial\Sigma$. Indeed, if $\widetilde{p'\Lambda}=0$, then since
$\sum_j(p'\Lambda)_j=p'(\Lambda\mathbf{1})=0$, every component of
$p'\Lambda$ vanishes; that is, $p'\Lambda=0$, so $p'$ is a left null
vector of $\Lambda$ and hence $p'\in\langle\pi\rangle$. The
constraint $\sum_j p'_j=1$ then forces $p'=\pi$. But $\pi_j>0$ for
all $j$ places $\pi$ in $\Sigma_K^\circ$, contradicting
$p'\in\partial\Sigma_K$.

\vspace{1mm}

\nd
We observe that the map $v(p):=\widetilde{p'\Lambda}=pA+b$, where
$b=(\Lambda_{K,1},\ldots,\Lambda_{K,K-1})$, is \emph{affine} on
$\R^{K-1}$ (writing $p_K'=1-\sum_{k<K}p_k$), so
$\dot p = \widetilde{p'\Lambda}$ generates a global flow
$(\phi_t)_{t\ge 0}$ on $\Sigma$ with unique equilibrium
$\tilde\pi:=(\pi_1,\ldots,\pi_{K-1})$. Its linearization
\[
A = \bigl(\Lambda_{jk}-\Lambda_{Kk}\bigr)_{j,k=1}^{K-1}
\]
is the matrix of $\Lambda$ (acting on rows) restricted to the
invariant subspace $\{x\in\R^K:\sum_i x_i=0\}$, which is
complementary to the $0$-eigenspace $\{\alpha\pi:\alpha\in\R\}$. Hence the
eigenvalues of $A$ are the non-zero eigenvalues of $\Lambda$, which
all have strictly negative real parts. Indeed, Gershgorin's theorem
places the spectrum of $\Lambda$ in $\{\operatorname{Re}z\le0\}$, and
irreducibility implies that the zero eigenvalue is simple while all
remaining spectral values lie in the open left half-plane (equivalently,
by the Perron--Frobenius theorem for the stochastic semigroup
$e^{t\Lambda}$, $e^{t\Lambda}\to\mathbf{1}\pi$ exponentially on the
subspace $\{\sum_i x_i=0\}$). Consequently the flow satisfies
\[
|\phi_t(p) - \tilde\pi| \leq C e^{-\gamma t}|p-\tilde\pi|,
\qquad p\in\Sigma,\; t\ge 0,
\]
for some constants $C,\gamma>0$, and can be written explicitly as
\[
\phi_t(p) = pe^{tA} + bA^{-1}(e^{tA}-I),\qquad p\in\Sigma.
\]
We may now construct $\psi$ as follows. Since $\pi_j>0$ for all
$j\in\mathbb{S}$, every $p'\in\partial\Sigma_K$ has some coordinate
$p_k'=0$ for some $k\in\mathbb{S}$, so
\[
\delta_0 := \min_{p\in\partial\Sigma}|p'-\pi|^2
  \;\ge\; \min_j\pi_j^2 \;>\; 0.
\]
We choose $\rho\in C^\infty([0,\infty);[0,1])$ with $\rho\equiv 0$
on $[0,\delta_0/2]$ and $\rho\equiv 1$ on $[\delta_0,\infty)$, and we
define
\[
h(p) := \rho\!\bigl(|p'-\pi|^2\bigr) \in C^\infty(\Sigma).
\]
By construction, $h=0$ in a neighborhood of $\tilde\pi$ and
$h\equiv 1$ on $\partial\Sigma$. By the exponential convergence,
there exists a uniform $T_0<\infty$ such that $h(\phi_t(p))=0$ for
all $t\ge T_0$ and all $p\in\Sigma$. Define
\[
\psi(p) := -\int_0^\infty h(\phi_t(p))\,dt
  = -\int_0^{T_0} h(\phi_t(p))\,dt.
\]
Since $\phi_t(p)$ is affine in $p$ for each $t$ and $h\in C^\infty$,
the integrand and all its $p$-derivatives are bounded on
$\Sigma\times[0,T_0]$. Differentiating under the integral sign yields
$\psi\in C^\infty(\Sigma)$.

\vspace{1mm}

\nd
We finally verify inequality \eqref{ineq1}. For any $p\in\Sigma$ and
$\varepsilon>0$, the semigroup property
$\phi_t(\phi_\varepsilon(p))=\phi_{t+\varepsilon}(p)$ gives
\begin{align*}
\psi(\phi_\varepsilon(p)) - \psi(p)
&= -\int_0^\infty h(\phi_t(\phi_\varepsilon(p)))\,dt
  + \int_0^\infty h(\phi_t(p))\,dt \\
&= -\int_0^\infty h(\phi_{t+\varepsilon}(p))\,dt
  + \int_0^\infty h(\phi_t(p))\,dt \\
&= \int_0^\varepsilon h(\phi_t(p))\,dt.
\end{align*}
Dividing by $\varepsilon$ and sending $\varepsilon\to 0$, the chain
rule gives
\[
D\psi(p)\cdot v(p) = h(p),
\]
where $v(p)=\widetilde{p'\Lambda}$. Hence
$\widetilde{p'\Lambda}\cdot D\psi(p) = h(p)$. On $\partial\Sigma$,
every $p$ satisfies $|p'-\pi|^2\ge\delta_0$, so
$h(p)=\rho(|p'-\pi|^2)=1$. Therefore
\[
\widetilde{p'\Lambda}\cdot D\psi(p)\,\Big|_{\partial\Sigma} = 1\ge 1,
\]
which completes the proof.
\end{proof}
\vspace{3mm}

\begin{lemma}[{\cite[Lemma 3.4]{ishii2002class}}]\label{2nd}
    Consider $\Lambda$ from Assumption~\ref{assumption1} and recall
    Notation~\ref{restrictions}. There exists
    $w\in C^{1,1}(\Sigma\times\Sigma)$ and constants $C,r>0$ such that
    $$
    \widetilde{y'\Lambda}\cdot D_yw(x,y)\le 0
    \qquad\text{for all }y\in\partial\Sigma
       \text{ and }x\in\Sigma\cap B_r(y),
    $$
    and for all $x,y\in\Sigma$,
    \begin{align*}
        |x-y|^2 \le w(x,y) &\le C|x-y|^2,\\
        \max\bigl\{|D_xw(x,y)|,\,|D_yw(x,y)|\bigr\}
            &\le C|x-y|,\\
        |D_xw(x,y)+D_yw(x,y)|
            &\le C|x-y|^2,\\
        D^2w(x,y) &\le C\begin{pmatrix}I & -I \\ -I & I\end{pmatrix}
            + C|x-y|^2\begin{pmatrix}I & 0 \\ 0 & I\end{pmatrix},
    \end{align*}
    where $D^2w$ is understood in the distributional sense.
\end{lemma}

\begin{proof}[Proof of the comparison principle]
    By Lemma \ref{viscositysol}, it suffices to show that \eqref{correctPDE2} has a comparison principle. 
    \vspace{2mm}

    \nd
    Suppose that $V_1,\;V_2\colon [0,T]\times\R^d\times \Sigma_K\to \R$ are viscosity sub- and super-solutions of \eqref{correctPDE2}, respectively, in the sense of Lemma \ref{viscositysol} such that \eqref{terminalassumption} holds. Let $\tilde{V}_1,\;\tilde{V}_2\colon [0,T]\times\R^d\times\Sigma\to \R$ be their restrictions defined in \eqref{res}. To ease notation, we still denote the restricted functions by $V_1,\; V_2$. In what follows, we also unwind the definitions of $\beta,\widetilde{Q}$ from \eqref{jointdrifttilde} and \eqref{Qtildered} to treat all terms separately.
    \vspace{1mm}

    \nd
    We aim to show that $V_1\le V_2$ everywhere. Assume, for contradiction, that 
    \be\label{contr2}
    \sup_{(t,y,p)\in [0,T]\times \R\times\Sigma}( V_1(t,y,p)-V_2(t,y,p))=\theta>0.
    \ee
    Let $\psi,w$ be the functions constructed in Lemmas \ref{1st} and \ref{2nd}, respectively. Consider the function
    \begin{align}
        \Phi(p_1,p_2,y_1,y_2,t,s):=V_1(t,&\;y_1,p_1)-V_2(s,y_2,p_2)-L(w(p_1,p_2)+|t-s|^2)-\sqrt{L}|y_1-y_2|^2\nonumber\\
        &-L'(|y_1|^2+|y_2|^2)+\gamma_1t+\gamma_2 s+\frac{\e }{2}|p_2|^2+M(\psi(p_2)-\psi(p_1)),\label{Phivisc}
    \end{align}
    where the constants $L,L',\gamma_1,\gamma_2,\e,M>0$ will be chosen later. Let $r$ be the constant from Lemma \ref{2nd}. We note that there are $\hat{p_1},\hat{p_2}\in\tilde{\Sigma}$, $\hat{y_1},\hat{y_2}\in \R^d$ and $\hat{t},\hat{s}\in [0,T]$ such that
        $$\Phi(\hat{p_1},\hat{p_2},\hat{y_1},\hat{y_2},\hat{t},\hat{s})=\sup\left\{\Phi(p_1,p_2,y,t,s)|\;  p_1,p_2\in \Sigma,\; |p_1-p_2|\le r, \;y_1,y_2\in \R^d, \;t,s\in [0,T]  \right\}$$
    and by choosing $L$ to be large enough, we may assume that $|\hat{p_1}-\hat{p_2}|<r$. In addition, by choosing $L'$ to be small enough, we may assume that 
    \be\label{positivity}
    \Phi(\hat{p_1},\hat{p_2},\hat{y_1},\hat{y_2},\hat{t},\hat{s})\ge \frac{\theta}{2}>0.
    \ee
    We now use the viscosity properties of $V_1,V_2$, depending on whether either $\hat{t}$ or $\hat{s}$ equals $T$.
    We first note that by \cite[Proposition 3.7]{crandall1992user} we have
    \begin{align}
        L|\hat{p_1}-\hat{p_2}|^2+\sqrt{L}|\hat{y_1}-\hat{y_2}|^2+L|\hat{t}-\hat{s}|^2&\xrightarrow{L\rightarrow +\infty}0,\;\;\text{and}\label{conv2'}\\
        L'(|\hat{y_1}|^2+|\hat{y_2}|^2)&\xrightarrow{L'\to 0} 0.\label{conv1'}
    \end{align}
     If $\hat{t}=T$ or $\hat{s}=T$, then \eqref{conv2'} gives $\hat{t},\hat{s}\xrightarrow{L\to +\infty}T$. Choose $\gamma_1,\gamma_2$ and $\e$ small enough so that $T(\gamma_1+\gamma_2)<\theta/8$ and $\frac{\e}{2}\sup_{p\in \Sigma}|p|^2<\theta/8$. Then \eqref{positivity}, after letting $L\to +\infty$ and keeping $L'$ fixed, gives
     $$0<\frac{\theta}{4}\le \limsup_{\substack{(t,y_1,p_1),(s,y_2,p_2)\to (T,y,p)}}\left( V_1(t,y_1,p_1)-V_2(s,y_2,p_2)\right),$$
     for some $(y,p)\in \R^d\times \Sigma$.
     This contradicts our assumption. Thus, we may assume that $\hat{t},\hat{s}\in [0,T)$. 
     \vspace{2mm}

    \nd
    If $\hat{p_2}\in\partial\Sigma$, then the super-solution property of $V_2$ gives
    \be\nonumber
    \begin{split}
        -\gamma_2+2L(\hat s-\hat t)&+\sqrt{L}\tr\!\left(\sigma_\nu^2(\hat s,\hat{y_2})\right)+b_\nu(\hat s,\hat{y_2})\cdot\bigl(2\sqrt{L}(\hat{y_2}-\hat{y_1})+2L'\hat{y_2}\bigr)-\tilde f(\hat{p_2},\hat{y_2})\\
        &+\bigl(-\e\hat{p_2}-MD_p\psi(\hat{p_2})+LD_yw(\hat{p_1},\hat{p_2})\bigr)\cdot\widetilde{\hat{p_2}'\Lambda}\ge 0.
    \end{split}
    \ee
    By the properties of $\psi,w$, this inequality yields
    $$2L(\hat s-\hat t)+\sqrt{L}\tr\!\left(\sigma_\nu^2(\hat s,\hat{y_2})\right)+b_\nu(\hat s,\hat{y_2})\cdot2\sqrt{L}(\hat{y_2}-\hat{y_1})+2L'b_\nu(\hat s,\hat{y_2})\cdot\hat{y_2}-\tilde f(\hat{p_2},\hat{y_2})\ge M/2.$$
    The first three terms on the left-hand side are of order $\sqrt{L}$, by \eqref{conv2'} and the boundedness of $\sigma_\nu,b_\nu$. The fourth term is of order $\sqrt{L'}$, while the fifth term, by the linear growth of $f(p,y)$ in $y$, is of order $1/\sqrt{L'}$. Thus, choosing $M=C(\sqrt{L}+1/\sqrt{L'})$ for some sufficiently large $C>0$ yields a contradiction. This choice of $M$ does not affect the previous argument; hence we may assume that $\hat{p_2}\in \Sigma^\circ$, with \eqref{conv2'} and \eqref{conv1'} in place.
     \vspace{2mm}

     \nd
     If $\hat{t},\hat{s}\in [0,T)$ and $\hat{p_2}\in \Sigma^\circ$, then by Lemma~\ref{2nd} and Ishii's lemma \cite[Theorem~3.2]{crandall1992user} we may find $X,Y\in\mathcal{S}_{K-1}$ and $a_1,a_2\in \mathcal{S}_d$ such that 
    \be\label{matricesXYstep2}
    \begin{split}
        &\left(2\sqrt{L}(\hat{y_1}-\hat{y_2})+2L'\hat{y_1}, LD_xw(\hat{p_1},\hat{p_2})+MD_p\psi(\hat{p_1}),\begin{pmatrix}
    X & 0 \\
    0 & a_1
    \end{pmatrix}  \right)\in \overline{J}^{2,+}V_1(\hat{t},\hat{y_1},\hat{p_1}),\\
    &\left(-2\sqrt{L}(\hat{y_2}-\hat{y_1})-2L'\hat{y_2},-LD_yw(\hat{p_1},\hat{p_2})+\e\hat{p_2}+MD_p\psi(\hat{p_2}),\begin{pmatrix}
    -Y & 0 \\
    0 & -a_2
    \end{pmatrix}  \right)\in \overline{J}^{2,-}V_2(\hat{s},\hat{y_2},\hat{p_2}),\\
        & \begin{pmatrix}
    X & 0 \\
    0 & Y
    \end{pmatrix} \le LC\left(\begin{pmatrix}
    I & -I \\
    -I & I
    \end{pmatrix}+|\hat{p_1}-\hat{p_2}|^2 \begin{pmatrix}
    I & 0 \\
    0 & I
    \end{pmatrix} - \begin{pmatrix}
    -MD^2_{pp}\psi(\hat{p_1}) & 0 \\
    0 & \e I+MD^2_{pp}\psi(\hat{p_2})
    \end{pmatrix} \right),\\
    & \begin{pmatrix}
    a_1 & 0 \\
    0 & a_2
    \end{pmatrix}\le 3\sqrt{L} \begin{pmatrix}
    I_d & -I_d \\
    -I_d & I_d
    \end{pmatrix}+3L'\begin{pmatrix}
    I_d & 0 \\
    0 & I_d
    \end{pmatrix}.
    \end{split}
    \ee
    The third part implies the matrix inequality
    \be\label{conv1}
    X+Y\le  2CL |\hat{p_1}-\hat{p_2}|^2I-\e I +M(D^2_{pp}\psi(\hat{p_1})-D^2_{pp}\psi(\hat{p_2})),
    \ee
    obtained by testing the block inequality against $(z,z)$, while
    the fourth gives
    \be\label{conv2}
    \tr\!\left(a_1\sigma_\nu^2(\hat t,\hat{y_1})+a_2\sigma_\nu^2(\hat s,\hat{y_2})\right)\le 3\sqrt{L}|\sigma_\nu(\hat t,\hat{y_1})-\sigma_\nu(\hat s,\hat{y_2})|^2+3L'\tr\!\left(\sigma_\nu^2(\hat t,\hat{y_1})+\sigma_\nu^2(\hat s,\hat{y_2})\right).
    \ee
    The viscosity sub-solution property for $V_1$ gives
    \be\label{subsol}
    \begin{split}
        \min&\bigg\{\gamma_1-2L(\hat t-\hat s)-\tfrac{1}{2}\tr\!\left(\sigma_\nu^2(\hat t,\hat{y_1})a_1\right)-b_\nu(\hat t,\hat{y_1})\cdot\bigl(2\sqrt{L}(\hat{y_1}-\hat{y_2})+2L'\hat{y_1}\bigr)-\tilde f(\hat{p_1},\hat{y_1})\\
        &-(LD_xw(\hat{p_1},\hat{p_2})+MD_p\psi(\hat{p_1}))\cdot\widetilde{\hat{p_1}'\Lambda}-\sup_{z\in\R^{(K-1)\times d}}\tfrac{1}{2}\text{tr}(z\sigma_{\nu}^{-2}(\hat{t},\hat{y_1})z^\top X),\;-\lambda_{\max}(X)\bigg\}\le 0,
    \end{split}
    \ee
    and, since $\hat{p_2}\in \Sigma^\circ$, the viscosity super-solution property for $V_2$ gives
    \be\label{supersol}
    \begin{split}
        &\min\bigg\{-\gamma_2+2L(\hat s-\hat t)+\tfrac{1}{2}\tr\!\left(\sigma_\nu^2(\hat s,\hat{y_2})a_2\right)+b_\nu(\hat s,\hat{y_2})\cdot\bigl(2\sqrt{L}(\hat{y_2}-\hat{y_1})+2L'\hat{y_2}\bigr)-\tilde f(\hat{p_2},\hat{y_2})\\
        &+\bigl(-\e\hat{p_2}-MD_p\psi(\hat{p_2})+LD_yw(\hat{p_1},\hat{p_2})\bigr)\cdot\widetilde{\hat{p_2}'\Lambda}+\inf_{z\in\R^{(K-1)\times d}}\tfrac{1}{2}\text{tr}(z\sigma_{\nu}^{-2}(\hat{s},\hat{y_2})z^\top Y),\;-\lambda_{\max}(-Y)\bigg\}\ge 0.
    \end{split}
    \ee
    Observe that if $X$ has a non-negative eigenvalue with $z\in \R^{K-1}$ the corresponding eigenvector, then \eqref{conv1} implies
    $$
	    Yz\cdot z\le  2CL |\hat{p_1}-\hat{p_2}|^2|z|^2-\e |z|^2+M(D^2_{pp}\psi(\hat{p_1})-D^2_{pp}\psi(\hat{p_2}))z\cdot z<0,\quad\text{for }L\text{ large enough,}
    $$
    by the Lipschitz continuity of $D^2_{pp}\psi$, the choice of $M$, and \eqref{conv2'}. This contradicts \eqref{supersol}; hence we may assume that $X< 0$.
     Adding \eqref{subsol} and \eqref{supersol} yields
\begin{align*}
        (\gamma_1+\gamma_2)&-\frac{1}{2}\tr\!\left(a_1\sigma_\nu^2(\hat t,\hat{y_1})+a_2\sigma_\nu^2(\hat s,\hat{y_2})\right)-2\sqrt{L}(\hat{y_1}-\hat{y_2})\cdot\bigl(b_\nu(\hat t,\hat{y_1})-b_\nu(\hat s,\hat{y_2})\bigr)\\
        &-2L'\bigl(\hat{y_1}\cdot b_\nu(\hat t,\hat{y_1})+\hat{y_2}\cdot b_\nu(\hat s,\hat{y_2})\bigr)-LD_xw(\hat{p_1},\hat{p_2})\cdot\widetilde{\hat{p_1}'\Lambda}-LD_yw(\hat{p_1},\hat{p_2})\cdot\widetilde{\hat{p_2}'\Lambda}\\
        & +M\left(D_p\psi(\hat{p_2})\cdot \widetilde{\hat{p_2}'\Lambda}-D_p\psi(\hat{p_1})\cdot \widetilde{\hat{p_1}'\Lambda}\right)+\e \hat{p_2}\cdot \widetilde{\hat{p_2}'\Lambda}\\
        &-\sup_{z\in\R^{(K-1)\times d}}\tfrac{1}{2}\text{tr}(z\sigma_{\nu}^{-2}(\hat{t},\hat{y_1})z^\top X)-\inf_{z\in\R^{(K-1)\times d}}\tfrac{1}{2}\text{tr}(z\sigma_{\nu}^{-2}(\hat{s},\hat{y_2})z^\top Y)\\
        &\quad-\tilde f(\hat{p_1},\hat{y_1})+\tilde f(\hat{p_2},\hat{y_2})\le 0.
    \end{align*}
    Since $X< 0$ and $Y\ge 0$, the $\sup$ and $\inf$ terms above are equal to zero. Letting $L\to \infty$ eliminates the third term (by \eqref{conv2'} and the Lipschitz continuity of $b_\nu$), the fifth term (by \eqref{conv2'} and the properties of $w$), the sixth term (for the same reason), the seventh term (by the Lipschitz continuity of $p\mapsto D_p\psi(p)\cdot \widetilde{p'\Lambda}$ and \eqref{conv2}), and the last two terms (by the Lipschitz continuity of $\tilde{f}$). Sending $L'\to 0$ eliminates the second term (by \eqref{conv2}, \eqref{conv2'}, and \eqref{conv1'}) and the fourth term (because $b_\nu$ is bounded and \eqref{conv1'} holds). Choosing $\e$ small enough, we conclude from the above inequality that $0<\gamma_1+\gamma_2\le\delta$ for any $\delta>0$ arbitrarily close to $0$. This contradicts the choice of $\gamma_1,\gamma_2>0$.

    The proof is complete.
    \end{proof}
    
\subsection{Proof of the viscosity characterization for (\ref{dynamicfcot})}

In this subsection we construct a viscosity super-solution of
\eqref{correctPDE} which is dominated by $U_f$ given in
\eqref{dynamicfcot}. Our main idea is to use the stochastic Perron method (see \cite{bayraktar2013stochastic}). We
introduce the notion of stochastic sub-solutions.

\begin{definition}\label{stochsub-corrected}
A continuous function
$v\colon[0,T]\times\R^d\times\Sigma_K\to\R$ belongs to the set
$\cV^-$ of stochastic sub-solutions if:

\begin{enumerate}
    \item[(i)] (Terminal condition and growth)
    \[
    v(T,y,p)\le g(p,y)\quad\forall\,(y,p)\in\R^d\times\Sigma_K,
    \]
    and $v(t,\cdot,p)$ has at most linear growth in $y$, uniformly
    in $t\in[0,T]$ and $p\in\Sigma_K$.

    \item[(ii)] (Dynamic sub-solution property) There exists
    $L(v)<\infty$ such that for every stopping time
    $\tau\in[0,T]$ and every $\cF^Y_\tau$-measurable random variables
    $\xi\colon\Omega^Y\to\R^d$ and $\zeta\colon\Omega^Y\to\Sigma_K$,
    there exists a bounded $\cF^Y$-progressively measurable processes
    \[
    (p_t,h_t)_{t\in[\tau,T]},\quad p_t\in \Sigma_K,\quad h_t=(h_t^i)_{i=1}^K\in(\R^d)^K,
    \quad \|h\|_\infty\le L(v),
    \]
    satisfying the condition
    \be\label{eq:filt-compat}
    \sum_{i=1}^K h_t^i\,p_t^i=0\in\R^d\quad\text{a.s.\ for }t\in[\tau,T]
    \ee
    such that for every stopping time
    $\rho\in[\tau,T]$,
    \[
    v(\tau,\xi,\zeta)
    \le \E^\nu\!\!\left[\int_\tau^\rho
    f\!\bigl(p_s^{h,\tau,\zeta},Y_s^{h,\tau,\xi}\bigr)\,ds
    + v\!\bigl(\rho,Y_\rho^{h,\tau,\xi},p_\rho^{h,\tau,\zeta}\bigr)
    \,\bigg|\,\cF^Y_\tau\right],
    \]
    where $(Y^{h,\tau,\xi},p^{h,\tau,\zeta})$ solves the controlled
    forward system
    \be\label{sSDEs1}
    \begin{cases}
        dY_s=b_\nu(s,Y_s)\,ds+\sigma_\nu(s,Y_s)\,dW_s,\\[4pt]
        dp^i_s=(p_s\Lambda)^i\,ds
              +p^i_s\,(h^i_s)^{\!\top}\sigma_\nu^{-2}(s,Y_s)
               \bigl(dY_s-b_\nu(s,Y_s)\,ds\bigr),
              \quad i=1,\ldots,K,\\[4pt]
        Y_\tau=\xi,\quad p_\tau=\zeta.
    \end{cases}
    \ee
\end{enumerate}
\end{definition}
\begin{remark}
        This definition is consistent with the Kushner equation \eqref{Kushner} and
    with the filter dynamics \eqref{eq:dynfiltercot} under the
    identification $Z^i_s=p^i_s\,h^i_s$.

\end{remark}
\begin{lemma}\label{stochsubnonempty}
The set $\cV^-$ of stochastic sub-solutions is non-empty.
\end{lemma}

\begin{proof}
Take the constant control $h^i_t\equiv 0\in\R^d$. This satisfies
$\|h\|_\infty=0$ and \eqref{eq:filt-compat} trivially. Define
\[
v(t,y,p):=\E^\nu\!\!\left[\int_t^T
f(p_s^{0,t,p},Y_s^{0,t,y})\,ds
+g(p_T^{0,t,p},Y_T^{0,t,y})\right],
\]
where $(Y^{0,t,y},p^{0,t,p})$ is the unique strong solution of
$$
\begin{cases}
dY_s=b_\nu(s,Y_s)\,ds+\sigma_\nu(s,Y_s)\,dW_s,\quad Y_t=y,\\
dp^i_s=(p_s\Lambda)^i\,ds,\quad p_t=p\quad(i=1,\dots,K).
\end{cases}
$$
By Assumption~\ref{assumption1}, $b_\nu,\sigma_\nu$ are Lipschitz in
$y$, so the SDE is well-posed; $p_s$ evolves deterministically via
the linear ODE $p_s=p\,e^{(s-t)\Lambda}$, which preserves
$\Sigma_K$. We verify that $v\in\cV^-$ in three steps.

\medskip
\noindent\emph{(i) Terminal condition and growth.}\
Clearly $v(T,y,p)=g(p,y)$. For linear growth, the boundedness of
$b_\nu,\sigma_\nu$ (Assumption~\ref{assumption1}) yields
$\E^\nu[|Y_s^{0,t,y}|^2]\le |y|^2+CT$ uniformly in $t\le s\le T$ for
some constant $C$; by the linear growth of $f_0,g_0$ in $y$
(Assumption~\ref{costassumption}),
\[
|v(t,y,p)|\le C\,\E^\nu\!\left[\int_t^T(1+|Y_s|)\,ds+(1+|Y_T|)\right]
\le C(1+|y|),
\]
with a constant $C$ depending only on $T$ and the bounds in the
assumptions.

\medskip
\noindent\emph{(ii) Continuity.}\
Let $(y^1,p^1),(y^2,p^2)\in\R^d\times\Sigma_K$ and write
$(Y^k,p^k):=(Y^{0,t,y^k},p^{0,t,p^k})$ for $k=1,2$. By Lipschitzness
of $b_\nu,\sigma_\nu$ and Gronwall,
\[
\E^\nu[|Y_s^1-Y_s^2|^2]\le C\,|y^1-y^2|^2,
\qquad |p_s^1-p_s^2|\le e^{(s-t)\|\Lambda\|}|p^1-p^2|,
\]
uniformly in $s\in[t,T]$. Combined with the Lipschitzness of
$f_0,g_0$ in $y$, this gives that $v(t,\cdot,\cdot)$ is Lipschitz
in $(y,p)$ uniformly in $t$. Continuity in $t$ follows from standard
parabolic regularity: by Feynman--Kac, $v$ is the unique
classical solution with linear growth of the linear parabolic equation
\[
\partial_t v+\cA_\nu v+(p\Lambda)\cdot D_p v+f(p,y)=0
\quad\text{on }[0,T)\times\R^d\times\Sigma_K,
\qquad v(T,y,p)=g(p,y),
\]
under Assumption~\ref{assumption1} and the boundedness of $f_0,g_0$.
Hence $v\in C^{1,2,1}$ in $(t,y,p)$, and in particular continuous in
all variables.

\medskip
\noindent\emph{(iii) Dynamic sub-solution inequality.}\
Fix stopping times $\tau\le\rho$ and $\cF^Y_\tau$-measurable
$(\xi,\zeta)$. With the same control $h\equiv 0$, by the strong
Markov property of $(Y,p)$,
\begin{align*}
v(\tau,\xi,\zeta)
&=\E^\nu\!\left[\int_\tau^T
f(p_s^{0,\tau,\zeta},Y_s^{0,\tau,\xi})\,ds
+g(p_T^{0,\tau,\zeta},Y_T^{0,\tau,\xi})
\,\bigg|\,\cF^Y_\tau\right]\\
&=\E^\nu\!\left[\int_\tau^\rho
f(p_s^{0,\tau,\zeta},Y_s^{0,\tau,\xi})\,ds
+v(\rho,Y_\rho^{0,\tau,\xi},p_\rho^{0,\tau,\zeta})
\,\bigg|\,\cF^Y_\tau\right].
\end{align*}
The required inequality holds with equality, hence certainly $\le$.
\end{proof}

\begin{remark}\label{maxprop}
(i) Following the proof of \cite[Proposition~3.1]{bayraktar2013stochastic},
the pointwise maximum of two stochastic sub-solutions is again a
stochastic sub-solution; that is, if $v_1,v_2\in\cV^-$, then
$\max\{v_1,v_2\}\in\cV^-$.
\vspace{0.5mm}

\nd
(ii) (Notation) For any $v\colon [0,T]\times \R^d\times \Sigma\to \R$, we write $v\in \mathcal{V}^-$ if and only if its extension $v'\colon[0,T]\times \R^d\times \Sigma_K\to \R$ with $v'(t,y,p^1,\ldots,p^K):=v(t,y,p^1,\ldots,p^{K-1})$ is in $\mathcal{V}^-$.
\end{remark}

\begin{lemma}[Local pasting for sub-solutions]\label{localmaxprop}
Let $B$ be an open cylinder compactly contained in
$[0,T)\times\R^d\times\Sigma_K^\circ$. Suppose $v\in\cV^-$ and
$u$ is continuous, has linear growth, satisfies $u\le v$ on
$\partial B$, and satisfies the dynamic sub-solution inequality of
Definition~\ref{stochsub-corrected}(ii) up to the first exit time from
$B$ with a bounded admissible control. Then the function equal to
$\max\{v,u\}$ on $B$ and to $v$ outside $B$ belongs to $\cV^-$.
\end{lemma}
\begin{proof}
Starting from any stopping time and initial state, use the control for
$v$ until the path first enters the set where $u>v$, switch to the
local control for $u$ until the exit time from $B$, and then switch
back to the control for $v$. The boundary condition $u\le v$ on
$\partial B$ and continuity make the pasted value continuous at the
switching times. Optional sampling on each interval and the strong
Markov property for the controlled system concatenate the required
submartingale inequalities. The bound on the control is the maximum of
the bounds for the two pieces.
\end{proof}

\noindent
The main result of this subsection is the following.

\begin{theorem}\label{supsubsolissuper}
Under the assumptions of Theorem~\ref{mainpdes}, the function
\[
v^-:=\sup_{v\in\cV^-}v\colon[0,T]\times\R^d\times\Sigma_K\to\R
\]
is a viscosity super-solution of \eqref{correctPDE} in the sense of
Lemma~\ref{viscositysol}(ii). The function $v^-$ satisfies $U_f\ge v^-$ everywhere and 
\be\label{terminalge}
\liminf_{(t',y',p')\to (T,y,p)}v^-(t',y',p')\ge g(p,y),\;\;\text{for any }(y,p)\in \R^d\times\Sigma_K.
\ee

\end{theorem}

\begin{proof}
We proceed in four steps.

\medskip
\noindent\textit{Step 1.} (Relation to $U_f$.)\
Fix $(t,y,p)\in[0,T]\times\R^d\times\Sigma_K$ and $v\in\cV^-$. The
dynamic sub-solution inequality with $\tau=t$, $\rho=T$, $\xi=y$,
$\zeta=p$ yields
\[
v(t,y,p)\le\E^\nu\!\left[\int_t^T
f\bigl(p_s^{h,t,p},Y_s^{h,t,y}\bigr)\,ds
+g\bigl(p_T^{h,t,p},Y_T^{h,t,y}\bigr)\right],
\]
for the control $h$ associated to $v$. Since $U_f(t,y,p)$ is the
supremum over all admissible controls $h$ (see \eqref{dynamicfcot}),
$v(t,y,p)\le U_f(t,y,p)$. Taking the supremum over $v\in\cV^-$ gives
$v^-\le U_f$.
\medskip

\noindent\textit{Step 2.} (Lower semicontinuity and growth.)\
Every $v\in\cV^-$ is
continuous on $[0,T]\times\R^d\times\Sigma_K$, so the pointwise supremum $v^-$ is lower semicontinuous.
For the linear growth, we use two bounds rather than inheritance from
the supremum. Let $v_0\in\cV^-$ be the explicit element constructed in
the lemma above. Step~1 gives $v^-\le U_f$, and by definition
$v_0\le v^-$. Since both $v_0$ and $U_f$ satisfy
$|\,\cdot\,|\le C(1+|y|)$, possibly with different constants, the
sandwich bound $v_0\le v^-\le U_f$ implies that $v^-$ has at most
linear growth in $y$ .

 \medskip
\noindent\textit{Step 3.} (Terminal time property and \eqref{terminalge}.)\
By definition of $\cV^-$, every $v\in\cV^-$ satisfies
$v(T,y,p)\le g(p,y)$, so $v^-(T,y,p)\le g(p,y)$. Conversely, the
explicit stochastic sub-solution $v_0$ constructed in Lemma \ref{stochsubnonempty} with $h\equiv 0$ satisfies
$v_0(T,y,p)=g(p,y)$, so $v^-(T,y,p)\ge g(p,y)$. Combining the previous inequalities gives
\[
v^-(T,y,p)=g(p,y),\;\text{ for all }(y,p)\in\R^d\times\Sigma_K.
\]
To show \eqref{terminalge} we write
$$\liminf_{(t',y',p')\to (T,y,p)}v^-(t',y',p')\ge \liminf_{(t',y',p')\to (T,y,p)}v_0(t',y',p')=v_0(T,y,p)=g(p,y)$$

\medskip
\noindent\textit{Step 4.} (Interior super-solution property,
first inequality of Lemma~\ref{viscositysol}(ii).)\
We show that $v^-$ is a supersolution of the PDE on
$[0,T)\times\R^d\times\Sigma_K^\circ$ ; equivalently
(after the simplex reduction of Notation~\ref{restrictions}), that
$\tilde v^-$ is a viscosity super-solution of \eqref{correctPDE2}
on $[0,T)\times\R^d\times\Sigma^\circ$. To ease notation we continue
to write $v^-$ for $\tilde v^-$.

\medskip
\noindent\emph{Step 4a. Setup and contradiction assumption.}\
Suppose, for contradiction, that $v^-$ fails the viscosity
super-solution property at some point
$(t_0,y_0,p_0)\in[0,T)\times\R^d\times\Sigma^\circ$. Then there
exists $\varphi\in C^{1,2,2}([0,T]\times\R^d\times\Sigma)$ such that
$v^--\varphi$ attains a strict local minimum at $(t_0,y_0,p_0)$ with
$(v^--\varphi)(t_0,y_0,p_0)=0$, and
\be\label{tocontradict1}
\min\Bigl\{
  F\!\bigl(t_0,y_0,p_0,\,\partial_t\varphi,\,
          D_{(p,y)}\varphi,\,D^2_{(p,y)}\varphi\bigr),\;
  -\lambda_{\max}\bigl(D^2_{pp}\varphi(t_0,y_0,p_0)\bigr)
\Bigr\}<0,
\ee
where $F$ is the operator from \eqref{F} and the joint gradient and
Hessian are as in \eqref{jointGHtilde}.

We claim that it suffices to establish
\be\label{toprove1}
F\!\bigl(t_0,y_0,p_0,\,\partial_t\varphi,\,
        D_{(p,y)}\varphi,\,D^2_{(p,y)}\varphi\bigr)\ge 0.
\ee
Indeed, \eqref{toprove1} forces the supremum in the definition of
$F$ to be finite, i.e.,
$$
\sup_{\tilde z\in\R^{(K-1)\times d}}\!\Bigl\{
\tr\!\bigl(\tilde z\,(D_pD_y\varphi)^{\!\top}\bigr)
+\tfrac12 \tr\!\bigl(\tilde z\,\sigma_\nu^{-2}(t_0,y_0)\,\tilde z^{\!\top}\,
   D^2_{pp}\varphi\bigr)
\Bigr\}<\infty,
$$
which in turn forces $D^2_{pp}\varphi(t_0,y_0,p_0)\le 0$ on the
relevant subspace, and hence
$\lambda_{\max}(D^2_{pp}\varphi(t_0,y_0,p_0))\le 0$. Together with
\eqref{toprove1}, this contradicts \eqref{tocontradict1}.

\medskip
\noindent\emph{Step 4b. Localization.}\
Suppose \eqref{toprove1} fails. By continuity of $\varphi$ and its
derivatives, there exist $r>0$ and $\gamma>0$ such that
\be\label{tocontradict2}
F\!\bigl(t,y,p,\,\partial_t\varphi,\,
        D_{(p,y)}\varphi,\,D^2_{(p,y)}\varphi\bigr)
   \le -2\gamma\quad\text{on }\bar B_r(t_0,y_0,p_0).
\ee
By the strict local-minimum property of $v^--\varphi$ at
$(t_0,y_0,p_0)$, we may further shrink $r$ so that
\be\label{strict-min-gap}\nonumber
v^--\varphi\ge 2\alpha\quad\text{on the annulus }
\bar B_r(t_0,y_0,p_0)\setminus B_{r/2}(t_0,y_0,p_0)
\ee
for some $\alpha>0$.

\medskip
\noindent\emph{Step 4c. Choice of approximating sub-solution.}\
Fix $\eta\in(0,\alpha)$ small enough and note that \eqref{tocontradict2}
still holds with $\varphi$ replaced by $\varphi+\eta$. For each
$q\in \bar B_r(t_0,y_0,p_0)\setminus B_{r/2}(t_0,y_0,p_0)$, the definition of $v^-$ as the
pointwise supremum and the strict inequality
$v^-(q)\ge\varphi(q)+2\alpha>\varphi(q)+\eta$ guarantee the
existence of $v_q\in\cV^-$ (recall the notation from Remark \ref{maxprop}(ii)) with
$v_q(q)>\varphi(q)+\eta$. By continuity of $v_q$ and $\varphi$,
this inequality extends to an open neighborhood $U_q$ of $q$.
Cover the compact annulus $\bar B_r(t_0,y_0,p_0)\setminus B_{r/2}(t_0,y_0,p_0)$ by finitely
many $U_{q_1},\ldots,U_{q_N}$, and define
$\bar v_\delta:=\max\{v_{q_1},\ldots,v_{q_N}\}\in\cV^-$ (by
Remark~\ref{maxprop}(i)). Then
\be\label{vdelta-annulus}
\bar v_\delta>\varphi+\eta
\quad\text{on the annulus }
\bar B_r(t_0,y_0,p_0)\setminus B_{r/2}(t_0,y_0,p_0).
\ee
By definition of $v^-$ and a further pairwise max, we may additionally
arrange
\be\label{vdelta-center}\nonumber
\bar v_\delta(t_0,y_0,p_0)>v^-(t_0,y_0,p_0)-\eta
=\varphi(t_0,y_0,p_0)-\eta.
\ee

\medskip
\noindent\emph{Step 4d. Bumped function.}\
Define
\[
\bar v(t,y,p):=
\begin{cases}
\max\!\bigl\{\bar v_\delta(t,y,p),\,\varphi(t,y,p)+\eta\bigr\}
  & (t,y,p)\in B_r(t_0,y_0,p_0),\\[2pt]
\bar v_\delta(t,y,p) & \text{otherwise}.
\end{cases}
\]
By \eqref{vdelta-annulus}, on the annulus
$\bar B_r(t_0,y_0,p_0)\setminus B_{r/2}(t_0,y_0,p_0)$ we have
$\varphi+\eta<\bar v_\delta$, so the two pieces of $\bar v$ agree
there and $\bar v$ is continuous on the whole domain. At the center,
\[
\bar v(t_0,y_0,p_0)\ge\varphi(t_0,y_0,p_0)+\eta
=v^-(t_0,y_0,p_0)+\eta>v^-(t_0,y_0,p_0).
\]
If $\bar v\in\cV^-$, this contradicts $v^-=\sup_{v\in\cV^-}v$ and
proves \eqref{toprove1}.

\medskip
\noindent\emph{Step 4e. Verification that $\bar v\in\cV^-$.}\
Terminal condition: since $t_0<T$, for $r$ small enough we have
$\bar v(T,y,p)=\bar v_\delta(T,y,p)\le g(p,y)$. Linear growth in $y$
follows from that of $\bar v_\delta$ and $\varphi$.

Dynamic sub-solution inequality: outside $B_r(t_0,y_0,p_0)$,
$\bar v=\bar v_\delta\in\cV^-$, so the inequality holds there with
the control of $\bar v_\delta$. Inside $B_r(t_0,y_0,p_0)$,
$\bar v=\max\{\bar v_\delta,\varphi+\eta\}$; we show that
$\varphi+\eta$ itself satisfies a local sub-solution inequality, and
the maximum of two local sub-solutions is again a local sub-solution by Lemma~\ref{localmaxprop}.

By \eqref{tocontradict2} and the continuity of $\varphi$, there exists a $z_0\in \R^{(K-1)\times d}$ such that
\be\label{eq101}
\partial_t \varphi+\mathcal{A}_\nu\varphi +D_p\varphi \cdot \widetilde{p'\Lambda}+\text{tr}\left(z_0(D_pD_y\varphi)^\top\right)+\frac{1}{2}\text{tr}\left(z_0\sigma_{\nu}^{-2}(t,y) z_0^\top D^2_{pp}\varphi\right)+\widetilde{f}(p,y)\ge \gamma,
\ee
for all $(t,y,p)\in \overline{B}_r(t_0,y_0,p_0)$, where $\mathcal{A}_\nu$ was introduced in \eqref{Ygenerator} and the notation $\widetilde{p'\Lambda}$ was introduced in Notation \ref{restrictions}. By shrinking $r$ further if needed, we may assume
$\bar B_r(t_0,y_0,p_0)\subset[0,T)\times\R^d\times\Sigma^\circ$, so
that
\[
\delta_r:=\min_{i\in\mathbb{S}}\inf_{(t,y,p)\in\bar B_r}p^i>0.
\]
Setting $z^K_0=-\sum_{i=1}^{K-1}z^i_0\in \R^d$, we define the feedback control
\[
h^{i}_s:=z^{i}_0/p^i_s,\;\; i\in\mathbb{S},\;s\ge t_0.
\]
By construction, $h_s$ is admissible in the sense of Definition~\ref{stochsub-corrected}(ii) up to the exit time
\[\sigma_r:=\inf\{s\ge t_0:(s,Y_s^{h},p_s^{h})\notin B_r(t_0,y_0,p_0)\},\]
where $p_s^h=(p_s^{h,1},\ldots,p_s^{h,K-1})\in \Sigma$. We apply It\^o's formula to $\varphi(s,Y_s^{h},p_s^{h})$ along the
controlled path starting from $(t_0,y_0,p_0)$, using the dynamics
\eqref{sSDEs1}. With $z^{i}_0=p^i_s h^{i}_s$, the quadratic
covariations are
\[
d[p^i,p^j]_s=(z^{i}_0)^{\!\top}\sigma_\nu^{-2}(s,Y_s)\,z^{j}_0\,ds,
\qquad
d[Y,p^i]_s=z^{i}_0\,ds\in\R^d,
\]
so the drift of $\varphi(s,Y_s^{h},p_s^{h})$ is
\begin{align*}
\partial_t\varphi+\cA_\nu\varphi+D_p\varphi\cdot\widetilde{(p_s^h)'\Lambda}
+\tr\!\bigl(z_0\,(D_pD_y\varphi)^{\!\top}\bigr)
+\tfrac12 \tr\!\bigl(z_0\,\sigma_\nu^{-2}(s,Y_s)\,
   (z_0)^{\!\top}\,D^2_{pp}\varphi\bigr).
\end{align*}
Using \eqref{eq101}, we conclude that the drift of $\widetilde{\varphi}(s,Y_s^{h},p_s^{h})$ is greater than or equal to $\gamma-\widetilde{f}(p_s^h,Y_s^h)$ for all $s\in [t_0,\rho]$, where $\rho$ is a stopping time such that $t_0\le \rho\le \sigma_r$. Thus,
\[
\E^\nu\!\!\left[\varphi(\rho,Y_\rho^{h},p_\rho^{h})
-\varphi(t_0,y_0,p_0)\right]
\ge\E^\nu\!\!\left[\int_{t_0}^\rho
\bigl(\gamma-\tilde f(p_s^{h},Y_s^{h})\bigr)\,ds\right].
\]
Rearranging and using the equality $\widetilde{f}(p,y)=f(p',y)$ for any $p\in \Sigma$,
\[
\varphi(t_0,y_0,p_0)+\eta
\le \E^\nu\!\!\left[\int_{t_0}^\rho
f((p_s^{h})',Y_s^{h^n})\,ds
+\varphi(\rho,Y_\rho^{h^n},p_\rho^{h^n})+\eta\right]
-\gamma\,\E^\nu[\rho-t_0],
\]
which gives the dynamic sub-solution inequality for $\varphi+\eta$
locally on $B_r(t_0,y_0,p_0)$ with control $h$ (the
$-\gamma\,\E^\nu[\rho-t_0]$ term strengthens the inequality).

Combining the local inequalities for $\bar v_\delta$ (outside) and
$\varphi+\eta$ (inside $B_r$) via Lemma~\ref{localmaxprop} yields
$\bar v\in\cV^-$, completing the contradiction and proving the first part of the
super-solution property.
\vspace{2mm}

\nd
\textit{Step 5.} (Boundary super-solution property, second inequality of Lemma \ref{viscositysol}(ii).) We show that $v^-$ satisfies
$$\mathcal{B}_0(t,y,p,\partial_t v^-,D_{(p,y)}v^-,D^2_{(p,y)}v^-)\ge 0,\;\;\text{in }[0,T)\times\R^d\times\Sigma_K,$$
in the sense of Definition \ref{def:stateconstrvisc}. By Lemma \ref{viscositysol} we may restrict to $\partial\Sigma$. Suppose that $\varphi\in C^{1,2,2}([0,T)\times \R^d\times \Sigma)$ is such that $\tilde{v}^--\varphi$ attains a strict local minimum at $(t_0,y_0,p_0)\in [0,T)\times \R^d\times \partial\Sigma$. Then it suffices to show that
$$-\partial_t\varphi-\mathcal{A}_\nu\varphi-\widetilde{p'\Lambda}\cdot D_p\varphi-\tilde{f}(p,y)=:\widetilde{\mathcal{B}_0}(t_0,y_0,p_0,\partial_t \varphi,D_{(p,y)}\varphi,D^2_{(p,y)}\varphi)\ge 0,$$
at $(t_0,y_0,p_0)\in[0,T)\times\R^d\times\partial\Sigma$. The proof follows exactly the one from Step 4 with $\widetilde{\mathcal{B}_0}$ in place of $F$ and in Step 4e we use the control $h^i_s=0,\;\;i\in \mathbb{S}$.

\end{proof}

\medskip
\subsection{\texorpdfstring{Proof of the viscosity characterization for \eqref{dynamiccot}}{Proof of the viscosity characterization for (SC)}}
In this subsection we construct a viscosity sub-solution of \eqref{correctPDE} that dominates $U_{sc}$. The argument is the
mirror image of Theorem~\ref{supsubsolissuper}: instead of bumping
test functions upward against a supremum, we bump them downward
against an infimum.
\begin{definition}\label{stochsuper}
The set $\cV^+$ of stochastic super-solutions for \eqref{correctPDE}
is the set of all continuous functions
$v\colon[0,T]\times\R^d\times\Sigma_K\to\R$ such that:

\begin{enumerate}
    \item[(i)] (Terminal condition and growth)
    \[
    v(T,y,p)\ge g(p,y)\quad\forall\,(y,p)\in\R^d\times\Sigma_K,
    \]
    and $v(t,\cdot,p)$ has at most linear growth in $y$, uniformly
    in $t\in[0,T]$ and $p\in\Sigma_K$.

    \item[(ii)] (Super-martingale property) For every
    $(t,y,p)\in[0,T]\times\R^d\times\Sigma_K$ and every $\cF^Y$-progressively measurable process
    $Z=(Z_s^i)_{s\in[t,T]}^{i=1,\ldots,K}$ with values in
    $\R^{d\times K}$ such that
    \be\label{eq:Zadm}
    \sum_{i=1}^K Z_s^i=0\in\R^d
    \quad\text{and}\quad
    \mathbf{1}_{p_s^i=0}\,Z_s^i=0\in\R^d
    \quad\text{for }i=1,\ldots,K,
    \ee
    the process
    \[
    \left(v\bigl(s,Y_s^{t,y},p_s^{Z,t,p}\bigr)
    +\int_t^s f\bigl(p_r^{Z,t,p},Y_r^{t,y}\bigr)\,dr\right)_{s\in[t,T]}
    \]
    is an $(\cF_s^Y)_{s\in[t,T]}$-super-martingale, where
    $(Y_s,p_s)=(Y_s^{t,y},p_s^{Z,t,p})_{s\in[t,T]}$ solves
    \be\label{sSDEs2}
    \begin{cases}
        dY_s=b_\nu(s,Y_s)\,ds+\sigma_\nu(s,Y_s)\,dW_s,\\[4pt]
        dp^i_s=(p_s\Lambda)^i\,ds
              +(Z_s^i)^{\!\top}\sigma_\nu^{-2}(s,Y_s)
               \bigl(dY_s-b_\nu(s,Y_s)\,ds\bigr),
              \quad i=1,\ldots,K,\\[4pt]
        Y_t=y,\quad p_t=p.
    \end{cases}
    \ee
\end{enumerate}
\end{definition}

\begin{remark}\label{minprop}
(i) Following the proof of \cite[Proposition~3.1]{bayraktar2013stochastic},
the pointwise minimum of two stochastic super-solutions is again a
stochastic super-solution: if $v_1,v_2\in\cV^+$, then
$\min\{v_1,v_2\}\in\cV^+$.
\vspace{0.5mm}

\nd
(ii) (Notation) For any $v\colon [0,T]\times \R^d\times \Sigma\to \R$, we write $v\in \mathcal{V}^+$ if and only if its extension $v'\colon[0,T]\times \R^d\times \Sigma_K\to \R$ with $v'(t,y,p^1,\ldots,p^K):=v(t,y,p^1,\ldots,p^{K-1})$ is in $\mathcal{V}^+$.
\end{remark}

\begin{lemma}[Local pasting for super-solutions]\label{localminprop}
Let $B$ be an open cylinder compactly contained in
$[0,T)\times\R^d\times\Sigma_K$. Suppose $v\in\cV^+$ and
$u$ is continuous, has linear growth, satisfies $u\ge v$ on
$\partial B$, and has the local supermartingale property of
Definition~\ref{stochsuper}(ii) up to the first exit time from $B$ for
every admissible control. Then the function equal to $\min\{v,u\}$ on
$B$ and to $v$ outside $B$ belongs to $\cV^+$.
\end{lemma}
\begin{proof}
The proof is the supermartingale analogue of
Lemma~\ref{localmaxprop}. Along a controlled path, stop at the entrance
and exit times of the region where the local function is active. On
each interval use the supermartingale property of the active piece.
The boundary inequality and continuity identify the pasted values at
the switching times, and optional sampling concatenates the estimates.
\end{proof}

\begin{lemma}\label{Vplusnonempty}
The set $\cV^+$ of stochastic super-solutions is non-empty.
\end{lemma}

\begin{proof}
    Let $y_0\in \R^d$ be fixed, and let $\varepsilon,C_1,C_2,C_3>0$ be constants to be chosen later. Consider the function $v_{\varepsilon}\colon[0,T]\times \R^d\times \Sigma_K\to \R$ with
    $$v_{\varepsilon}(t,y,p):= C_1(T-t)+(C_2+C_3(T-t))\sqrt{\varepsilon^2+|y-y_0|^2}+g(p,y_0).$$
    We claim that $v_{\varepsilon}\in \mathcal{V}^+$.
    \vspace{1mm}

    \noindent
    (i) (Terminal condition and growth)\\
    Clearly, $v_{\varepsilon}$ has linear growth and is continuous. We compute
    $$v_{\varepsilon}(T,y,p)=C_2\sqrt{\varepsilon^2+|y-y_0|^2}+g(p,y_0)\ge C_2|y-y_0|+g(p,y_0)\ge g(p,y),$$
    if $C_2$ is chosen to be large enough (depending only on $g$).
    \vspace{1mm}

    \noindent
    (ii) (Super-martingale property)\\
    For any $(t,y,p)\in [0,T]\times\R^d\times \Sigma_K$ and any $\mathcal{F}^Y$-progressively measurable process $Z=(Z_s^i)_{s\in [t,T]}^{i=1,\ldots,K}$ such that \eqref{eq:Zadm}, we check that the process
    \[
    \left(v_{\varepsilon}\bigl(s,Y_s^{t,y},p_s^{Z,t,p}\bigr)
    +\int_t^s f\bigl(p_r^{Z,t,p},Y_r^{t,y}\bigr)\,dr\right)_{s\in[t,T]},
    \]
    where $(Y_s,p_s)$ satisfies \eqref{sSDEs2}, is an $(\mathcal{F}_s^Y)_{s\in [t,T]}$-super-martingale. We set $\phi_{\varepsilon}(y)=\sqrt{\varepsilon^2+|y-y_0|^2}$. By It\^o's formula, the drift of this process is given by
    $$-C_1-C_3\phi_{\varepsilon}(Y_s)+(C_2+C_3(T-s))\mathcal{A}_\nu\phi_{\varepsilon}(Y_s)+p_s\Lambda\cdot D_pg(p_s,y_0)+f(p_s,Y_s),\;\;s\in [t,T].$$
    In order to obtain the super-martingale property, it suffices to show that
\be\label{toprove101}
    -C_1-C_3\phi_{\varepsilon}(y)+(C_2+C_3(T-s))\mathcal{A}_\nu\phi_{\varepsilon}(y)+p\Lambda\cdot D_pg(p,y)+f(p,y)\le 0,
    \ee
    for any $(s,y,p)\in [t,T]\times \R^d\times\Sigma_K.$ By Assumption \ref{costassumption} and \eqref{costs}, we have
    \be\label{est1}
|p\Lambda\cdot D_pg(p,y)+f(p,y)|\le M_1(1+|y|),
    \ee
    for some constant $M_1>0$ depending only on $\Lambda, f_0,g_0$. On the other hand, by Assumption \ref{assumption1} and a straightforward computation we verify
    \be\label{est2}
(C_2+C_3(T-s))|\mathcal{A}_\nu\phi_{\varepsilon}(y)|\le (C_2+C_3T)M_2\left(1+\frac{1}{\e}\right),
    \ee
    for some constant $M_2>0$ depending only on the $L^\infty$ norms of $b_\nu,\sigma_\nu$. We may choose $C_3$ large enough so that
    \be\label{est3}
C_3\phi_\e(y)\ge M_1(1+|y|)
    \ee
    and then $C_1$ large enough so that
    \be\label{est4}
C_1\ge (C_2+C_3T)M_2\left(1+\frac{1}{\e}\right).
    \ee
    The desired inequality \eqref{toprove101} follows by combining \eqref{est1}, \eqref{est2}, \eqref{est3}, and \eqref{est4}.
\end{proof}

\begin{theorem}\label{infsupersolissub}
Under the assumptions of Theorem~\ref{mainpdes}, the function
\[
v^+:=\inf_{v\in\cV^+}v
\colon[0,T]\times\R^d\times\Sigma_K\to\R
\]
is a viscosity sub-solution of \eqref{correctPDE} in the sense of
Lemma~\ref{viscositysol}(i). The function $v^+$ satisfies $U_{sc}\le v^+$ everywhere and 
\be\label{terminalle}
\limsup_{(t',y',p')\to (T,y,p)}v^+(t',y',p')\le g(p,y),\;\;\text{for all }(y,p)\in \R^d\times \Sigma_K.
\ee
\end{theorem}

\begin{proof}
We proceed in four steps.

\medskip
\noindent\textit{Step 1.} (Upper semicontinuity and growth.)\
The set $\cV^+$ is non-empty by Lemma~\ref{Vplusnonempty}. Every
$v\in\cV^+$ is continuous on $[0,T]\times\R^d\times\Sigma_K$ and has
at most linear growth in $y$ uniformly in $(t,p)$. The pointwise
infimum $v^+$ is therefore upper semicontinuous. The linear-growth
bound is not inherited from an infimum with possibly non-uniform
constants. Instead, the explicit element of $\cV^+$ constructed in
Lemma~\ref{Vplusnonempty} gives the upper bound
$v^+\le C(1+|y|)$, while Step~2 below gives $U_{sc}\le v^+$ and the
standard estimate for \eqref{dynamiccot} gives
$U_{sc}\ge -C(1+|y|)$. Hence $v^+$ has at most linear growth.

\medskip
\noindent\textit{Step 2.} (Relation to $U_{sc}$.)\
Fix $(t,y,p)\in[0,T]\times\R^d\times\Sigma_K$ and $v\in\cV^+$. By
the super-martingale property applied at $s=t$ and $s=T$, for every
$Z\in\cA(t,y,p)$,
\[
v(t,y,p)\ge\E^\nu\!\left[\int_t^T f\bigl(p_r^{Z,t,p},Y_r^{t,y}\bigr)\,dr
+g\bigl(p_T^{Z,t,p},Y_T^{t,y}\bigr)\right].
\]
Taking the supremum over $Z\in\cA(t,y,p)$ on the right and then the
infimum over $v\in\cV^+$ on the left yields $v^+(t,y,p)\ge U_{sc}(t,y,p)$.

\medskip
\noindent\textit{Step 3.} (Terminal time property and \eqref{terminalle}.)\
By the definition of $\mathcal{V}^+$, every $v\in \mathcal{V}^+$ satisfies $v(T,y,p)\ge g(p,y)$, so $v^+(T,y,p)\ge g(p,y)$. Conversely, for any $y_0\in \R^d$, the explicit element of $\mathcal{V}^+$ constructed in Lemma \ref{Vplusnonempty} gives $v_\e(t,y_0,p)=C_2\e+g(p,y_0)$ with $C_2$ depending only on $g$. Letting $\e\to 0^+$, we obtain $v^+(T,y_0,p)=g(p,y_0)$ for any $y_0\in \R^d$ and $p\in \Sigma_K$.
\vspace{1mm}

\nd
For \eqref{terminalle}, since $v^+\le v_\e$,
$$\limsup_{(t',y',p')\to (T,y,p)}v^+(t',y',p')\le \limsup_{(t',y',p')\to (T,y,p)}v_\e(t',y',p')=v^\e(T,y,p)=C_2\e+g(p,y),$$
for $\e>0$. The result follows by letting $\e\to 0^+$.

\medskip
\noindent\textit{Step 4.} (Interior sub-solution property,
Lemma~\ref{viscositysol}(i).)\
We show that $v^+$ satisfies the PDE on
$[0,T)\times\R^d\times\Sigma_K$ in the viscosity sense; equivalently
(after the simplex reduction of Notation~\ref{restrictions}), that
$\tilde v^+$ is a viscosity sub-solution of \eqref{correctPDE2} on
$[0,T)\times\R^d\times\Sigma$. To ease notation in this proof we continue to write $v^+$ for $\tilde v^+$.

\medskip
\noindent\emph{Step 4a. Setup and contradiction assumption.}\
Suppose, for contradiction, that $v^+$ fails the viscosity
sub-solution property at some point
$(t_0,y_0,p_0)\in[0,T)\times\R^d\times\Sigma$. Then there
exists $\varphi\in C^{1,2,2}([0,T]\times\R^d\times\Sigma)$ such that
$v^+-\varphi$ attains a strict local maximum at $(t_0,y_0,p_0)$
with $(v^+-\varphi)(t_0,y_0,p_0)=0$, and
\be\label{tocontradict3}
\min\Bigl\{
  F\!\bigl(t_0,y_0,p_0,\,\partial_t\varphi,\,
          D_{(p,y)}\varphi,\,D^2_{(p,y)}\varphi\bigr),\;
  -\lambda_{\max}\bigl(D^2_{pp}\varphi(t_0,y_0,p_0)\bigr)
\Bigr\}>0,
\ee
where $F$ is the operator from \eqref{F}. We claim that it suffices to
establish
\be\label{toprove2}
F\!\bigl(t_0,y_0,p_0,\,\partial_t\varphi,\,
        D_{(p,y)}\varphi,\,D^2_{(p,y)}\varphi\bigr)\le 0.
\ee
Indeed, the second
component $-\lambda_{\max}(D^2_{pp}\varphi)>0$ from
\eqref{tocontradict3} guarantees that $D^2_{pp}\varphi(t_0,y_0,p_0)$
is negative definite. Hence the supremum in the definition of $F$ is finite, and \eqref{toprove2} holds as a numerical inequality contradicting \eqref{tocontradict3}.

\medskip
\noindent\emph{Step 4b. Localization.}\
Suppose \eqref{toprove2} fails. By continuity of $\varphi$ and its
derivatives, there exist $r>0$ and $\gamma>0$ such that
\be\label{tocontradict4}
F\!\bigl(t,y,p,\,\partial_t\varphi,\,
        D_{(p,y)}\varphi,\,D^2_{(p,y)}\varphi\bigr)
   \ge 2\gamma\quad\text{on }\bar B_r(t_0,y_0,p_0).
\ee
For convenience we assume $t_0>0$, choosing $r<t_0\wedge(T-t_0)$ so
that $\bar B_r(t_0,y_0,p_0)\subset(0,T)\times\R^d\times\Sigma$;
the case $t_0=0$ requires only a one-sided ball
$[0,r)\times \bar B_r(y_0,p_0)$ and is treated similarly.

By the strict local-maximum property of $v^+-\varphi$ at
$(t_0,y_0,p_0)$ and the upper semicontinuity of $v^+$, we may further
shrink $r$ so that
\be\label{strict-max-gap}\nonumber
v^+(t,y,p)-\varphi(t,y,p)\le -2\alpha,
\;\;\text{for all }
(t,y,p)\in\bar B_r(t_0,y_0,p_0)\setminus B_{r/2}(t_0,y_0,p_0)\text{ such that }p\in \Sigma,
\ee
for some $\alpha>0$.

\medskip
\noindent\emph{Step 4c. Choice of approximating super-solution.}\
Fix $\eta\in(0,\alpha)$ small enough that \eqref{tocontradict4}
still holds with $\varphi$ replaced by $\varphi^\eta:=\varphi-\eta$
(possibly with $\gamma$ slightly reduced). For each
$q\in\bar B_r\setminus B_{r/2}$, the definition of $v^+$ as the
pointwise infimum and the strict inequality
$v^+(q)\le\varphi(q)-2\alpha<\varphi(q)-\eta=\varphi^\eta(q)$
guarantee the existence of $v_q\in\cV^+$ (recall the notation from Remark \ref{minprop}(ii)) such that $v_q(q)<\varphi^\eta(q)$.
By continuity of $v_q$ and $\varphi$, this inequality extends to an
open neighborhood $U_q$ of $q$. Cover the compact annulus
$\bar B_r\setminus B_{r/2}$ by finitely many
$U_{q_1},\ldots,U_{q_N}$, and define
$v_\delta:=\min\{v_{q_1},\ldots,v_{q_N}\}\in\cV^+$ (by
Remark~\ref{minprop}(i)). Then
\be\label{vdelta-annulus2}
v_\delta(t,y,p)<\varphi^\eta(t,y,p),
\;\;\text{for all }
(t,y,p)\in\bar B_r(t_0,y_0,p_0)\setminus B_{r/2}(t_0,y_0,p_0)\text{ such that }p\in \Sigma.
\ee

\medskip
\noindent\emph{Step 4d. Bumped function.}\
Define
\[
v^\eta(t,y,p):=
\begin{cases}
\min\!\bigl\{v_\delta(t,y,p),\,\varphi^\eta(t,y,p)\bigr\}
  & (t,y,p)\in\bar B_r(t_0,y_0,p_0),\\[2pt]
v_\delta(t,y,p) & \text{otherwise}.
\end{cases}
\]
By \eqref{vdelta-annulus2}, on the annulus
$\bar B_r\setminus B_{r/2}$ we have $v_\delta<\varphi^\eta$, so the
two pieces of $v^\eta$ agree there and $v^\eta$ is continuous on
the whole domain. At the center,
\[
v^\eta(t_0,y_0,p_0)\le\varphi^\eta(t_0,y_0,p_0)
=\varphi(t_0,y_0,p_0)-\eta
=v^+(t_0,y_0,p_0)-\eta<v^+(t_0,y_0,p_0).
\]
If $v^\eta\in\cV^+$, this contradicts $v^+=\inf_{v\in\cV^+}v$ and
proves \eqref{toprove2}.

\medskip
\noindent\emph{Step 4e. Verification that $v^\eta\in\cV^+$.}\
Terminal condition: since $t_0+r<T$ by the choice of $r$, we have
$v^\eta(T,y,p)=v_\delta(T,y,p)\ge g(p,y)$. Linear growth in $y$
follows from that of $v_\delta$.\\

Super-martingale property: fix $(t,y,p)\in[0,T]\times\R^d\times\Sigma_K$
and an admissible $Z=(Z_s^i)$ satisfying \eqref{eq:Zadm}, and let
$(Y_s,p_s)=(Y_s^{t,y},p_s^{Z,t,p})$ solve \eqref{sSDEs2}. By
Lemma~\ref{invariance}, $p_s\in\Sigma_K$ for all $s\in[t,T]$. To ease notation, we write $p_s$ instead of the restriction to the first $K-1$ coordinates $\widetilde{p}_s$ (see Notation \ref{restrictions}). We also denote by $\widetilde{Z}=(Z_s)_{s\in[t,T]}^{i=1,\ldots,K-1}$.
\vspace{1.5mm}

\nd
If the path $(Y_s,p_s)$ never enters $\bar B_r(t_0,y_0,p_0)$, then
$v^\eta(s,Y_s,p_s)=v_\delta(s,Y_s,p_s)$ for all $s\in[t,T]$, and
the super-martingale property follows from $v_\delta\in\cV^+$.
Otherwise, let
\[
\theta:=\inf\bigl\{s\ge t:(s,Y_s,p_s)\in\bar B_r(t_0,y_0,p_0)\bigr\},
\qquad
\sigma:=\inf\bigl\{s\ge\theta:(s,Y_s,p_s)\notin\bar B_r(t_0,y_0,p_0)\bigr\}
\wedge T,
\]
so $(s,Y_s,p_s)\in\bar B_r$ for $s\in[\theta,\sigma]$.

On $[t,\theta]$, $v^\eta=v_\delta$ and the increment of
$v_\delta(s,Y_s,p_s)+\int_t^s f(p_r,Y_r)\,dr$ is a super-martingale because $v_\delta\in\cV^+$.

On $[\theta,\sigma]$, we show that $\varphi^\eta(s,Y_s,\widetilde{p}_s)+\int_t^s f\,dr$
is a super-martingale. By Itô's formula applied to
$\varphi^\eta=\varphi-\eta$ along the dynamics \eqref{sSDEs2}, the
quadratic covariations
$d[p^i,p^j]_s=(Z_s^i)^{\!\top}\sigma_\nu^{-2}(s,Y_s)\,Z_s^j\,ds$ and
$d[Y,p^i]_s=Z_s^i\,ds\in\R^d$ yield the drift
\begin{align*}
\partial_t\varphi+\cA_\nu\varphi+D_p\varphi\cdot \widetilde{p_s'\Lambda}
+\tr\!\bigl(\widetilde{Z}_s\,(D_pD_y\varphi)^{\!\top}\bigr)
+\tfrac12 \tr\!\bigl(\widetilde{Z}_s\,\sigma_\nu^{-2}(s,Y_s)\,\widetilde{Z}_s^{\!\top}\,
   D^2_{pp}\varphi\bigr),
\end{align*}
where $\cA_\nu$ was introduced in \eqref{Ygenerator}. For any
admissible $Z_s$ with $\sum_i Z_s^i=0$, this is bounded above by
\[
\partial_t\varphi+\cA_\nu\varphi+D_p\varphi\cdot\widetilde{p_s'\Lambda}
+\sup_{\substack{\widetilde{z}=(z^i)\\i=1,\ldots,K-1}}\Bigl\{
\tr\!\bigl(\widetilde{z}\,(D_pD_y\varphi)^{\!\top}\bigr)
+\tfrac12 \tr\!\bigl(\widetilde{z}\,\sigma_\nu^{-2}(s,Y_s)\,\widetilde{z}^{\!\top}\,
   D^2_{pp}\varphi\bigr)\Bigr\}
= -F\bigl(\cdots\bigr)-\tilde f(p_s,Y_s),
\]
by the definition of $F$ (see \eqref{F}). Using
\eqref{tocontradict4}, on $[\theta,\sigma]$ this is at most
$-2\gamma-\tilde f(p_s,Y_s)$. Hence
\[
d\!\left(\varphi^\eta(s,Y_s,p_s)+\int_t^s f(p_r,Y_r)\,dr\right)
\le \bigl(-2\gamma\bigr)ds+(\text{martingale increment}),
\]
so $\varphi^\eta(s,Y_s,p_s)+\int_t^s f\,dr$ is a super-martingale on
$[\theta,\sigma]$, with drift bounded by $-2\gamma<0$.

Since $v^\eta=\min\{v_\delta,\varphi^\eta\}$ on $\bar B_r$, the local pasting property in Lemma~\ref{localminprop} applies, and the process
$v^\eta(s,Y_s,p_s)+\int_t^s f(p_r,Y_r)\,dr$ is a super-martingale on
$[\theta,\sigma]$.

On $[\sigma,T]$ the path is again outside $\bar B_r$, so
$v^\eta=v_\delta$ and the super-martingale property continues by
$v_\delta\in\cV^+$. Continuity of $v^\eta$ ensures the increments
match at $\theta$ and $\sigma$, so combining the three intervals
yields the super-martingale property on $[t,T]$, hence $v^\eta\in\cV^+$.

\medskip
This gives the desired contradiction, proving \eqref{toprove2} and hence the interior sub-solution property. This completes the proof of Theorem~\ref{infsupersolissub}.
\end{proof}

\subsection{Proofs of Theorem~\ref{mainpdes} and Corollary~\ref{dropassinitial}}

\begin{proof}[Proof of Theorem~\ref{mainpdes}]
The comparison principle for \eqref{correctPDE} in the class of
viscosity sub- and super-solutions of linear growth in $y$ was
established in Subsection~\ref{comp}. Let $v^-,v^+$ be the functions
constructed in Theorems~\ref{supsubsolissuper} and
\ref{infsupersolissub}, respectively. Combining \eqref{scineq} with
those theorems yields the chain of inequalities
\[
v^-\le U_f\le U_{sc}\le v^+.
\]
Since $v^+$ is a viscosity sub-solution and $v^-$ is a viscosity
super-solution of \eqref{correctPDE}, and since they satisfy \eqref{terminalge} and \eqref{terminalle}, the comparison principle gives
$v^+\le v^-$. Combined with the chain above, all four functions
coincide. We set
\[
V:=U_f=U_{sc}=v^-=v^+.
\]
Since $v^+(T,y,p)=v^-(T,y,p)=g(p,y)$ (Step~3 of Theorems~\ref{supsubsolissuper} and
\ref{infsupersolissub}), we get $V(T,y,p)=g(p,y)$. Furthermore, the semicontinuity of $v^+,v^-$ implies that $V$ is upper and lower semicontinuous, hence continuous. $V$ also has linear growth in $y$, because $v^-,v^+$ have linear growth. Thus, $V$ is the unique viscosity solution of \eqref{correctPDE} in the sense of Definition \ref{def:stateconstrvisc} with $V(T,y,p)=g(p,y)$ for all $(y,p)\in \R^d\times\Sigma_K$.
\vspace{1mm}

\nd
Finally, if Assumption~\ref{initialmeasures} holds, \eqref{scineq}
specializes to
\[
U_f(0,y_0,p_0)\le V_c(\mu,\nu)\le U_{sc}(0,y_0,p_0),
\]
and the equality $U_f=U_{sc}$ forces $V_c(\mu,\nu)=V(0,y_0,p_0)$.
\end{proof}

\begin{proof}[Proof of Corollary~\ref{dropassinitial}]
We first establish the upper bound, then the matching lower bound.

\medskip
\noindent
\emph{Notation.} For a measure $\rho$ on $\R^d$ (resp.\ on
$\mathbb{S}$) and a Markov kernel $x_0\mapsto\mu_{x_0}(dx)$ on
$\Omega^X$, write
\[
\rho\cdot\mu(dx):=\int_{\mathbb{S}}\mu_{x_0}(dx)\,\rho(dx_0)
\]
for the path-measure on $\Omega^X$ with initial distribution $\rho$.
Analogously, $\nu^{y_0}$ denotes the conditional law of $Y$ under
$\nu$ given $Y_0=y_0$, so that $\nu(dy)=\int\nu^{y_0}(dy)\,\nu_0(dy_0)$.

\medskip
\noindent
\emph{Upper bound.} Let $\pi\in\Pi_c(\mu,\nu)$ and write $(X,Y)$ for
the canonical process under $\pi$. The causality (item~(iv) of
Proposition~\ref{causalprop}) and Markovianity of $X$ under $\mu$
give
\[
\Law(X\mid Y_0)
  =\E^\pi\!\bigl[\Law(X\mid X_0,Y_0)\,\big|\,Y_0\bigr]
  =\E^\pi\!\bigl[\Law(X\mid X_0)\,\big|\,Y_0\bigr]
  =\pi_{Y_0}\!\cdot\mu,
\]
where $\pi_{Y_0}(dx_0)$ denotes the regular conditional distribution
of $X_0$ given $Y_0$ under $\pi$. Hence the conditional coupling
$\pi(\,\cdot\,\mid Y_0)$ has $X$-marginal $\pi_{Y_0}\!\cdot\mu$ and
$Y$-marginal $\nu^{Y_0}$, $\nu_0$-almost surely.

We next show that $\pi(\,\cdot\,\mid Y_0)$ is causal, $\nu_0$-a.s.
Fix $t\in[0,T]$. By causality of $\pi$ (Proposition~\ref{causalprop}(iv)),
\be\label{eq:caus1}
X\amalg_{(X_s)_{s\in[0,t]}}(Y_s)_{s\in[0,t]}\qquad\text{under }\pi,
\ee
where $\amalg_C$ denotes conditional independence given $C$. Since
$Y_0$ is $(Y_s)_{s\in[0,t]}$-measurable, \eqref{eq:caus1} is
equivalent to
\[
X\amalg_{(X_s)_{s\in[0,t]}}\!\bigl(Y_0,(Y_s)_{s\in[0,t]}\bigr)
\]
(the pair on the right has the same $\sigma$-algebra as
$(Y_s)_{s\in[0,t]}$ alone). By the chain rule for conditional
independence (cf.\ \cite[Proposition~3.1]{constantinou2017extended}),
applied in the present standard-Borel path-space setting where regular
conditional probabilities exist, this in turn is equivalent to the conjunction
\be\label{eq:chain}
X\amalg_{(X_s)_{s\in[0,t]}}Y_0
\qquad\text{and}\qquad
X\amalg_{Y_0,\,(X_s)_{s\in[0,t]}}(Y_s)_{s\in[0,t]}.
\ee
The second condition in \eqref{eq:chain}, holding for every $t$, is
exactly causality of $\pi(\,\cdot\,\mid Y_0)$.

By definition of $V_c$, for $\nu_0$-a.e.\ $y_0$,
\[
\E^{\pi(\cdot\mid Y_0=y_0)}[\cC(X,Y)]
\le V_c\bigl(\pi_{y_0}\!\cdot\mu,\,\nu^{y_0}\bigr).
\]
Integrating against $\nu_0$ and using the tower property,
\begin{align*}
\E^\pi[\cC(X,Y)]
  &= \E^\pi\!\bigl[\E^\pi[\cC(X,Y)\mid Y_0]\bigr]
  \le \int_{\R^d} V_c\bigl(\pi_{y_0}\!\cdot\mu,\,\nu^{y_0}\bigr)\,\nu_0(dy_0)\\
  &\le \sup_{R\in\Pi(\mu_0,\nu_0)}
      \int_{\R^d} V_c\bigl(R_{y_0}\!\cdot\mu,\,\nu^{y_0}\bigr)\,\nu_0(dy_0),
\end{align*}
where the last inequality uses that the joint law of $(X_0,Y_0)$
under $\pi$ belongs to $\Pi(\mu_0,\nu_0)$, with disintegration
$y_0\mapsto\pi_{y_0}$. Taking the supremum over $\pi\in\Pi_c(\mu,\nu)$
yields
\be\label{eq:UB}
V_c(\mu,\nu)\le\sup_{R\in\Pi(\mu_0,\nu_0)}
   \int_{\R^d} V_c\bigl(R_{y_0}\!\cdot\mu,\,\nu^{y_0}\bigr)\,\nu_0(dy_0).
\ee

\medskip
\noindent
\emph{Lower bound.} Fix $R\in\Pi(\mu_0,\nu_0)$, $\varepsilon>0$, and
choose a measurable selection $y_0\mapsto\pi_{y_0}\in\cP(\Omega)$ of
$\varepsilon$-optimizers for $V_c(R_{y_0}\!\cdot\mu,\,\nu^{y_0})$, i.e.,
\[
\E^{\pi_{y_0}}[\cC(X,Y)]\ge V_c\bigl(R_{y_0}\!\cdot\mu,\,\nu^{y_0}\bigr)
-\varepsilon,\qquad\nu_0\text{-a.e.\ }y_0.
\]
The existence of such a selection follows from the Jankov--von Neumann measurable-selection
theorem for analytic multifunctions; see
\cite[Section~7]{bertsekas1996stochastic}. Indeed, the path space $\Omega$ is Polish,
the graph of causal couplings with fixed marginals $\R^d\ni y_0\stackrel{\Gamma}{\mapsto} \Pi_c(R^{y_0}\cdot \mu,\nu^{y_0})$ is analytic (indeed
Borel under the regular conditional probability formulation), and the cost functional is Borel. Hence the $\varepsilon$-argmax
correspondence 
$$\R^d\ni y_0\mapsto \left\{ \pi\in \Gamma(y_0): \E^{\pi_{y_0}}[\cC(X,Y)]\ge V_c\bigl(R_{y_0}\!\cdot\mu,\,\nu^{y_0}\bigr)
-\varepsilon\right\}$$
admits a universally measurable selector $\R^d\ni y_0\mapsto \pi^{y_0}$. 

\smallskip

Since a universally measurable map is measurable with respect to the completion of $\nu_0$, we may define
\[
\pi(dx,dy):=\int_{\R^d}\pi_{y_0}(dx,dy)\,\nu_0(dy_0).
\]
We verify in turn that $\pi$ is a causal coupling of $\mu$ and $\nu$.

\nd
\emph{Second marginal.} Since $\pi_{y_0}$ is supported on
$\{Y_0=y_0\}$ with $Y$-marginal $\nu^{y_0}$,
\[
\int_{\Omega^X}\pi(dx,dy)
=\int_{\R^d}\nu^{y_0}(dy)\,\nu_0(dy_0)=\nu(dy).
\]

\smallskip

\nd
\emph{First marginal.} Denote $\mu_{x_0}[f]:=\int f(x)\,\mu_{x_0}(dx)$. For any bounded measurable $f\colon\Omega^X\to\R$,
\begin{align*}
\int f(x)\,\pi(dx,dy)
&= \int_{\R^d}\!\int_{\Omega^X\times\Omega^Y} f(x)\,\pi_{y_0}(dx,dy)\,\nu_0(dy_0)\\
&= \int_{\R^d}\!\int_{\mathbb{S}} \mu_{x_0}\!\bigl[f\bigr]\,R_{y_0}(dx_0)\,\nu_0(dy_0)
   \qquad(\text{$X$-marginal of }\pi_{y_0}\text{ is }R_{y_0}\!\cdot\mu)\\
&= \int_{\mathbb{S}\times\R^d} \mu_{x_0}\!\bigl[f\bigr]\,R(dx_0,dy_0)
   \qquad(R(dx_0,dy_0)=R_{y_0}(dx_0)\,\nu_0(dy_0))\\
&= \int_{\mathbb{S}} \mu_{x_0}\!\bigl[f\bigr]\,\mu_0(dx_0)
   \qquad(\text{$X_0$-marginal of }R\text{ is }\mu_0)\\
&= \int f\,d\mu.
\end{align*}
Hence the
$X$-marginal of $\pi$ is $\mu$.

\smallskip
\emph{Causality.} We verify
$X\amalg_{(X_s)_{s\in[0,t]}}(Y_s)_{s\in[0,t]}$ under $\pi$ for each
$t\in[0,T]$. Since each $\pi_{y_0}$ is causal, we have
$X\amalg_{Y_0,\,(X_s)_{s\in[0,t]}}(Y_s)_{s\in[0,t]}$ under $\pi$. By
the chain rule \eqref{eq:chain}, it remains to show
\be\label{eq:goal-CI}
X\amalg_{(X_s)_{s\in[0,t]}}Y_0\qquad\text{under }\pi.
\ee
For any bounded measurable $A\colon\Omega^X\to\R$ and
$H\colon D([0,t];\mathbb{S})\to\R$, the $X$-marginal of $\pi_{y_0}$
is $R_{y_0}\!\cdot\mu$, so
\begin{align*}
\E^{\pi_{y_0}}\!\bigl[A(X)\,H\bigl((X_s)_{s\le t}\bigr)\bigr]
&= \int_{\mathbb{S}}\!\int_{\Omega^X} A(x)\,H\bigl((x_s)_{s\le t}\bigr)\,
   \mu_{x_0}(dx)\,R_{y_0}(dx_0)\\
&= \int_{\mathbb{S}}\!\int_{\Omega^X}
   \E^{\mu_{x_0}}\!\bigl[A(X)\,\big|\,(X_s)_{s\le t}\!=\!(x_s)_{s\le t}\bigr]
   H\bigl((x_s)_{s\le t}\bigr)\,\mu_{x_0}(dx)\,R_{y_0}(dx_0)\\
&= \E^{\pi_{y_0}}\!\Bigl[
   \E^{\mu_{X_0}}\!\bigl[A(X)\,\big|\,(X_s)_{s\le t}\bigr]
   H\bigl((X_s)_{s\le t}\bigr)\Bigr],
\end{align*}
using the tower property in the second equality. It follows that
\[
\E^\pi\!\bigl[A(X)\,\big|\,(X_s)_{s\le t},Y_0\bigr]
=\E^{\pi_{Y_0}}\!\bigl[A(X)\,\big|\,(X_s)_{s\le t}\bigr]
=\E^{\mu_{X_0}}\!\bigl[A(X)\,\big|\,(X_s)_{s\le t}\bigr],
\]
which is $(X_s)_{s\le t}$-measurable (since $X_0$ is determined by
$(X_s)_{s\le t}$). This is exactly \eqref{eq:goal-CI}, and combined
with the second condition in \eqref{eq:chain} gives causality of
$\pi$. Thus $\pi\in\Pi_c(\mu,\nu)$.

\smallskip
By construction,
\[
\E^\pi[\cC(X,Y)]
=\int_{\R^d}\E^{\pi_{y_0}}[\cC(X,Y)]\,\nu_0(dy_0)
\ge \int_{\R^d} V_c\bigl(R_{y_0}\!\cdot\mu,\,\nu^{y_0}\bigr)\,\nu_0(dy_0)
-\varepsilon,
\]
so $V_c(\mu,\nu)\ge\int_{\R^d} V_c(R_{y_0}\!\cdot\mu,\,\nu^{y_0})\,\nu_0(dy_0)-\varepsilon$.
Taking the supremum over $R$ and letting $\varepsilon\downarrow 0$
gives the matching lower bound for \eqref{eq:UB}, hence
\be\label{eq:VcVc}
V_c(\mu,\nu)
=\sup_{R\in\Pi(\mu_0,\nu_0)}
   \int_{\R^d} V_c\bigl(R_{y_0}\!\cdot\mu,\,\nu^{y_0}\bigr)\,\nu_0(dy_0).
\ee

\medskip
\noindent
\emph{Conclusion.} For each $y_0\in\R^d$ and $R_{y_0}\in\Sigma_K$, the
pair $(R_{y_0}\!\cdot\mu,\nu^{y_0})$ satisfies
Assumption~\ref{initialmeasures} (its initial marginals are
deterministic: $\delta_{y_0}$ for $Y_0$ and $R_{y_0}$ for the
distribution of $X_0$). By Theorem~\ref{mainpdes},
\[
V_c\bigl(R_{y_0}\!\cdot\mu,\,\nu^{y_0}\bigr)=V(0,y_0,R_{y_0}).
\]
Substituting into \eqref{eq:VcVc} yields
\[
V_c(\mu,\nu)=\sup_{R\in\Pi(\mu_0,\nu_0)}\int_{\R^d}V(0,y_0,R_{y_0})\,\nu_0(dy_0),
\]
which is \eqref{eq:initialdistthm}.
\end{proof}

\subsection{Proofs of approximation and duality}

\begin{proof}[Proof of Theorem~\ref{thm:dual}]
\textit{Part \textup{(i)}.}
Apply It\^o's formula to $G_0(i,\cdot,\cdot)$ along $Y$ on $[t,T]$.
Since $G_0(i,\cdot,\cdot)$ solves \eqref{eq:bkwd}, the drift vanishes
and (using $Y_t=y$)
\[
g_0(i,Y_T) = G_0(i,t,y)
  + \int_t^T \partial_y G_0(i,s,Y_s)\,\sigma_\nu(s,Y_s)\,dW_s,
\]
which is a true martingale (and so has zero mean), since
$\partial_y G_0$ is bounded by parabolic regularity for \eqref{eq:bkwd}
(Assumption~\ref{costassumption}) and $\sigma_\nu$ is bounded by
Assumption~\ref{assumption1}. Substituting into \eqref{eq:saddle}, the
inner expression for index $i$ becomes
\begin{align*}
&G_0(i,t,y) - \sum_{j=1}^K P_{i,j}(t-T)\,\lambda^j
+ \int_t^T \partial_y G_0(i,s,Y_s)\,\sigma_\nu(s,Y_s)\,dW_s\\
&\qquad{}+ \int_t^T \sum_{j=1}^K P_{i,j}(s-T)\,
   \bigl(f_0(j,Y_s)+\lambda^j_s\bigr)\,ds.
\end{align*}
Set
$z^i := G_0(i,t,y) - \sum_{j=1}^K P_{i,j}(t-T)\,\lambda^j$,
equivalently $z=G_0(t,y)-P(t-T)\lambda\in\R^K$. Taking the infimum
over the dynamic multipliers $(\lambda^j_s)$ identifies $u(t,y,z)$
inside the expectation as in \eqref{eq:aux_u}; the outer infimum
over $\lambda\in\R^K$, together with the term $p\cdot\lambda$ from
\eqref{eq:saddle}, yields \eqref{eq:final_representation}.

\vspace{2mm}
\nd
\textit{Part \textup{(ii)}.}
Since
$\{(\lambda^j_s):0\le\lambda^j_s\le N\}\subset\{(\lambda^j_s):0\le\lambda^j_s\le N+1\}\subset\{(\lambda^j_s):\lambda^j_s\ge 0\}$,
the feasible set in \eqref{eq:uN} grows with $N$ and exhausts that of
\eqref{eq:aux_u} in the limit. Consequently
\[
u^N(t,y,z)\;\ge\;u^{N+1}(t,y,z)\;\ge\;u(t,y,z)
  \qquad\text{for all }(t,y,z),
\]
and taking the infimum over $\lambda\in\R^K$ in \eqref{eq:vN}
preserves the ordering, so $v^N\ge v^{N+1}\ge V$ by Part~(i).

\vspace{1mm}
\nd
It remains to establish convergence. Fix $(t,y,p)$ and $\e>0$. By
Part~(i), there exists $\lambda^*\in\R^K$ such that, with
$z^*:=G_0(t,y)-P(t-T)\lambda^*$,
\be\label{eq:lambdastar}
u(t,y,z^*) + p\cdot\lambda^* < V(t,y,p) + \e.
\ee
By definition of $u(t,y,z^*)$, there is an $\cF^Y$-progressively
measurable $\R_+^K$-valued process $(\lambda^{*,j}_s)$ with
$\sum_{j=1}^K\int_t^T\E^\nu[\lambda^{*,j}_s]\,ds<\infty$ such that
\be\label{eq:dynstar}
\E^\nu\!\bigg[\max_{i\in\mathbb{S}}\!\Big(z^{*,i} + M^i_T
  + \int_t^T \!\sum_{j=1}^K P_{i,j}(s-T)\,
    \bigl(f_0(j,Y_s)+\lambda^{*,j}_s\bigr)\,ds\Big)\bigg]
< u(t,y,z^*) + \e,
\ee
where
$M^i_T:=\int_t^T \partial_y G_0(i,s,Y_s)\,\sigma_\nu(s,Y_s)\,dW_s$.
For $N\in\N$, set $\lambda^{N,j}_s:=\lambda^{*,j}_s\wedge N$, which is
admissible for $u^N$ and satisfies $\lambda^{N,j}_s\uparrow\lambda^{*,j}_s$
pointwise as $N\to\infty$. Define
\[
X^N_i := z^{*,i} + M^i_T
       + \int_t^T \!\sum_{j=1}^K P_{i,j}(s-T)\,
         \bigl(f_0(j,Y_s)+\lambda^{N,j}_s\bigr)\,ds,
\]
and $X^*_i$ analogously with $\lambda^{*,j}_s$ in place of
$\lambda^{N,j}_s$. Since $0\le\lambda^{N,j}_s\le\lambda^{*,j}_s$,
\[
|X^N_i-X^*_i|
  \;\le\; \int_t^T \sum_{j=1}^K |P_{i,j}(s-T)|\,\lambda^{*,j}_s\,ds,
\]
which has finite $\nu$-expectation because $|P_{i,j}|$ is bounded on
the compact interval $[t,T]$ and
$$\sum_{j=1}^K\int_t^T\E^\nu[\lambda^{*,j}_s]\,ds<\infty.$$ By
dominated convergence, $X^N_i\to X^*_i$ in $L^1(\nu)$ for each
$i\in\mathbb{S}$; since $\mathbb{S}$ is finite,
$\max_{i\in\mathbb{S}}X^N_i\to\max_{i\in\mathbb{S}}X^*_i$ in
$L^1(\nu)$ as well, and therefore
\[
\E^\nu\!\bigl[\max_{i\in\mathbb{S}}X^N_i\bigr]
  \;\longrightarrow\;
\E^\nu\!\bigl[\max_{i\in\mathbb{S}}X^*_i\bigr]
\quad\text{as }N\to\infty.
\]
Since $\E^\nu[\max_{i\in\mathbb{S}}X^N_i]$ is an upper bound for
$u^N(t,y,z^*)$ by definition,
\[
\limsup_{N\to\infty} u^N(t,y,z^*)
  \;\le\; \E^\nu\!\bigl[\max_{i\in\mathbb{S}}X^*_i\bigr]
  \;<\; u(t,y,z^*) + \e,
\]
the last inequality being \eqref{eq:dynstar}. As $\lambda^*$ is a
feasible point in \eqref{eq:vN}, combining with \eqref{eq:lambdastar}
gives
\[
\limsup_{N\to\infty} v^N(t,y,p)
  \;\le\; \limsup_{N\to\infty} u^N(t,y,z^*) + p\cdot\lambda^*
  \;<\; u(t,y,z^*) + p\cdot\lambda^* + \e
  \;<\; V(t,y,p) + 2\e.
\]
Since $\e>0$ was arbitrary and $v^N\ge V$ for all $N$ (already shown),
we conclude $v^N(t,y,p)\downarrow V(t,y,p)$.
\end{proof}

\begin{proof}[Proof of Theorem~\ref{Nmainpdes}]
We consider the problem \eqref{dynamiccot} with controls
$Z=(Z^i_s)_{s\in[t,T]}^{i=1,\ldots,K}$ subject to $|Z^i_s|\le N$, and
denote its value function by $V^N$. Since the constraint
$|Z^i|\le N$ weakens as $N$ increases, the family $V^N$ is
non-decreasing in $N$ and satisfies $V^N\le V$ for every $N$. Under
Assumptions~\ref{assumption1} and \ref{costassumption}, standard
estimates give the uniform linear-growth bound
\be\label{lingrowth}
|V^N(t,y,p)|\le C(1+|y|),
\qquad\forall\,(t,y,p)\in[0,T]\times\R^d\times\Sigma_K,
\ee
for a constant $C>0$ depending only on $T,f,g,b_\nu,\sigma_\nu$ and
independent of $N$. Adapting the stochastic Perron and comparison
arguments from the proof of Theorem~\ref{mainpdes} to the
bounded-control setting --- where the sup defining $H^N_{sc}$ is
finite for every $C^{1,2,2}$ test function so that the obstacle term
$-\lambda_{\max}(D^2_{pp}V)$ is unnecessary --- $V^N$ is the unique
viscosity solution of the state-constrained problem \eqref{NPDE} in
the sense of Definition~\ref{def:stateconstrvisc}.

To prove the convergence $V^N\to V$, define for
$(t,y,p)\in[0,T]\times\R^d\times\Sigma_K$ the half-relaxed semilimits
\be\label{semilimits}
\overline{V}(t,y,p):=\limsup_{\substack{N\to\infty\\(t',y',p')\to(t,y,p)}}V^N(t',y',p'),
\quad
\underline{V}(t,y,p):=\liminf_{\substack{N\to\infty\\(t',y',p')\to(t,y,p)}}V^N(t',y',p').
\ee
By \eqref{lingrowth} both are well-defined and inherit the same growth
bound; in particular, the family $\{V^N\}$ is locally uniformly bounded
on every compact subset of $[0,T]\times\R^d\times\Sigma_K$. 
We know that $\underline{V}\le\overline{V}$ by construction, with
$\overline{V}$ upper- and $\underline{V}$ lower-semicontinuous. We
will show that $\overline{V}$ is a viscosity sub-solution and
$\underline{V}$ a viscosity super-solution of \eqref{correctPDE}, so
the comparison principle from Subsection~\ref{comp} gives the reverse
inequality $\overline{V}\le\underline{V}$, hence
$\overline{V}=\underline{V}=:V_\infty$, which is then the unique
viscosity solution of \eqref{correctPDE} and equals $V$ by
Theorem~\ref{mainpdes}. Combined with the monotonicity
$V^N\le V^{N+1}\le V$, this yields $V^N\uparrow V$ pointwise.

\vspace{2mm}
\nd
\textit{Step 1.} (Sub-solution property of $\overline{V}$)

\vspace{1mm}
\nd
\emph{1a. Interior.}\
Let $\varphi\in C^{1,2,2}([0,T)\times\R^d\times\Sigma_K)$ and
$(t_0,y_0,p_0)\in[0,T)\times\R^d\times\Sigma_K$ be such that
$\overline{V}-\varphi$ attains a strict local maximum at
$(t_0,y_0,p_0)$. By the standard half-relaxed-limit construction
\cite[Lemma~4.2]{barles1994solutions}, there exist sequences
$N_k\to\infty$ and $(t_k,y_k,p_k)\to(t_0,y_0,p_0)$ with $t_k<T$
eventually, such that $V^{N_k}-\varphi$ attains a local maximum at
$(t_k,y_k,p_k)$ and $V^{N_k}(t_k,y_k,p_k)\to\overline{V}(t_0,y_0,p_0)$.
The sub-solution property of $V^{N_k}$ for \eqref{NPDE} gives
\be\label{HscN_ub}
H^{N_k}_{sc}(\varphi)(t_k,y_k,p_k):=H^{N_k}_{sc}\bigl(t_k,y_k,p_k,\,\partial_t\varphi,\,D_{(p,y)}\varphi,\,
   D^2_{(p,y)}\varphi\bigr)\le 0.
\ee
If $-\lambda_{\max}(D^2_{pp}\varphi(t_0,y_0,p_0))>0$, then, by continuity, $D^2_{pp}\varphi$ is negative definite in a neighborhood of $(t_0,y_0,p_0)$. Since $D_yD_p\varphi$ remains bounded in that neighborhood, for all sufficiently large $k$ the supremum in $H^{N_k}_{sc}$ is attained in a fixed compact set independent of $k$.
Thus $H^{N_k}_{sc}\to H_{sc}$ pointwise at
$(t_0,y_0,p_0)$. Passing to the limit in \eqref{HscN_ub} yields
$H_{sc}(\varphi)(t_0,y_0,p_0)\le 0$,
and 
the sub-solution inequality 
$$H_{vi}(\varphi)(t_0,y_0,p_0)=\min\left\{H_{sc}(\varphi)(t_0,y_0,p_0),-\lambda_{\max}(D^2_{pp}\varphi(t_0y_0,p_0))\right\}\le 0$$
follows. If instead $-\lambda_{\max}(D^2_{pp}\varphi(t_0,y_0,p_0))\le 0$,
we directly get $H_{vi}(\varphi)(t_0,y_0,p_0)\le-\lambda_{\max}(D^2_{pp}\varphi)\le 0$.

\vspace{2mm}
\nd
\textit{Step 2.} (Super-solution property of $\underline{V}$)

\vspace{1mm}
\nd
\emph{2a. Interior.}\
Let $\varphi\in C^{1,2,2}([0,T)\times\R^d\times\Sigma_K)$ and
$(t_0,y_0,p_0)\in[0,T)\times\R^d\times\Sigma_K^\circ$ be such that
$\underline{V}-\varphi$ attains a strict local minimum at
$(t_0,y_0,p_0)$ (recall from Definition~\ref{def:stateconstrvisc}
that super-solution tests are performed only at interior simplex
points, by the state-constraint convention). We again obtain sequences $N_k\to\infty$ and
$[0,T)\times \R^d\times \Sigma_K^\circ\ni (t_k,y_k,p_k)\to(t_0,y_0,p_0)$ with $t_k<T$ eventually, such that
$V^{N_k}-\varphi$ attains a local minimum at $(t_k,y_k,p_k)$ and
$V^{N_k}(t_k,y_k,p_k)\to\underline{V}(t_0,y_0,p_0)$. The
super-solution property of $V^{N_k}$ for \eqref{NPDE} gives
\be\label{HscN_lb}
H^{N_k}_{sc}\bigl(t_k,y_k,p_k,\,\partial_t\varphi,\,D_{(p,y)}\varphi,\,
   D^2_{(p,y)}\varphi\bigr)\ge 0.
\ee
\emph{The eigenvalue obstacle is forced.}\
We claim $-\lambda_{\max}(D^2_{pp}\varphi(t_0,y_0,p_0))\ge 0$. Suppose
otherwise: $D^2_{pp}\varphi(t_0,y_0,p_0)$ admits a direction
$z^*\in\R^{K\times d}$ with $\sum_i z^{*,i}=0$, $|z^{*,i}|\le 1$, and
\[
\tfrac12\,\tr\bigl(z^*\sigma_\nu^{-2}(z^*)^{\!\top}D^2_{pp}\varphi(t_0,y_0,p_0)\bigr)=c>0.
\]
By continuity, this lower bound also holds on a neighborhood of
$(t_0,y_0,p_0)$. Selecting $z=N_k z^*$ (feasible in the constrained
sup for $N_k$ large, because $\max_i|z^{*,i}|\le1$), the supremum term in the definition of $H_{sc}^{N_k}$ is of order at least $cN_k^2-\mathcal O(N_k)$,
where the order-$N_k$ term comes from the linear cross term
$\tr(z(D_pD_y\varphi)^\top)$ and the remaining drift and running-cost
terms are independent of $z$. This forces $H^{N_k}_{sc}\to-\infty$, contradicting \eqref{HscN_lb}. We thus have $-\lambda_{\max}(D^2_{pp}\varphi(t_0,y_0,p_0))\ge 0$.
\vspace{1mm}

\nd
Under this negative semidefiniteness, the supremum in $H^{N_k}_{sc}$ is
uniformly bounded for $k$ large and $H^{N_k}_{sc}\to H_{sc}$
pointwise at $(t_0,y_0,p_0)$. Passing to the limit in
\eqref{HscN_lb}, $H_{sc}(\varphi)(t_0,y_0,p_0)\ge 0$. Combining this with
the obstacle inequality we just established,
\[
H_{vi}(\varphi)(t_0,y_0,p_0)
=\min\!\bigl\{H_{sc}(\varphi),\,-\lambda_{\max}(D^2_{pp}\varphi)\bigr\}(t_0,y_0,p_0)\ge 0.
\]
\vspace{1mm}

\nd
\emph{2b. Inequality up to the boundary.} Let $\varphi\in C^{1,2,2}([0,T)\times\R^d\times\Sigma_K)$ and
$(t_0,y_0,p_0)\in[0,T)\times\R^d\times\Sigma_K$ be such that
$\underline{V}-\varphi$ attains a strict local minimum at
$(t_0,y_0,p_0)$. We again obtain sequences $N_k\to\infty$ and
$[0,T)\times \R^d\times\Sigma_K\ni(t_k,y_k,p_k)\to(t_0,y_0,p_0)$ with $t_k<T$ eventually, such that
$V^{N_k}-\varphi$ attains a local minimum at $(t_k,y_k,p_k)$ and
$V^{N_k}(t_k,y_k,p_k)\to\underline{V}(t_0,y_0,p_0)$. If \eqref{HscN_lb} holds for infinitely many $k\in \mathbb{N}$ (that is if $p_k\in\Sigma_K^\circ$ for infinitely many $k$), then choosing $z=0$ in the supremum term of $H_{sc}^{N_k}$, we get
\be\label{ineq101}
\widetilde{\mathcal{B}_0}\big(t_k,y_k,p_k,\partial_t\varphi,D_{(p,y)}\varphi, D^2_{(p,y)}\varphi \big)\ge0,\;\;\text{ for infinitely many }k\in\mathbb{N}.
\ee
Otherwise, by the definition of state-constrained viscosity super solutions,  \eqref{ineq101} still holds. Letting $k\to +\infty$, we get \eqref{eq:bdsupersol}, as desired.
\vspace{2mm}

\nd
\textit{Step 3.} (Terminal condition, inequality \eqref{terminalassumption}.)\
\vspace{1mm}

\nd
\emph{Lower bound.}\ 
The admissible control $Z\equiv 0$ produces, from initial data
$(t',y',p')\in [0,T)\times \R^d\times\Sigma_K$, a trajectory $(Y_s,p_s)$ with $Y$ governed by $dY_s=b_\nu(s,Y_s)\,ds+\sigma_\nu(s,Y_s)\,dW_s$
and $p$ the deterministic flow $dp_s=p\Lambda\,ds$. By the optimal control formulation of $V^N$ we have
\[
V^N(t',y',p')\ge\E\!\Bigl[g(p_T,Y_T)+\int_t^T\!f(p_s,Y_s)\,ds\,\Big|\,Z\equiv 0\Bigr].
\]
Consider $(t,y,p)\in \R^d\times \Sigma_K$. Letting $(t',y',p')\to (t,y,p)$ and $N\to \infty$ in the above inequality, gives standard SDE
estimates and the linear growth of $f,g$ give
\[
\underline{V}(t,y,p)=\liminf_{\substack{N\to\infty\\(t',y',p')\to (t,y,p)}}V^N(t',y',p')\ge\E\!\Bigl[g(p_T,Y_T)+\int_{t}^T\!f(p_s,Y_s)\,ds\,\Big|\,Z\equiv 0\Bigr],
\]
where now the dynamics $(Y_s,p_s)$ start from $(y,p)$ with initial time $s=t$. Letting $(t,y,p)\to (T,y_0,p_0)\in [0,T]\times \R^d\times \Sigma_K$ in the above yields
\be\label{ineq102}
\liminf_{(t,y,p)\to (T,y_0,p_0)}\underline{V}(t,y,p)\ge g(p_0,y_0).
\ee
\vspace{1mm}

\nd
\emph{Upper bound.} By the control formulation of $V^N$ we know that $V^N(t,y,p)\le V(t,y,p)$, for any $(t,y,p)\in [0,T]\times \R^d\times \Sigma_K$, where $V$ is the continuous function constructed in Theorem \ref{mainpdes}. Thus, we have $\overline{V}(t,y,p)\le V(t,y,p)$. Letting $(t,y,p)\to (T,y_0,p_0)$ yields
\be\label{ineq103}
\limsup_{(t,y,p)\to (T,y_0,p_0)}\overline{V}(t,y,p)\le \lim_{(t,y,p)\to (T,y_0,p_0)}V(t,y,p)=g(p_0,y_0).
\ee

Inequality \eqref{terminalassumption} for the subsolution $\overline{V}$ and the supersolution $\underline{V}$ follows be subtracting \eqref{ineq102} and \eqref{ineq103}.
\end{proof}

\appendix

\section{Technical proofs}\label{sec:technical-proofs}\label{technical}
\nd
In this appendix we collect the longer technical proofs of
propositions stated in the main part of the paper.

\begin{proof}[Proof of Proposition~\ref{inclusions}]
Let $\pi\in\Pi_f(\mu,\nu)$. Under $\pi$, the dynamics of $Y$ are
given by \eqref{filteringcond}:
\[
dY_t = \bigl(h(t,Y_\cdot,X_t)+b_\nu(t,Y_t)\bigr)\,dt
       + \sigma_\nu(t,Y_t)\,dW_t^\pi,
\qquad Y_0\sim\nu_0.
\]

\vspace{1mm}
\nd
\textit{Step 1: The law of $Y$ under $\pi$ is $\nu$.}\
By Assumption~\ref{assumption1}, $\sigma_\nu(t,y)\succeq\kappa I_d$ for
some $\kappa>0$, so $\sigma_\nu(t,Y_t)$ is $\pi$-a.s.\ invertible for
every $t\in[0,T]$. Define the $\R^d$-valued innovation process
\[
\tilde W_t:=\int_0^t\sigma_\nu(s,Y_s)^{-1}
\bigl(dY_s-b_\nu(s,Y_s)\,ds\bigr).
\]
Substituting \eqref{filteringcond} gives
\be\label{eq:tildeW-decomp}
\tilde W_t = W^\pi_t + \int_0^t\sigma_\nu(s,Y_s)^{-1}\,h(s,Y_\cdot,X_s)\,ds,
\ee
so $\tilde W$ differs from $W^\pi$ by a continuous, bounded-variation
process. In particular, $\tilde W$ is continuous, $\cF^Y$-adapted
(both the integrand $\sigma_\nu(\cdot,Y_\cdot)^{-1}$ and the
integrator $Y$ are $\cF^Y$-adapted), and shares the quadratic
covariation of $W^\pi$:
\[
[\tilde W^i,\tilde W^j]_t = [W^{\pi,i},W^{\pi,j}]_t = \delta_{ij}\,t,
\qquad 1\le i,j\le d.
\]

We verify the martingale property. For $0\le s\le t\le T$ and each
component $\ell\in\{1,\ldots,d\}$, \eqref{eq:tildeW-decomp} gives
\[
\E^\pi\!\left[\tilde W^\ell_t-\tilde W^\ell_s\,\Big|\,\cF^Y_s\right]
=\E^\pi\!\left[(W^\pi_t-W^\pi_s)_\ell\,\Big|\,\cF^Y_s\right]
+\E^\pi\!\left[\int_s^t
\bigl(\sigma_\nu(r,Y_r)^{-1}h(r,Y_\cdot,X_r)\bigr)_\ell\,dr\,\Big|\,\cF^Y_s\right].
\]
Since $W^\pi$ is a $(\pi,\cF^{X,Y})$-Brownian motion and
$\cF^Y_s\subset\cF^{X,Y}_s$, the first term vanishes by the tower
property. For the second, applying the tower property with
intermediate $\sigma$-algebra $\cF^Y_r$, using the
$\cF^Y_r$-measurability of $\sigma_\nu(r,Y_r)^{-1}$ and the filtering
condition \eqref{hcond},
\[
\E^\pi\!\left[\bigl(\sigma_\nu(r,Y_r)^{-1}h(r,Y_\cdot,X_r)\bigr)_\ell
\,\Big|\,\cF^Y_s\right]
=\E^\pi\!\left[\bigl(\sigma_\nu(r,Y_r)^{-1}
\,\E^\pi[h(r,Y_\cdot,X_r)\,|\,\cF^Y_r]\bigr)_\ell\,\Big|\,\cF^Y_s\right]
=0
\]
for every $r\in[s,t]$, so Fubini gives that the entire integral
vanishes. Hence $\E^\pi[\tilde W^\ell_t-\tilde W^\ell_s\,|\,\cF^Y_s]=0$
for every $\ell$, and by L\'evy's characterization $\tilde W$ is a
$d$-dimensional $(\pi,\cF^Y)$-Brownian motion.

Substituting into \eqref{filteringcond}, $Y$ satisfies the
$\cF^Y$-adapted SDE
\[
dY_t = b_\nu(t,Y_t)\,dt + \sigma_\nu(t,Y_t)\,d\tilde W_t,
\qquad Y_0\sim\mu.
\]
Since $b_\nu$ and $\sigma_\nu$ are Lipschitz
(Assumption~\ref{assumption1}), this SDE has a pathwise unique strong
solution, and the law of $Y$ under $\pi$ coincides with $\nu$. In
particular, $\pi\in\Pi(\mu,\nu)$.

\vspace{2mm}
\nd
\textit{Step 2: $\pi$ is causal.}\
By Proposition~\ref{causalprop}, causality is equivalent to the
conditional-independence condition
\be\label{eq:cond-ind}
\pi\bigl(\Omega^X\times A\,\big|\,\cF^X_t\otimes\{\emptyset,\Omega^Y\}\bigr)
=\pi\bigl(\Omega^X\times A\,\big|\,\cF^X_T\otimes\{\emptyset,\Omega^Y\}\bigr)
\quad\text{for all }A\in\cF^Y_t,\;t\in[0,T].
\ee
By the definition of a filtering coupling, the driving noise $W^\pi$
is independent of $X$ under $\pi$. The SDE \eqref{filteringcond}
expresses $Y$ as a strong functional of $W^\pi$ and the path
$X_{[0,T]}$, with the $X$-dependence at time $s$ entering through
$(X_u)_{u\le s}$ alone (non-anticipativity of $h$). Consequently every
$A\in\cF^Y_t$ is measurable with respect to
$\sigma\bigl(X_{[0,t]},W^\pi_{[0,t]}\bigr)$, and the conditional
distribution of $A$ given $\cF^X_t$ — equivalently, the joint law of
$W^\pi_{[0,t]}$ given $X_{[0,t]}$ — does not change upon enlarging the
conditioning to $\cF^X_T$, since $W^\pi$ is independent of
$X_{(t,T]}$. Hence \eqref{eq:cond-ind} holds and
$\pi\in\Pi_c(\mu,\nu)$.

The inequality $V_f(\mu,\nu)\le V_c(\mu,\nu)$ is then immediate from
$\Pi_f(\mu,\nu)\subset\Pi_c(\mu,\nu)$.
\end{proof}

\begin{proof}[Proof of Proposition~\ref{lem:forward}]
Fix $i\in\mathbb S=\{1,\ldots,K\}$. For any bounded
$\phi\colon\mathbb S\to\R^K$, write $\phi_j\colon\mathbb S\to\R$ for
its $j$-th component, and define
\[
M^\phi_t:=\phi_i(X_t)-\int_0^t\bigl(\phi(X_s)\La\bigr)^i ds,
\qquad t\in[0,T],
\]
where $\phi(X_s)\La$ is the row-vector–matrix product, so
$(\phi(X_s)\La)^i=\sum_j\phi_j(X_s)\La_{ji}$. By the martingale
problem for the Markov chain $X$, $M^\phi$ is a bounded
$(\pi,\cF^X_t\otimes\{\emptyset,\Omega^Y\})$-martingale, and by the
$H$-hypothesis from Proposition~\ref{causalprop} it is also a
$(\pi,\cF^X_t\otimes\cF^Y_t)$-martingale. Hence, for
$0\le s\le t\le T$,
\[
\E^\pi\!\left[\E^\pi[M^\phi_t\,|\,\cF^Y_t]\,\big|\,\cF^Y_s\right]
=\E^\pi[M^\phi_t\,|\,\cF^Y_s]
=\E^\pi\!\left[\E^\pi[M^\phi_t\,|\,\cF^X_s\otimes\cF^Y_s]\,\big|\,\cF^Y_s\right]
=\E^\pi[M^\phi_s\,|\,\cF^Y_s],
\]
so $\tilde M^\phi_t:=\E^\pi[M^\phi_t\,|\,\cF^Y_t]$ is a
$(\pi,\cF^Y)$-martingale.

Since $\sigma_\nu\succeq\kappa I_d$, the filtration $\cF^Y$ is
generated by the $d$-dimensional $(\pi,\cF^Y)$-Brownian motion
$\tilde W$ constructed in the proof of Proposition~\ref{inclusions}.
By the $d$-dimensional martingale representation theorem, there
exists a square-integrable $\cF^Y$-adapted process
$\tilde Z^\phi\colon[0,T]\times\Omega\to\R^d$ such that
\[
\E^\pi[\phi_i(X_t)\,|\,\cF^Y_t]
=\E^\pi[\phi_i(X_0)]
+\E^\pi\!\left[\int_0^t\bigl(\phi(X_s)\La\bigr)^i ds\,\Big|\,\cF^Y_t\right]
+\int_0^t(\tilde Z^\phi_s)^{\!\top}d\tilde W_s
\quad\text{for all }t\in[0,T].
\]
Splitting the conditional integral via Fubini and writing
\[
\E^\pi\!\left[\bigl(\phi(X_s)\La\bigr)^i\,\big|\,\cF^Y_t\right]
=\E^\pi\!\left[\bigl(\phi(X_s)\La\bigr)^i\,\big|\,\cF^Y_s\right]
+\Bigl(\E^\pi\!\left[\bigl(\phi(X_s)\La\bigr)^i\,\big|\,\cF^Y_t\right]
-\E^\pi\!\left[\bigl(\phi(X_s)\La\bigr)^i\,\big|\,\cF^Y_s\right]\Bigr),
\]
the second piece integrates to a $(\pi,\cF^Y)$-martingale (by the
same tower-property argument as for $\tilde M^\phi$) and can be
absorbed into the $\tilde W$-integral. This yields the
Kushner--Stratonovich-type representation
\be\label{eq:representation}
\E^\pi[\phi_i(X_t)\,|\,\cF^Y_t]
=\E^\pi[\phi_i(X_0)]
+\int_0^t\E^\pi\!\left[\bigl(\phi(X_s)\La\bigr)^i\,\big|\,\cF^Y_s\right]ds
+\int_0^t(\tilde Z^\phi_s)^{\!\top}d\tilde W_s.
\ee

Set $Z^i_s:=\sigma_\nu(s,Y_s)\,\tilde Z^\phi_s\in\R^d$. Since
$\sigma_\nu^{\!\top}=\sigma_\nu$ and
$d\tilde W_s=\sigma_\nu(s,Y_s)^{-1}\bigl(dY_s-b_\nu(s,Y_s)\,ds\bigr)$,
\[
(\tilde Z^\phi_s)^{\!\top}d\tilde W_s
=(Z^i_s)^{\!\top}\sigma_\nu(s,Y_s)^{-2}
\bigl(dY_s-b_\nu(s,Y_s)\,ds\bigr).
\]
Specializing \eqref{eq:representation} to the indicator vector
$\phi(x):=(\1_{x=j})_{j\in\mathbb S}$, so that
$\phi_i(X_t)=\1_{X_t=i}$ and $(\phi(X_s)\La)^i=\La_{X_s,i}$, and
recalling that $p^i_t:=\E^\pi[\1_{X_t=i}\,|\,\cF^Y_t]$ satisfies
$\E^\pi[\La_{X_s,i}\,|\,\cF^Y_s]=\sum_j p^j_s\La_{ji}=(p_s\La)^i$,
we obtain \eqref{eq:dynfiltercot} and \eqref{eq:initialcot}.

\vspace{1mm}
\nd
\emph{Constraint $\sum_{j=1}^K Z^j_t=0$.}\
Since $\sum_j p^j_t=\E^\pi[\sum_j\1_{X_t=j}\,|\,\cF^Y_t]=1$
identically, summing \eqref{eq:dynfiltercot} over $j$ gives
\[
0=d\!\left(\textstyle\sum_j p^j_t\right)
=\textstyle\sum_j(p_t\La)^j\,dt
+\bigl(\sum_j Z^j_t\bigr)^{\!\top}\sigma_\nu^{-2}(t,Y_t)
\bigl(dY_t-b_\nu(t,Y_t)\,dt\bigr).
\]
The drift coefficient $\sum_j(p_t\La)^j=\sum_j\sum_k p^k_t\La_{kj}
=\sum_k p^k_t\sum_j\La_{kj}=0$ vanishes since the rows of the
generator $\La$ sum to zero. Identifying the diffusion part and
using the invertibility of $\sigma_\nu^{-2}$ (Assumption~\ref{assumption1}),
$\sum_{j=1}^K Z^j_t=0\in\R^d$, $dt\otimes d\pi$-a.s.

\vspace{1mm}
\nd
\emph{Constraint $\1_{p^i_t=0}Z^i_t=0$.}\
Since $p^i_t\ge 0$, for every $t$, we have $(p^i_t)^+=p^i_t$. Applying the
Itô–Tanaka formula to the convex function $x\mapsto x^+$ and using
\eqref{eq:dynfiltercot},
\[
d(p^i_t)^+
=\1_{p^i_t>0}\,dp^i_t+\tfrac12\,d\ell^0_t
=\1_{p^i_t>0}\!\left[(p_t\La)^i\,dt
+(Z^i_t)^{\!\top}\sigma_\nu^{-2}(t,Y_t)
\bigl(dY_t-b_\nu(t,Y_t)\,dt\bigr)\right]+\tfrac12\,d\ell^0_t,
\]
where $\ell^0$ is the local time of $p^i$ at $0$. Since
$d(p^i_t)^+=dp^i_t$, we obtain
\[
\1_{p^i_t=0}\!\left[(p_t\La)^i\,dt
+(Z^i_t)^{\!\top}\sigma_\nu^{-2}(t,Y_t)
\bigl(dY_t-b_\nu(t,Y_t)\,dt\bigr)\right]
=\tfrac12\,d\ell^0_t.
\]
The right-hand side is a continuous non-decreasing process supported
on $\{p^i_t=0\}$, in particular of bounded variation. Identifying the
martingale (diffusion) part on the left, which is driven by
$dY_t-b_\nu\,dt=\sigma_\nu\,d\tilde W_t$ and has coefficient
$\1_{p^i_t=0}(Z^i_t)^{\!\top}\sigma_\nu^{-1}$, with $0$ on the right,
\[
\1_{p^i_t=0}(Z^i_t)^{\!\top}\sigma_\nu(t,Y_t)^{-1}=0,
\]
and since $\sigma_\nu^{-1}$ is invertible, $\1_{p^i_t=0}Z^i_t=0\in\R^d$,
$dt\otimes d\pi$-a.s.
\end{proof}

\begin{proof}[Proof of Lemma~\ref{invariance}]
Since $\sum_i p^i_t=1$ holds by \eqref{eq:dynfiltercot}, it suffices
to show $p^i_t\ge 0$ for all $i$. Define
$V_t:=\sum_{i=1}^K(p^i_t)^-\ge 0$. Since $p_0$ is a probability
vector, $V_0=0$. Applying the It\^o--Tanaka formula to each
$(p^i_t)^-$ gives
\[
(p^i_t)^- = -\int_0^t\mathbf{1}_{p^i_s<0}\,dp^i_s
+\tfrac{1}{2}\int_0^t\mathbf{1}_{p^i_s=0}\|Z^i_s\|^2_{\sigma_\nu^{-2}}\,ds,
\]
where $\|v\|^2_{\sigma_\nu^{-2}}:=v^\top\sigma_\nu^{-2}v$.
The second term is the local time contribution at $0$; by
\eqref{Zcond} we have $Z^i_s=0$ on $\{p^i_s=0\}$, so it vanishes.
Substituting
$dp^i_s=(p_s\La)^i\,ds+(Z^i_s)^\top\sigma_\nu^{-2}(s,Y_s)\bigl(dY_s-b_\nu(s,Y_s)\,ds\bigr)$,
summing over $i$, and taking expectations (the stochastic integrals
are martingales since $Z\in L^2$ and $dY_s-b_\nu(s,Y_s)\,ds=\sigma_\nu(s,Y_s)\,dW_s$
under $\nu$) yields
\[
\E[V_t]
= -\,\E\!\left[\int_0^t\sum_i\mathbf{1}_{p^i_s<0}(p_s\La)^i\,ds\right].
\]
We show the integrand is non-positive. For fixed $s$, write
\[
-\sum_i\mathbf{1}_{p^i_s<0}(p_s\La)^i
= -\sum_j p^j_s\!\sum_{i:\,p^i_s<0}\La_{ji}.
\]
If $p^j_s\ge 0$, then $\sum_{i:\,p^i_s<0}\La_{ji}=\sum_{i\neq
j,\,p^i_s<0}\La_{ji}\ge 0$ since $\La_{ji}\ge 0$ for $i\neq j$,
giving a non-positive contribution. If $p^j_s<0$, then using
$\sum_i\La_{ji}=0$,
\[
\sum_{i:\,p^i_s<0}\La_{ji}
= -\!\sum_{i:\,p^i_s\ge 0}\La_{ji}
= -\La_{jj}-\!\sum_{i\neq j,\,p^i_s\ge 0}\La_{ji}\le 0,
\]
and since $-p^j_s>0$, the contribution is again non-positive. Hence
$\E[V_t]\le 0$, and since $V_t\ge 0$ we conclude $V_t=0$ a.s.\ for
each $t$, i.e.\ $p^i_t\ge 0$ for all $i$, $dt\otimes\nu$-a.e.
\end{proof}

\bibliographystyle{alpha}
\bibliography{sample}

@article{BayraktarHanFVI2025,
  author  = {Bayraktar, Erhan and Han, Bingyan},
  title   = {Fitted Value Iteration Methods for Bicausal Optimal Transport},
  journal = {Applied Mathematics \& Optimization},
  volume  = {92},
  number  = {1},
  pages   = {15},
  year    = {2025},
  doi     = {10.1007/s00245-025-10283-1}
}

@article{BayraktarCoxStoev2018,
  author  = {Bayraktar, Erhan and Cox, Alexander M. G. and Stoev, Yavor},
  title   = {Martingale Optimal Transport with Stopping},
  journal = {SIAM Journal on Control and Optimization},
  volume  = {56},
  number  = {1},
  pages   = {417--433},
  year    = {2018},
  doi     = {10.1137/17M1114065}
}

@article{BayraktarHanGoal2025,
  author        = {Bayraktar, Erhan and Han, Bingyan},
  title         = {Goal-based portfolio selection with mental accounting},
  journal       = {arXiv preprint arXiv:2506.06654},
  year          = {2025},
  eprint        = {2506.06654},
  archivePrefix = {arXiv},
  primaryClass  = {q-fin.PM}
}

@article{BayraktarZhang2015ruin,
  author  = {Bayraktar, Erhan and Zhang, Yuchong},
  title   = {Stochastic Perron's Method for the Probability of Lifetime Ruin Problem Under Transaction Costs},
  journal = {SIAM Journal on Control and Optimization},
  volume  = {53},
  number  = {1},
  pages   = {91--113},
  year    = {2015},
  doi     = {10.1137/140967052}
}

@article{Kar81,
  author = {Karatzas, Ioannis},
  title = {The monotone follower problem in stochastic decision theory},
  journal = {Applied Mathematics and Optimization},
  volume = {7},
  number = {2},
  pages = {175--189},
  year = {1981}
}

@article{He96,
  title={Sequential decisions under uncertainty and the maximum theorem},
  author={Hellwig, Martin F},
  journal={Journal of Mathematical Economics},
  volume={25},
  number={4},
  pages={443--464},
  year={1996},
  publisher={Elsevier}
}

@article{KS84,
  author = {Karatzas, Ioannis and Shreve, Steven E.},
  title = {Connections between optimal stopping and singular stochastic control I. Monotone follower problems},
  journal = {SIAM Journal on Control and Optimization},
  volume = {22},
  number = {6},
  pages = {856--877},
  year = {1984}
}

@article{KS85,
  author = {Karatzas, Ioannis and Shreve, Steven E.},
  title = {Connections between optimal stopping and singular stochastic control II. Reflected follower problems},
  journal = {SIAM Journal on Control and Optimization},
  volume = {23},
  number = {3},
  pages = {433--451},
  year = {1985}
}

@article{BSW80,
  author = {Bene\v{s}, V. E. and Shepp, L. A. and Witsenhausen, H. S.},
  title = {Some solvable stochastic control problems},
  journal = {Stochastics},
  volume = {4},
  number = {1},
  pages = {39--83},
  year = {1980}
}

@article{bayraktar-li2016target,
  author  = {Bayraktar, Erhan and Li, Jiaqi},
  title   = {Stochastic {P}erron for Stochastic Target Problems},
  journal = {Journal of Optimization Theory and Applications},
  volume  = {170},
  number  = {3},
  pages   = {1026--1054},
  year    = {2016},
  doi     = {10.1007/s10957-016-0958-2}
}

@article{bayraktar-li2016games,
  author  = {Bayraktar, Erhan and Li, Jiaqi},
  title   = {Stochastic {P}erron for Stochastic Target Games},
  journal = {The Annals of Applied Probability},
  volume  = {26},
  number  = {2},
  pages   = {1082--1110},
  year    = {2016},
  doi     = {10.1214/15-AAP1112}
}

@article{bayraktar-sirbu2012linear,
  author  = {Bayraktar, Erhan and S{\^\i}rbu, Mihai},
  title   = {Stochastic Perron's Method and Verification without Smoothness Using Viscosity Comparison: The Linear Case},
  journal = {Proceedings of the American Mathematical Society},
  volume  = {140},
  number  = {10},
  pages   = {3645--3654},
  year    = {2012},
  doi     = {10.1090/S0002-9939-2012-11336-X}
}

@article{bayraktar-sirbu2014obstacle,
  author  = {Bayraktar, Erhan and S{\^\i}rbu, Mihai},
  title   = {Stochastic {P}erron's Method and Verification without Smoothness Using Viscosity Comparison: Obstacle Problems and Dynkin Games},
  journal = {Proceedings of the American Mathematical Society},
  volume  = {142},
  number  = {4},
  pages   = {1399--1412},
  year    = {2014},
  doi     = {10.1090/S0002-9939-2014-11860-0}
}

@article{KuDi78,
  title={Approximations for functionals and optimal control problems on jump diffusion processes},
  author={Kushner, Harold J and DiMasi, Giovanni},
  journal={Journal of Mathematical Analysis and Applications},
  volume={63},
  number={3},
  pages={772--800},
  year={1978},
  publisher={Elsevier}
}

@book{Ka01,
  title={Foundations of modern probability, second edition},
  author={Kallenberg, Olav},
  year={2001},
  publisher={Springer}
}

@article{gangbo1996geometry,
  title={The geometry of optimal transportation},
  author={Gangbo, Wilfrid and McCann, Robert J},
  journal={Acta Mathematica},
  volume={177},
  number={2},
  pages={113--161},
  year={1996},
  publisher={Springer}
}

@article{back2020optimal,
  title={Optimal Transport and Risk Aversion in {K}yle's Model of Informed Trading},
  author={Back, Kerry and Cocquemas, Francois and Ekren, Ibrahim and Lioui, Abraham},
  journal={arXiv e-prints},
  pages={arXiv--2006},
  year={2020}
}

@book{santambrogio2015optimal,
  title={Optimal Transport for Applied Mathematicians},
  author={Santambrogio, Filippo},
  publisher={Birk{\"a}user},
  address={New York},
  year={2015}
}

@article{katsoulakis1994viscosity,
  title={Viscosity solutions of second order fully nonlinear elliptic equations with state constraints},
  author={Katsoulakis, Markos A},
  journal={Indiana University Mathematics Journal},
  pages={493--519},
  year={1994},
  publisher={JSTOR}
}

@article{rokhlin2014stochastic,
  title={Stochastic {P}erron's method for optimal control problems with state constraints},
  author={Rokhlin, Dmitry B},
  journal={Electronic Communications in Probability},
  volume={19},
  pages={1--15},
  year={2014}
}

@article{soner1986optimal,
  title={Optimal control with state-space constraint I},
  author={Soner, Halil Mete},
  journal={SIAM Journal on Control and Optimization},
  volume={24},
  number={3},
  pages={552--561},
  year={1986},
  publisher={SIAM}
}

@book{ambrosio2005gradient,
  title={Gradient flows: in metric spaces and in the space of probability measures},
  author={Ambrosio, Luigi and Gigli, Nicola and Savar{\'e}, Giuseppe},
  year={2005},
  publisher={Springer}
}

@article{first_paper,
  title={A unified approach to informed trading via {M}onge-{K}antorovich duality},
  author={Chhaibi, Reda and Ekren, Ibrahim and Noh, Eunjung and Vy, Lu},
  journal={https://arxiv.org/pdf/2210.17384.pdf},
  year={2022}
}

@article{bose2021multidimensional,
  title={Multidimensional {K}yle-{B}ack model with a risk averse informed trader},
  author={Bose, Shreya and Ekren, Ibrahim},
  journal={arXiv preprint arXiv:2111.01957},
  year={2021}
}

@article{bose2020kyle,
  title={Kyle-{B}ack Models with risk aversion and non-{G}aussian Beliefs},
  author={Bose, Shreya and Ekren, Ibrahim},
  journal={arXiv preprint arXiv:2008.06377},
  year={2020}
}

@article{b,
  title={Polar factorization and monotone rearrangement of vector-valued functions},
  author={Brenier, Yann},
  journal={Communications on Pure and Applied Mathematics},
  volume={44},
  number={4},
  pages={375--417},
  year={1991},
  publisher={Wiley Online Library}
}

@article{cho,
  title={Continuous auctions and insider trading: uniqueness and risk aversion},
  author={Cho, Kyung-Ha},
  journal={Finance and Stochastics},
  volume={7},
  number={1},
  pages={47--71},
  year={2003},
  publisher={Springer}
}

@incollection{LiPa25,
  title={Adapted Topologies and Higher-Rank Signatures},
  author={Liu, Chong and Pammer, Gudmund},
  booktitle={Signature Methods in Finance: An Introduction with Computational Applications},
  pages={333--380},
  year={2025},
  publisher={Springer}
}

@article{BoLiOb23,
  title={Adapted topologies and higher rank signatures},
  author={Bonnier, Patric and Liu, Chong and Oberhauser, Harald},
  journal={The Annals of Applied Probability},
  volume={33},
  number={3},
  pages={2136--2175},
  year={2023},
  publisher={Institute of Mathematical Statistics}
}

@article{AcKrPa14,
  title={Designing universal causal deep learning models: The geometric (hyper) transformer},
  author={Acciaio, Beatrice and Kratsios, Anastasis and Pammer, Gudmund},
  journal={Mathematical Finance},
  volume={34},
  number={2},
  pages={671--735},
  year={2024},
  publisher={Wiley Online Library}
}

@article{AcHoPa25,
  title={Entropic adapted Wasserstein distance on Gaussians},
  author={Acciaio, Beatrice and Hou, Songyan and Pammer, Gudmund},
  journal={Electronic Communications in Probability},
  volume={30},
  pages={1--14},
  year={2025},
  publisher={The Institute of Mathematical Statistics and the Bernoulli Society}
}

@article{AcBaGrHoPa26,
  title={The geometry of the adapted Bures--Wasserstein space},
  author={Acciaio, Beatrice and Bartl, Daniel and Grass, Anne and Hou, Songyan and Pammer, Gudmund},
  journal={arXiv preprint arXiv:2602.00623},
  year={2026}
}

@article{Pa24,
  title={A note on the adapted weak topology in discrete time},
  author={Pammer, Gudmund},
  journal={Electronic Communications in Probability},
  volume={29},
  pages={1--13},
  year={2024},
  publisher={The Institute of Mathematical Statistics and the Bernoulli Society}
}

@article{Ji24,
  title={Duality of causal distributionally robust optimization},
  author={Jiang, Yifan},
  journal={arXiv preprint arXiv:2401.16556},
  year={2024}
}

@article{JiLi25,
  title={A transfer principle for computing the adapted Wasserstein distance between stochastic processes},
  author={Jiang, Yifan and Lim, Fang Rui},
  journal={arXiv preprint arXiv:2505.21337},
  year={2025}
}

@article{SaLeLiHoDaLy21,
  title={Higher order kernel mean embeddings to capture filtrations of stochastic processes},
  author={Salvi, Cristopher and Lemercier, Maud and Liu, Chong and Horvath, Blanka and Damoulas, Theodoros and Lyons, Terry},
  journal={Advances in Neural Information Processing Systems},
  volume={34},
  pages={16635--16647},
  year={2021}
}

@article{ChVi25,
  title={Robust Optimization in Causal Models and G-Causal Normalizing Flows},
  author={Visentin, Gabriele and Cheridito, Patrick},
  journal={arXiv preprint arXiv:2510.15458},
  year={2025}
}

@article{ChEc25,
  title={Optimal transport and Wasserstein distances for causal models},
  author={Cheridito, Patrick and Eckstein, Stephan},
  journal={Bernoulli},
  volume={31},
  number={2},
  pages={1351--1376},
  year={2025},
  publisher={Bernoulli Society for Mathematical Statistics and Probability}
}

@article{Ho87,
  title={A characterization of adapted distribution},
  author={Hoover, Douglas N},
  journal={The Annals of Probability},
  pages={1600--1611},
  year={1987},
  publisher={JSTOR}
}

@article{GuWo25,
  title={Adapted optimal transport between Gaussian processes in discrete time},
  author={Gunasingam, Madhu and Leonard Wong, Ting-Kam},
  journal={Electronic Communications in Probability},
  volume={30},
  pages={1--14},
  year={2025},
  publisher={The Institute of Mathematical Statistics and the Bernoulli Society}
}

@article{xu2020cot,
  title={Cot-gan: Generating sequential data via causal optimal transport},
  author={Xu, Tianlin and Wenliang, Li Kevin and Munn, Michael and Acciaio, Beatrice},
  journal={Advances in neural information processing systems},
  volume={33},
  pages={8798--8809},
  year={2020}
}

@article{pardoux1982equations,
  title={{\'E}quations du filtrage non lin{\'e}aire de la pr{\'e}diction et du lissage},
  author={Pardoux, Etienne},
  journal={Stochastics},
  volume={6},
  number={3-4},
  pages={193--231},
  year={1982},
  publisher={Taylor \& Francis}
}

@inproceedings{lions1989viscosity2,
  title={Viscosity solutions of fully nonlinear second order equations and optimal stochastic control in infinite dimensions. Part II: Optimal control of {Z}akai's equation},
  author={Lions, Pierre-Louis},
  booktitle={Stochastic Partial Differential Equations and Applications II: Proceedings of a Conference held in Trento, Italy February 1--6, 1988},
  pages={147--170},
  year={1989},
  organization={Springer}
}

@article{lions1989viscosity3,
  title={Viscosity solutions of fully nonlinear second-order equations and optimal stochastic control in infinite dimensions. III. Uniqueness of viscosity solutions for general second-order equations},
  author={Lions, Pierre-Louis},
  journal={Journal of Functional Analysis},
  volume={86},
  number={1},
  pages={1--18},
  year={1989},
  publisher={Elsevier}
}

@article{bayraktar2026comparison,
  title={Comparison for semi-continuous viscosity solutions for second order PDEs on the {W}asserstein space},
  author={Bayraktar, Erhan and Ekren, Ibrahim and He, Xihao and Zhang, Xin},
  journal={Journal of Differential Equations},
  volume={455},
  pages={113963},
  year={2026},
  publisher={Elsevier}
}

@article{bayraktar2025comparison,
  title={Comparison of viscosity solutions for a class of second-order PDEs on the {W}asserstein space},
  author={Bayraktar, Erhan and Ekren, Ibrahim and Zhang, Xin},
  journal={Communications in Partial Differential Equations},
  volume={50},
  number={4},
  pages={570--613},
  year={2025},
  publisher={Taylor \& Francis}
}

@article{ishii2002class,
  title={A class of stochastic optimal control problems with state constraint},
  author={Ishii, Hitoshi and Loreti, Paola},
  journal={Indiana University Mathematics Journal},
  pages={1167--1196},
  year={2002},
  publisher={JSTOR}
}

@article{bartl2025wasserstein,
  title={The {W}asserstein Space of Stochastic Processes in Continuous Time},
  author={Bartl, Daniel and Beiglb{\"o}ck, Mathias and Pammer, Gudmund and Schrott, Stefan and Zhang, Xin},
  journal={arXiv preprint arXiv:2501.14135},
  year={2025}
}

@article{gozzi2000hamilton,
  title={Hamilton--{J}acobi--{B}ellman equations for the optimal control of the {D}uncan--{M}ortensen--{Z}akai equation},
  author={Gozzi, Fausto and {\'S}wiech, Andrzej},
  journal={Journal of Functional Analysis},
  volume={172},
  number={2},
  pages={466--510},
  year={2000},
  publisher={Elsevier}
}

@article{AlvarezLasryLions1997,
  title={Convex viscosity solutions and state constraints},
  author={Alvarez, Olivier and Lasry, J-M and Lions, P-L},
  journal={Journal de Math{\'e}matiques Pures et Appliqu{\'e}es},
  volume={76},
  number={3},
  pages={265--288},
  year={1997},
  publisher={Elsevier}
}

@article{martini2023kolmogorov,
  title={Kolmogorov equations on spaces of measures associated to nonlinear filtering processes},
  author={Martini, Mattia},
  journal={Stochastic Processes and their Applications},
  volume={161},
  pages={385--423},
  year={2023},
  publisher={Elsevier}
}

@article{k,
  title={On the maximum principle for quasi-linear parabolic equations of the second order},
  author={Kusano, Takasi},
  journal={Proceedings of the Japan Academy},
  volume={39},
  number={4},
  pages={211--216},
  year={1963},
  publisher={The Japan Academy}
}

@article{kyle,
  title={Continuous auctions and insider trading},
  author={Kyle, Albert S},
  journal={Econometrica},
  volume={53},
  number={6},
  pages={1315--1335},
  year={1985},
  publisher={JSTOR}
}

@book{l,
  title={Second order parabolic differential equations},
  author={Lieberman, Gary M},
  year={1996},
  publisher={World Scientific}
}

@book{liptser1977statistics,
  title={Statistics of random processes: General theory},
  author={Liptser, Robert Shevilevich and Shiriaev, Al'bert Nikolaevich},
  volume={394},
  year={1977},
  publisher={Springer}
}

@article{m,
  title={Existence and uniqueness of monotone measure-preserving maps},
  author={McCann, Robert J},
  journal={Duke Mathematical Journal},
  volume={80},
  number={2},
  pages={309--323},
  year={1995},
  publisher={Duke University Press}
}

@phdthesis{s,
  title={Study of new models for insider trading and impulse control},
  author={Shi, Pucheng},
  year={2013},
  school={London School of Economics and Political Science (LSE)}
}

@article{bion-nadal2019wasserstein,
  author  = {Bion-Nadal, Jocelyne and Talay, Denis},
  title   = {On a {W}asserstein-type distance between solutions to stochastic differential equations},
  journal = {Annals of Applied Probability},
  volume  = {29},
  number  = {3},
  pages   = {1609--1639},
  year    = {2019},
  doi     = {10.1214/18-AAP1423}
}

@book{kushner-dupuis2001,
  author    = {Kushner, Harold J. and Dupuis, Paul},
  title     = {Numerical Methods for Stochastic Control Problems in Continuous Time},
  edition   = {2nd},
  series    = {Stochastic Modelling and Applied Probability},
  volume    = {24},
  publisher = {Springer},
  year      = {2001}
}

@article{barles-souganidis1991,
  author  = {Barles, Guy and Souganidis, Panagiotis E.},
  title   = {Convergence of approximation schemes for fully nonlinear second order equations},
  journal = {Asymptotic Analysis},
  volume  = {4},
  year    = {1991},
  pages   = {271--283}
}

@article{fahim-touzi-warin2011,
  author  = {Fahim, Arash and Touzi, Nizar and Warin, Xavier},
  title   = {A probabilistic numerical method for fully nonlinear parabolic {PDE}s},
  journal = {Annals of Applied Probability},
  volume  = {21},
  number  = {4},
  year    = {2011},
  pages   = {1322--1364}
}

@article{gobet-lemor-warin2005,
  author  = {Gobet, Emmanuel and Lemor, Jean-Philippe and Warin, Xavier},
  title   = {A regression-based {M}onte {C}arlo method to solve backward stochastic differential equations},
  journal = {Annals of Applied Probability},
  volume  = {15},
  number  = {3},
  year    = {2005},
  pages   = {2172--2202}
}

@article{han-jentzen-e2018,
  author  = {Han, Jiequn and Jentzen, Arnulf and E, Weinan},
  title   = {Solving high-dimensional partial differential equations using deep learning},
  journal = {Proceedings of the National Academy of Sciences},
  volume  = {115},
  number  = {34},
  year    = {2018},
  pages   = {8505--8510}
}

@article{hure-pham-warin2020,
  author  = {Hur{\'e}, C{\^o}me and Pham, Huy{\^e}n and Warin, Xavier},
  title   = {Deep backward schemes for high-dimensional nonlinear {PDE}s},
  journal = {Mathematics of Computation},
  volume  = {89},
  number  = {324},
  year    = {2020},
  pages   = {1547--1579}
}

@article{sirignano-spiliopoulos2018,
  author  = {Sirignano, Justin and Spiliopoulos, Konstantinos},
  title   = {{DGM}: A deep learning algorithm for solving partial differential equations},
  journal = {Journal of Computational Physics},
  volume  = {375},
  year    = {2018},
  pages   = {1339--1364}
}

@article{cont-lim2024causal,
  author  = {Cont, Rama and Lim, Fang Rui},
  title   = {Causal transport on path space},
  journal = {arXiv preprint arXiv:2412.02948},
  year    = {2024}
}

@article{constantinou2017extended,
  author  = {Constantinou, Panayiota and Dawid, A. Philip},
  title   = {Extended conditional independence and applications in causal inference},
  journal = {The Annals of Statistics},
  year    = {2017},
  volume  = {45},
  number  = {6},
  pages   = {2618--2653},
  doi     = {10.1214/16-AOS1537}
}

@article{bandini-cosso-fuhrman-pham2019,
  author  = {Bandini, Elena and Cosso, Andrea and Fuhrman, Marco and Pham, Huy{\^e}n},
  title   = {Randomized filtering and {B}ellman equation in {W}asserstein space for partial observation control problem},
  journal = {Stochastic Processes and their Applications},
  volume  = {129},
  number  = {2},
  pages   = {674--711},
  year    = {2019}
}

@book{sv,
  title={Multidimensional diffusion processes},
  author={Stroock, Daniel W and Varadhan, S. R. Srinivasa},
  volume={233},
  year={1997},
  publisher={Springer Science \& Business Media}
}

@book{villani,
  title={Optimal transport: old and new},
  author={Villani, C{\'e}dric},
  volume={338},
  year={2009},
  publisher={Springer}
}

@article{acciaio2020causal,
  title={Causal optimal transport and its links to enlargement of filtrations and continuous-time stochastic optimization},
  author={Acciaio, Beatrice and Backhoff, Julio and Zalashko, Anastasiia},
  journal={Stochastic Processes and their Applications},
  volume={130},
  number={5},
  pages={2918--2953},
  year={2020},
  publisher={Elsevier}
}

@article{aldous1981weak,
  author  = {Aldous, David J.},
  title   = {Weak convergence and the general theory of processes},
  journal = {Unpublished manuscript},
  year    = {1981}
}

@article{hoover1984adapted,
  author  = {Hoover, Douglas N. and Keisler, H. Jerome},
  title   = {Adapted probability distributions},
  journal = {Transactions of the American Mathematical Society},
  volume  = {286},
  number  = {1},
  pages   = {159--201},
  year    = {1984}
}

@book{bensoussan1992partial,
  author    = {Bensoussan, Alain},
  title     = {Stochastic Control of Partially Observable Systems},
  publisher = {Cambridge University Press},
  year      = {1992},
  series    = {Cambridge Studies in Advanced Mathematics},
  volume    = {48},
  address   = {Cambridge},
  isbn      = {9780521445335}
}

@article{bayraktar2013stochastic,
  title={Stochastic {P}erron's method for Hamilton--Jacobi--Bellman equations},
  author={Bayraktar, Erhan and Sirbu, Mihai},
  journal={SIAM Journal on Control and Optimization},
  volume={51},
  number={6},
  pages={4274--4294},
  year={2013},
  publisher={SIAM}
}

@article{veraguas2016causal,
  title={Causal transport in discrete time and applications},
  author={Backhoff, Julio and Beiglb{\"o}ck, Mathias and Lin, Yiqing and Zalashko, Anastasiia},
  journal={arXiv preprint arXiv:1606.04062},
  year={2016}
}

@article{AkJe17,
  title={Enlargement of filtration with finance in view},
  author={Aksamit, Anna and Jeanblanc, Monique },
  year={2017},
  publisher={Springer}
}

@article{Ru85,
  title={The {W}asserstein distance and approximation theorems},
  author={R{\"{u}}schendorf, Ludger},
  journal={Probability Theory and Related Fields},
  volume={70},
  number={1},
  pages={117--129},
  year={1985},
  publisher={Springer}
}

@article{crandall1992user,
  title={User's guide to viscosity solutions of second order partial differential equations},
  author={Crandall, Michael G and Ishii, Hitoshi and Lions, Pierre-Louis},
  journal={Bulletin of the American Mathematical Society},
  volume={27},
  number={1},
  pages={1--67},
  year={1992}
}

@article{kinderlehrer1978smoothness,
  title={The smoothness of the free boundary in the one phase {S}tefan problem},
  author={Kinderlehrer, David and Nirenberg, Louis},
  journal={Communications on Pure and Applied Mathematics},
  volume={31},
  number={3},
  pages={257--282},
  year={1978},
  publisher={Wiley Online Library}
}

@article{Kushner1964,
  author  = {Kushner, Harold J.},
  title   = {On the differential equations satisfied by conditional probability densities of {M}arkov processes, with applications},
  journal = {Journal of the Society for Industrial and Applied Mathematics, Series A: Control},
  year    = {1964},
  volume  = {2},
  number  = {1},
  pages   = {106--119}
}

@book{bertsekas1996stochastic,
  author    = {Bertsekas, Dimitri P. and Shreve, Steven E.},
  title     = {Stochastic Optimal Control: The Discrete-Time Case},
  publisher = {Athena Scientific},
  address   = {Belmont, MA},
  year      = {1996},
  note      = {Originally published by Academic Press, 1978}
}

@book{barles1994solutions,
  author    = {Barles, Guy},
  title     = {Solutions de viscosit{\'e} des {\'e}quations de Hamilton--Jacobi},
  series    = {Math{\'e}matiques et Applications},
  volume    = {17},
  publisher = {Springer-Verlag},
  address   = {Paris},
  year      = {1994}
}

@article{lassalle2018causal,
  author  = {Lassalle, R{\'e}mi},
  title   = {Causal transport plans and their {M}onge--{K}antorovich problems},
  journal = {Stochastic Analysis and Applications},
  volume  = {36},
  number  = {3},
  pages   = {452--484},
  year    = {2018},
  doi     = {10.1080/07362994.2017.1422747}
}

@article{pflug-pichler2012,
  author  = {Pflug, Georg Ch. and Pichler, Alois},
  title   = {A distance for multistage stochastic optimization models},
  journal = {SIAM Journal on Optimization},
  volume  = {22},
  number  = {1},
  pages   = {1--22},
  year    = {2012},
  doi     = {10.1137/110825054}
}

@article{backhoff2020fundamental,
  author  = {Backhoff-Veraguas, Julio and Beiglb{\"o}ck, Mathias and Eder, Manu and Pichler, Alois},
  title   = {Fundamental properties of process distances},
  journal = {Stochastic Processes and their Applications},
  volume  = {130},
  number  = {9},
  pages   = {5575--5591},
  year    = {2020},
  doi     = {10.1016/j.spa.2020.03.017}
}

@article{GoRoSaTe17,
  title={Kantorovich duality for general transport costs and applications},
  author={Gozlan, Nathael and Roberto, Cyril and Samson, Paul-Marie and Tetali, Prasad},
  journal={Journal of Functional Analysis},
  volume={273},
  number={11},
  pages={3327--3405},
  year={2017},
  publisher={Elsevier}
}

@article{JiOb24,
  title={Sensitivity of causal distributionally robust optimization},
  author={Jiang, Yifan and Obloj, Jan},
  journal={arXiv preprint arXiv:2408.17109},
  year={2024}
}

@article{BaWi22,
  title={Sensitivity of multiperiod optimization problems in adapted {W}asserstein distance},
  author={Bartl, Daniel and Wiesel, Johannes},
  journal={arXiv preprint arXiv:2208.05656},
  year={2022}
}

@article{BlWiZhZh26,
  title={Empirical martingale projections via the adapted {W}asserstein distance},
  author={Blanchet, Jose and Wiesel, Johannes and Zhang, Erica and Zhang, Zhenyuan},
  journal={The Annals of Applied Probability},
  volume={36},
  number={1},
  pages={547--606},
  year={2026},
  publisher={Institute of Mathematical Statistics}
}

@article{BlLaPaWi24,
  title={Bounding adapted {W}asserstein metrics},
  author={Blanchet, Jose and Larsson, Martin and Park, Jonghwa and Wiesel, Johannes},
  journal={arXiv preprint arXiv:2407.21492},
  year={2024}
}

@article{LaPaWi25,
  title={The fast rate of convergence of the smooth adapted {W}asserstein distance},
  author={Larsson, Martin and Park, Jonghwa and Wiesel, Johannes},
  journal={arXiv preprint arXiv:2503.10827},
  year={2025}
}

@article{HiRo24,
  title={Bicausal optimal transport for SDEs with irregular coefficients},
  author={Hitz, Michaela and Robinson, Benjamin A},
  journal={arXiv preprint arXiv:2403.09941},
  year={2024}
}

@article{BaKaRo25,
  title={Adapted {W}asserstein distance between the laws of SDEs},
  author={Backhoff-Veraguas, Julio and K{\"a}llblad, Sigrid and Robinson, Benjamin A},
  journal={Stochastic Processes and their Applications},
  volume={189},
  pages={104689},
  year={2025},
  publisher={Elsevier}
}

@article{EcPa24,
  title={Computational methods for adapted optimal transport},
  author={Eckstein, Stephan and Pammer, Gudmund},
  journal={The Annals of Applied Probability},
  volume={34},
  number={1A},
  pages={675--713},
  year={2024},
  publisher={Institute of Mathematical Statistics}
}

@article{BaBePa21,
  title={The {W}asserstein space of stochastic processes},
  author={Bartl, Daniel and Beiglb{\"o}ck, Mathias and Pammer, Gudmund},
  journal={arXiv preprint arXiv:2104.14245},
  pages={1--41},
  year={2021}
}

@article{backhoff2020all,
  author  = {Backhoff-Veraguas, Julio and Bartl, Daniel and Beiglb{\"o}ck, Mathias and Eder, Manu},
  title   = {All adapted topologies are equal},
  journal = {Probability Theory and Related Fields},
  volume  = {178},
  pages   = {1125--1172},
  year    = {2020},
  doi     = {10.1007/s00440-020-00993-8}
}

@article{backhoff2020adapted,
  author  = {Backhoff-Veraguas, Julio and Bartl, Daniel and Beiglb{\"o}ck, Mathias and Eder, Manu},
  title   = {Adapted {W}asserstein distances and stability in mathematical finance},
  journal = {Finance and Stochastics},
  volume  = {24},
  number  = {3},
  pages   = {601--632},
  year    = {2020},
  doi     = {10.1007/s00780-020-00426-3}
}

@article{backhoff2022estimating,
  author  = {Backhoff, Julio and Bartl, Daniel and Beiglb{\"o}ck, Mathias and Wiesel, Johannes},
  title   = {Estimating processes in adapted {W}asserstein distance},
  journal = {The Annals of Applied Probability},
  volume  = {32},
  number  = {1},
  pages   = {529--550},
  year    = {2022},
  doi     = {10.1214/21-AAP1687}
}

@article{wonham1965applications,
  author  = {Wonham, W. M.},
  title   = {Some applications of stochastic differential equations to optimal nonlinear filtering},
  journal = {SIAM Journal on Control},
  volume  = {2},
  number  = {3},
  pages   = {347--369},
  year    = {1965},
  doi     = {10.1137/0302028}
}

\end{document}